\documentclass[10pt]{article}
\usepackage[utf8]{inputenc} 
\usepackage[T1]{fontenc} 
\usepackage{soul}
\usepackage[a4paper]{geometry} 
\usepackage{amsmath, amssymb, mathtools, amsthm, bbm}
\usepackage{mathtools} 
\usepackage{mathrsfs} 
\usepackage{graphicx} 
\usepackage{vwcol} 
\usepackage{mathpazo}
\usepackage[OT1]{eulervm}
\usepackage{xcolor}
\usepackage[backend=biber,maxbibnames = 99, url = false, isbn = false]{biblatex}
\AtEveryBibitem{\clearlist{issn}}
\AtEveryBibitem{\clearfield{pages}}
\AtEveryBibitem{\clearfield{note}}

\usepackage[title]{appendix}
\usepackage{varwidth}
\usepackage{enumitem} 
\usepackage{multicol}
\usepackage[colorlinks=true]{hyperref}
\hypersetup{
    breaklinks=true,
    colorlinks,
    citecolor={[rgb]{1,0,0}},
    filecolor={[rgb]{1,0,0}},
    linkcolor={[rgb]{1,0,0}},
    urlcolor= {[rgb]{1,0,0}}
}

\newcommand{\R}{\mathbb{R}}
\newcommand{\N}{\mathbb{N}}

\newcommand{\sP}{\mathscr{P}}

\newcommand {\e}  {\varepsilon}

\newcommand {\Chi} {{\bf \raise 2pt \hbox{$\chi$}} }

\newcommand {\sgn} { {\rm sgn} }

\newcommand {\caT} { {\mathcal T} }
\newcommand {\caK} { {\mathcal K} }
\newcommand {\bs}{\boldsymbol}
\newcommand {\f}   {\frac}

\newcommand{\co}{\colon}
\newcommand{\diff}{\mathop{}\!\mathrm{d}}

\theoremstyle{plain}
\newtheorem*{thm*}{Theorem}
\newtheorem{thm}{Theorem}[section]

\newtheorem{lemma}[thm]{Lemma}

\newtheorem{proposition}[thm]{Proposition}
\newtheorem{corollary}[thm]{Corollary}
\newtheorem{hypothesis}[thm]{\bf Hypothesis}

\theoremstyle{remark}
\newtheorem{remark}[thm]{\bf Remark}
\newtheorem{definition}[thm]{\bf Definition}

\newcommand{\ie}{\emph{i.e.}\;}

\numberwithin{equation}{section}
\newcommand*\samethanks[1][\value{footnote}]{\footnotemark[#1]}

\title{Existence Theory for a Cross-Diffusion System with Independent Drifts: Mixing Dynamics}

\author{Alp\'ar R. M\'esz\'aros\thanks{Department of Mathematical Sciences, 
Durham University,
Upper Mountjoy Campus,
Durham,
UK,
DH1 3LE. 
Email: alpar.r.meszaros@durham.ac.uk and guy.m.parker@durham.ac.uk} \quad
and \quad 
Guy Parker\samethanks
}

\date{}

\usepackage{xurl}

\begin{document}

\maketitle

\begin{abstract}
\noindent 
We consider a cross-diffusion system for which the diffusion of each species is governed solely by the aggregate density through a pressure law of logarithmic or fast diffusion type.
The model is set over a one dimensional bounded interval, equipped with no-flux boundary conditions, and accommodates for the presence of potential drifts which are allowed to differ across each species. We establish the global existence of solutions without having to assume the total mixing of solutions. As a consequence, we give a full resolution of the PDE systems recently studied by the authors in \cite{meszarosParker} and by Elbar--Santambrogio in \cite{elbarfilippo2025}, by allowing a general class of initial data with finite bounded variation, with no further structural assumptions on their supports.
\end{abstract}
\section{Introduction}\label{sec:Intro}

In this paper we continue our programme, initiated in \cite{meszarosParker}.
Thus, the goal of this manuscript is to establish a weak existence theory for a general class of two-species cross-diffusion systems, set on a bounded open interval $\Omega \coloneqq (0,L) \subset \R$.
The class of systems under consideration are further equipped with: 
\begin{enumerate}[label=\textopenbullet]
    \item A fixed a time horizon $T> 0$.
    \item Lipschitz continuous potentials $V_{1},V_{2}\co \Omega \to \R $.
    \item A non-linearity $f \in C^2(0,+\infty)$, corresponding to either:
    \begin{enumerate}
        \item A pressure law of fast-diffusion type $\ie f$ behaving like a power law $s^\alpha$ with $\alpha\in (0,1)$;
        \item The logarithmic pressure law $f(s) \coloneqq s\log(s)$. 
    \end{enumerate} 
\end{enumerate}

Subsequently, the model system reads as follows.
\begin{equation} \label{eq:cdid}
\begin{cases}
\begin{alignedat}{2}
  \partial_t \rho_i &= \partial_x(\rho_i\partial_x(f'(\rho_1+\rho_2)+V_i)) &&\text{ in } (0,T)\times \Omega,\\
    \rho_{i,0} &= \rho_i &&\text{ on } \{ 0 \} \times \overline{\Omega}, \\
     0 &= \rho_i\partial_x(f'(\rho_1+\rho_2)+V_i) &&\text{ on } (0,T)\times \partial\Omega,
\end{alignedat}
\end{cases}
 \text{ for } i \in \{1,2\}.
\end{equation}
It has been proven in \cite[Theorem 1.]{BertschPeletier} that, in the absence of convective effects, solutions of System \eqref{eq:cdid} may produce jump discontinuities within a finite time. 
Consequently, the following definition of weak-solution is provided, facilitating a notion of solution to System \eqref{eq:cdid} for which each species density is required only to be integrable.

\begin{definition}\label{mainsolution}
Let $g \in C^1(0,+\infty)$ denote the map
\[
s \mapsto g(s) \coloneqq \int_{1}^s zf''(z) \diff z.
\]
Given initial data $\rho_{1,0},\rho_{2,0}\in \sP^{ac}(\Omega)$, the pair $(\rho_1,\rho_2)$ defines a weak solution to System \eqref{eq:cdid} if, for each $i \in \{1,2\}$: 
    \begin{enumerate}
        \item $\rho_i \in L^\infty([0,T];\sP(\Omega))$,
        \item $\partial_x g(\rho_1 + \rho_2) \in L^1([0,T]\times \Omega)$,
        \item The following equality is satisfied for any $\varphi \in C_c^1([0,T)\times \overline{\Omega})$.
    \end{enumerate}
    \begin{equation}\label{eq:weakcdid}
    \begin{split}
   \int_{\Omega} \rho_{i,0}\varphi_0 \diff x & +  \int_0^T\int_{\Omega}\rho_i\partial_t\varphi \diff x \diff t \\
   & = \int_0^T\int_{\Omega} \rho_i\partial_x(f'(\rho_1+\rho_2)+V_i)\partial_x \varphi\diff x \diff t.
\end{split}
\end{equation}
\end{definition}
Consequently, the aim of this manuscript is to prove the following result, whose precise statement can be found in full detail in Theorem \ref{maintheorem}.
\begin{thm}\label{concisemaintheorem}
Consider System \ref{eq:cdid} equipped with a logarithmic or fast-diffusive type pressure.
Further, suppose that the initial data and drifts of System \ref{eq:cdid} are suitably regular and satisfy the compatibility condition 
\[
\partial_xV_1 = \partial_xV_2 \ \mathrm{ on } \ \partial\Omega.
\]
Then, there exists a weak solution to System \eqref{eq:cdid} in the sense of Definition \ref{eq:weakcdid}.
\end{thm}
\subsection{Motivation and Background Literature}
Thanks to the works \cite{meszarosParker, elbarfilippo2025}, System \eqref{eq:cdid} has enjoyed a well-posedness theory in the presence of species  independent drifts:
\begin{itemize}
    \item Firstly, in the context of the logarithmic pressure law (cf. \cite{meszarosParker}).
    \item Secondly, for a power-law type fast-diffusion (cf. \cite{elbarfilippo2025}).
\end{itemize}
The methodology for each of these works was to study the dissipation of a first order energy taking the form 
\begin{equation}\label{eq:energyformal}
    t\mapsto \left\|\partial_x \log\left(\frac{\rho_1(t)}{\rho_2(t)}\right) - \partial_x(V_1-V_2)\xi(\sigma(\rho_1(t)+\rho_2(t)))\right\|_{L^1(\Omega)}
\end{equation}
for a suitably defined non-linearity $\xi\co[0,+\infty) \to \R$.

\medskip 

Specifically, a first breakthrough came in \cite{meszarosParker} when Energy \eqref{eq:energyformal} was first introduced. 
In this instance, $\xi$ was taken to be a constant.
Then, a second breakthrough came with the later realisation, in \cite{elbarfilippo2025}, that an existence theory could be provided for a larger class of pressure laws by introducing a non-linearity dependent of the aggregate density.

\medskip 

By studying the dissipation of this energy, each of these works established a novel Gr\"onwall argument with which to establish  $L^\infty([0,T];BV(\Omega))$ estimates for the quantity 
\[
r(t,x) \coloneqq \frac{\rho_1(t,x)}{\rho_1(t,x)+\rho_2(t,x)}.
\]
Indeed, given a suitable viscous approximation of System \eqref{eq:cdid} and equipped with such an estimate, it is then possible to establish sufficient compactness with which to derive a weak existence theory in the sense of Definition \ref{mainsolution}.

\medskip 

On the other hand, to control the dissipation of Energy \eqref{eq:energyformal}, it is necessary to assume that the initial densities $\rho_{1,0},\rho_{2,0}$ satisfy 
\begin{equation}\label{eq:enforced}
    \log\left(\frac{\rho_{1,0}}{\rho_{2,0}}\right) \in BV(\Omega) \subset L^\infty(\Omega).
\end{equation}

The logarithm approaches infinity at both $0$ and $+\infty$ and, consequently, the existence theory established in \cite{meszarosParker, elbarfilippo2025} addresses initial data for which the densities are \emph{totally mixed}. 
That is to say, the two initial profiles must share the same support as one another and, whilst each profile can vanish or blow up, it follows from \eqref{eq:enforced} that the two densities must still vanish or approach points of blow up at the same asymptotic rate as another.

\medskip 

Indeed, by controlling the dissipation of Energy \eqref{eq:energyformal}, it follows that the total mixing property is propagated across time. 
This means that the solutions constructed in \cite{meszarosParker, elbarfilippo2025} can neither fully segregate nor form complete inter-species interfaces with one another. 
To the best of the authors' knowledge, this phenomenon had never previously been established for any cross-diffusion systems associated with independent drifts and is in complete dichotomy with segregated solutions which are investigated in \cite{BertschPeletier,MeszarosKim2018,CFSS,TomPerthame}.

\medskip 

In contrast to the methodology established in \cite{meszarosParker, elbarfilippo2025}, the work \cite{CFSS} established a one-dimensional existence theory for System \eqref{eq:cdid} in the {\it absence} of drift terms but also without the necessity of the total mixing condition. 
In particular, in this work, $L^\infty([0,T];BV(\Omega))$ estimates were established for the quantity $r$ by a direct calculation that the energy 
\begin{equation}\label{eq:energy2}
    t\mapsto \|\partial_x r(t)\|_{L^1(\Omega)}
\end{equation}
was preserved over time. 

\medskip 

By only studying the dissipation of Energy \eqref{eq:energy2}, 
the total mixing condition is not imposed at the initial time and, as such, the existence theory established in \cite{CFSS} facilitates solutions which may mix and segregate over time.

\subsection{Contribution}

The main contribution of this work is therefore to bridge the gap between the existence theories established in \cite{meszarosParker,elbarfilippo2025} and the existence theory of \cite{CFSS}. Since for the analysis in \cite{meszarosParker,elbarfilippo2025} the total mixing of the solutions was indispensable, significant new ideas are needed to go beyond that setting, and we achieve this by introducing a novel \emph{inhomogeneous ratio function}.
This new ratio function facilitates a more general weak existence theory, allowing for:
\begin{enumerate}
    \item Independent drifts;
    \item Partially mixed, totally mixed, or segregated initial data;
    \item The existence of weak solutions, which may segregate or mix as the system evolves.
\end{enumerate}

In addition to such a contribution, the theory of \cite{meszarosParker,elbarfilippo2025} is further extended by:
\begin{itemize}
    \item Providing an existence theory for a large class of pressure laws which exhibit a fast-diffusion structure but do not correspond exactly to a power law;
    \item Exploring the existence problem on a bounded domain equipped with no-flux boundary conditions.
\end{itemize}

\subsection{Method}
  The primary obstruction to a weak existence theory for System \eqref{eq:cdid} is as follows:
  Given any smooth approximation of System \eqref{eq:cdid}, which is denoted here by 
  \[
  (\rho_1^\eta,\rho_2^\eta)_{\eta \in I}, \ I \coloneqq (0,1],
  \]
  one wishes to establish a weak convergence of the product 
  \begin{equation}\label{product}
        \rho_i^\eta\partial_x f'(\rho_1^\eta+\rho_2^\eta).
  \end{equation}

  Upon first inspection, the only a priori estimate available to the system is accessible by exploiting the parabolic structure of the aggregate density, denoted by 
  \[
  \sigma(t,x) \coloneqq \rho_1(t,x) + \rho_2(t,x).
  \]
  Indeed, for a solution of System \eqref{eq:cdid}, the variable $\sigma$ formally satisfies an advection-diffusion equation of porous medium type. 
  This given by
  \begin{equation}\label{eq:nlfp}
    \partial_t\sigma = \partial_{x}(\partial_x g(\sigma)+ \sigma \bs{v}) \quad \quad \bs{v} = \f{\rho_1\partial_xV_1+ \rho_2\partial_xV_2}{\sigma} \in L^\infty_{t,x}. 
    \end{equation}
    Consequently, the standard parabolic theory is able to provide $L^p_{t,x}$ a priori estimates for the family $(\partial_x f'(\sigma^\eta))_{\eta \in I}$ - for some $p \in [1,+\infty)$ - which can be further used to obtain $L^p_{t,x}$ estimates for each density $(\rho_i^\eta)_{\eta \in I}$.
    Such estimates provide a weak $L^p_{t,x}$ compactness to each factor in Product \eqref{product} for which it may be identified that 
    \[
    \rho_i^\eta \rightharpoonup \rho_i \text{ and } \partial_xf'(\rho_1^\eta +\rho_2^\eta) \rightharpoonup  \partial_xf'(\rho_1 +\rho_2). 
    \]
    However, since the product of two weakly convergent factors does not necessarily converge to the weak limit of the two respective factors, it is not clear that the limit of \eqref{product} can be identified with the limiting object
    \[
    \rho_i \partial_xf'(\rho_1+\rho_2).
    \]
    Thus, a strategy for overcoming this obstacle is to obtain strong compactness for one or both of the factors in \eqref{product} allowing for the identification of the weak limit. 
    Indeed, this was the approach taken in \cite{CFSS, meszarosParker, elbarfilippo2025}.

    \medskip 
    
    In each of these works, 
    $L^\infty([0,T];BV(\Omega))$ estimates were first established for the ratio $r^\eta$ by studying the dissipation of the energies \eqref{eq:energyformal} and \eqref{eq:energy2} respectively. 
    Such estimates could then be coupled with an $L^p([0,T];BV(\Omega))$ estimate for the aggregate density, coming from the non-linear parabolic theory, in order to assert that each of the products
    \[
    \rho_1^\eta = r^\eta \sigma^\eta , \quad \rho_2^\eta = (1-r^\eta)\sigma^\eta
    \]
    were bounded in $L^p([0,T];BV(\Omega))$ uniformly in $\eta$.
    Moreover, equipped with such regularity for $\rho_i^\eta$, the application of the Aubin--Lions lemma endows each family of densities with a strong compactness in $L^p_{t,x}$, thus allowing for the identification of \eqref{product} in the weak limit.

    \subsection*{The Inhomogeneous Ratio}

    In this manuscript, a weak existence theory is also provided by establishing an a priori $L^\infty([0,T];BV(\Omega))$ estimate for the ratio $r$. 
    However, to avoid imposing the total mixing condition, and to allow for segregated or partially mixed initial densities, a new energy is introduced.

    \medskip 
    
    In the a-convective regime, that is to say: when 
    \[
    \partial_x V_1 = \partial_xV_2 = 0,
    \]
    the new energy introduced in this manuscript coincides exactly with Energy \ref{eq:energy2} which was studied in \cite{CFSS}.
    However, more generally, in the convective regime, this new energy can be seen as an \emph{inhomogeneous ratio} function. 
    That is, a function $\phi \co [0,T]\times \Omega \to [0,1]$ which describes the ratio between the density $\rho_1$ and the aggregate $\sigma$ in such a way that the ratio is weighted, with respect to the potentials and the aggregate density, depending on the given point in time and space at which the ratio is evaluated. 

    \medskip 

    The construction of such an energy is informed by the fact that, when 
    $j\co \R \to \R$ denotes the function 
    \begin{equation}\label{eq:j}
    j(s) = \frac{\exp(s)}{1+\exp(s)}
    \end{equation}
    the following equality is satisfied.
    \[
    j \circ \log\left(\frac{\rho_1}{\rho_2}\right) = r.
    \]
    In particular, the quantity 
    \[
    \log\left(\frac{\rho_1}{\rho_2}\right)
    \]
    may be viewed as the primitive of the function inside Energy \eqref{eq:energyformal} in the setting when $\partial_xV_1 = \partial_xV_2$.

    \medskip 

    Thus, by considering a general choice of $V_1,V_2$, it is by taking the primitive of the function inside Energy \eqref{eq:energyformal} and composing this function with the choice of $j$, given in Equation \eqref{eq:j}, that an inhomogeneous ratio type function is constructed. 

    \medskip 
    
    Indeed, supposing the existence of a solution of System \eqref{eq:cdid} in the convective regime, such a ratio function formally takes the following form.
    \begin{equation}\label{eq:inhomogeneousratio}
    \begin{split}
    \phi(t,x) & \coloneqq  \f{r(t,x)\omega(t,x)}{1-r(t,x) +r(t,x)\omega(t,x)}, \\
    & \omega(t,x) \coloneqq \exp\left(-\int_0^x \partial_s(V_1(s)- V_2(s)
) \xi(\sigma(t,s)) \diff s\right).
\end{split}
\end{equation}

Given the formal construction of such a ratio function and a smooth approximation of System \eqref{eq:cdid}, the derivation of an $L^\infty([0,T];BV(\Omega))$ estimate for $r$ then follows (at least formally) by establishing that $\phi$ satisfies the following three properties.

\begin{enumerate}[label=A$\arabic*$., ref = A$\arabic*$]
\item 
The evolution equation for $\phi$ is of the form
    \[
    \partial_t\phi = a \partial_x \phi(t,x) + b
    \]
    for suitably defined coefficient functions $a,b\co[0,T]\times \Omega \to \R.$ \label{introA}
\item 
The pair $(b,\phi)$ 
admits the existence of $\mathcal{B} \in L^1([0,T])$ for which the inequality
\begin{align}\label{eq:B}
\|\partial_x b(t)\|_{L^1(\Omega)}\leqslant \mathcal{B}(t)\left(1+ \|\partial_x \phi\|_{L^1(\Omega)}\right)
\end{align}
is satisfied for almost every $t \in [0,T]$. \label{introB}
\item 
The following equivalence holds on $[0,T]$.
\[ 
\quad \quad \phi(t) \in BV(\Omega) \Longleftrightarrow r(t) \in BV(\Omega).
\]
\label{introC}
\end{enumerate}
Indeed, if $\phi$ can be shown to satisfy Properties \ref{introA}, \ref{introB} then the following formal argument can be made for the dissipation of the energy
\[
t \mapsto \|\phi(t)\|_{L^1(\Omega)}.
\]
\begin{align*}
    \partial_t &\int_{\Omega} |\partial_x \phi| \diff x 
    =  \int_{\Omega} \sgn(\partial_x  \phi)\partial_x\partial_t  \phi\diff x \\
    = & \int_{\Omega} \sgn(\partial_x  \phi)\left(\partial_x(a \partial_x \phi) +\partial_xb \right)\diff x
    \\ \leqslant & \int_{\Omega} \partial_x (a |\partial_x \phi|) + |\partial_x b| \diff x 
    \\ \leqslant & \int_{\partial\Omega} a |\partial_x \phi|\cdot \nu \diff x + \mathcal{B}\left(1+\int_{\Omega}|\partial_x\phi| \diff x\right).
\end{align*}
In particular, if $|\partial_x\phi|$ can be shown to vanish on $(0,T)\times \partial\Omega$, then, since $\mathcal{B}\in L^1([0,T])$, in accordance with Property \ref{introB}, $\phi$ satisfies the inequality 
\begin{equation}\label{eq:pregronwall}
      \partial_t \int_{\Omega} |\partial_x \phi| \diff x \leqslant \mathcal{B}\left(1+\int_{\Omega}|\partial_x\phi| \diff x\right).
\end{equation}
Moreover, by applying a Gr\"onwall type argument to Inequality \eqref{eq:pregronwall}, it follows that $\phi$ formally satisfies the inequality
\begin{equation}\label{eq:gronwall}
      \|\partial_x \phi(t)\|_{L^1(\Omega)} +1 \leqslant \left(\|\partial_x \phi_0\|_{L^1(\Omega)} +1\right)\exp\left(\int_0^t \mathcal{B}(t)\diff t\right)
\end{equation}
and hence, $\phi \in L^\infty([0,T];BV(\Omega))$ if $\phi_0 \in BV(\Omega)$. 
Further, if $\phi$ further satisfies Property \ref{introC}, then it follows that $r \in L^\infty([0,T];BV(\Omega))$ if $r_0 \in BV(\Omega)$.

\subsection*{The Viscous Approximation}

So far, the introductory exposition has concerned a formal argument with which to establish $L^\infty([0,T];BV(\Omega))$ estimates for $r$, here we discuss how a smooth approximation of System \eqref{eq:cdid} can be constructed such that those arguments can also be made rigorous.

\medskip

Let $\eta \in I \coloneqq (0,1]$. 
Then, the viscous regularisation of System \eqref{eq:cdid} is given by
\begin{equation}\label{eq:vcdidintro}
\begin{cases}
    \partial_t \rho_1 = \eta\partial_{xx}^2\rho_1 + \partial_x(\rho_1\partial_x(f'(\rho_1+\rho_2)+V_1)),\\
    \partial_t \rho_2 = \eta\partial_{xx}^2\rho_2 + \partial_x(\rho_2\partial_x(f'(\rho_1+\rho_2)+V_2)),
\end{cases}
\text{in } (0,T)\times \Omega,
\end{equation}
and equipped with the no-flux boundary conditions
\begin{equation}\label{eq:bdvcdid}
\begin{cases}
   \eta \partial_x\rho_1 + \rho_1\partial_x(f'(\rho_1+\rho_2)+V_1) = 0, \\
   \eta \partial_x\rho_2 + \rho_2\partial_x(f'(\rho_1+\rho_2)+V_2) = 0, 
\end{cases}
\text{ on } (0,T)\times \partial\Omega.
\end{equation}
\medskip

The well-posedness of a classical solution of System \eqref{eq:vcdidintro} with sufficient regularity follows from the classical quasi-linear parabolic theory established by \cite{Amannqlsexistence,acquistapaceqps,Ladyzhenskaya}.
We prove this result in Appendix \hyperref[appendix:A]{A}, for the reader's convenience.

\medskip 

The existence of the second order term in System \eqref{eq:vcdidintro} of course means it is no longer directly possible to produce $L^\infty([0,T];BV(\Omega))$ by arguing via Properties \ref{introA}, \ref{introB} and \ref{introC}.
Indeed, if one is given $(\rho_1,\rho_2)$: a solution of the viscous System \eqref{eq:vcdidintro}, then, after making the coordinate transformation 
\[
(\rho_1,\rho_2) \mapsto (\sigma,\phi),
\] 
the evolution equation for $\phi$ is of second order and hence can not satisfy the transport equation specified in property \ref{introA}.

\medskip 

Instead, this manuscript introduces a family of ratio functions $\phi^\eta$, such, that for each $\eta \in I$, the following properties are satisfied.
\begin{enumerate}[label=B$\arabic*$., ref = B$\arabic*$]
\item
The evolution equation for $\phi^\eta$ is of the form
    \begin{equation}
    \begin{cases}
         \partial_t\phi^\eta = \eta\partial_{xx}\phi^\eta + a^\eta \partial_x \phi^\eta + b^\eta \text{ in } (0,T)\times \Omega, \\
         \partial_x \phi^\eta = 0 \text{ on } (0,T)\times \partial\Omega. 
     \end{cases}
    \end{equation}
    for suitably defined coefficient functions $a^\eta,b^\eta\co[0,T]\times \Omega \to \R.$ \label{B1}
\item 
The pair $(\phi^\eta, b^\eta)$ 
admits the existence of $\mathcal{B} \in L^1([0,T])$ for which the inequality
\begin{align}\label{eq:vB}
\|\partial_x b^\eta(t)\|_{L^1(\Omega)}\leqslant \mathcal{B}(t)\left(1+ \|\partial_x \phi^\eta\|_{L^1(\Omega)}\right)
\end{align}
is satisfied for almost every $t \in [0,T]$, uniformly in $\eta$. \label{B2}
\item 
There exist $\Lambda_1,\Lambda_2 > 0$ such that the following inequality holds, uniformly in $[0,T]$, uniformly in $\eta$.
 \begin{align*}
     \|\partial_x \phi^\eta(t)\|_{L^1(\Omega)} 
     &\leqslant \Lambda_1(\|\partial_x r(t)\|_{L^1(\Omega)}+1) \\
     &\leqslant \Lambda_2(\|\partial_x \phi(t) \|_{L^1(\Omega)}+1 ).
    \end{align*}
    \label{B3}
\end{enumerate}
Consider a family of inhomogeneous ratio functions $(\phi^\eta)_{\eta \in I}$ and $(r,\sigma)$ derived from a solution of System \eqref{eq:vcdid}, such that $\phi^\eta$ satisfies the Properties \ref{B1}-\ref{B2}. 
Then, the energy dissipation inequality for the time dissipation of the map $t \mapsto \|\partial_x\phi^\eta(t)\|_{L^1(\Omega)}$ may be controlled by a Gr\"onwall argument, in much similarity to the argument presented for Properties \ref{introA}-\ref{introC}.
The argument requires one notable modification in the viscous setting. 
\begin{itemize}
    \item The third order term which presents itself in the evolution equation for $|\partial_x \phi^\eta|$ is shown to be non-positive via a Kato inequality. 
\end{itemize}

To conclude the estimate, Property \ref{B3} ensures that a uniform estimate for $\phi$ in  $L^\infty([0,T];BV(\Omega))$ is sufficient to obtain a uniform $L^\infty([0,T];BV(\Omega))$ estimate for $r$ and that, to close the Gr\"onwall inequality, the assumption $\phi_0 \in BV(\Omega)$ is equivalent to the assumption that $r_0 \in BV(\Omega)$.

\subsection*{The Approximation Parameters}
System \ref{eq:cdid} and its viscous approximation, System \ref{eq:vcdid}, are of second order.
Meanwhile, the energy
\[
t \mapsto 
\|\partial_x\phi^\eta(t)\|_{L^1(\Omega)}
\]
is of first order. 
Consequently, the dissipation of this energy is described by an evolution equation of third order.
Thus, in order to precisely derive the relevant energy dissipation inequality, with which to obtain an $L^\infty([0,T];BV(\Omega))$ estimate, it is established that, for each $\eta \in I$, the variable 
$\phi^\eta$ is regular in the class 
\[
C([0,T];C^3(\overline{\Omega})) \cap C^1([0,T];C^1(\overline{\Omega})).
\]
To establish such regularity for the viscous approximation, the addendum of a viscosity term is insufficient.
Indeed, in order to derive the energy dissipation equality, such regularity must hold globally in the domain: up to the initial time and the boundary $\partial\Omega$. 
This means that the initial data and potentials must be taken to be suitably smooth and satisfy particular compatibility conditions at the boundary.

\medskip 

In particular, given a set of initial data and potential functions for the weak formulation, 
a smooth approximation of these objects is produced such that the $BV$ a priori estimates of the viscous system may be established uniformly in $\eta$. 
Then, after having obtained sufficient strong compactness via an Aubin--Lions type Lemma, the approximating potentials and initial data are passed to their limiting objects in concert with the disappearance of the viscosity term.
\subsection{Main Theorem}
To guarantee the existence of weak solutions to System \eqref{eq:weakcdid}, in the sense of Definition \ref{mainsolution}, it may first be assumed that $f$ corresponds exactly to a fast-diffusive power law or a logarithmic pressure and so satisfies the following assumption.
\begin{hypothesis}\label{hypothesis1}
Either:
     \begin{enumerate}[label=$(\mathbf{H\arabic*})$, ref = $\mathbf{H\arabic*}$]
    \item There exists $\alpha \in (0,1)$ and $\lambda > 0$ such that 
     \[
     f(s) \coloneqq \frac{\lambda}{\alpha-1}s^{\alpha};
     \]
    \item There exists $\lambda > 0$ such that 
    \[
    f(s) \coloneqq \lambda s\log(s).
    \]
\end{enumerate}
\end{hypothesis}
On the other hand, the weak existence theory can also be proven to hold for a much larger class of pressure laws, behaving either like either a fast-diffusive or a uniformly parabolic pressure. 
In this paper, the exposition focuses on providing a weak existence theory for a general class of pressure laws, which correspond to a fast diffusion.
That is, satisfying the following hypothesis which may be made alternately to Hypothesis \ref{hypothesis1}.
\begin{hypothesis}\label{hypothesis2}~
 \begin{enumerate}[label=$(\mathbf{I\arabic*})$, ref = $\mathbf{I\arabic*}$]
    \item \label{I1}
    There exists $\beta \in (0,1]$ such that
    $f \in C^{4,\beta}_{loc}(0,+\infty)$, $f'' > 0$ 
    \item \label{I2}
    There exists $\alpha \in (0,1)$ and $ \kappa > 0$ such that
    \[
   \frac{1}{\alpha\kappa}  s^\alpha \leqslant s^2 f''(s) \text{ for } s \in (0,+\infty); 
   \]  
   \item\label{I3}
   For the same choice of $\kappa$, the inequality 
    \[
    \left|\frac{sf'''(s)}{f''(s)}\right| \leqslant {\kappa}  \text{ for } s \in (0,+\infty)
    \]
    is satisfied.
    \item\label{I4}
    The map $\mathscr{H}\co \mathscr{M}_{+}^{ac}(\Omega) \to \R \cup \{+\infty\}$ given by
    \[
     \mathscr{H}(\mu) \coloneqq \int_{\Omega} f(\mu) \diff x
    \]
    admits a constant $\mathcal{K}_{\inf} > 0$ such that the following inequality holds for all $\mu \in \sP^{ac}(\Omega)$.
    \[
    \mathscr{H}(\mu) > -\mathcal{K}_{\inf}.
    \]
     \end{enumerate}
\end{hypothesis}
Then, given a choice of pressure potential $f$ satisfying either of the Hypotheses \ref{hypothesis1} or \ref{hypothesis2}, the following assumptions are further made on the potentials and the initial data.
\begin{hypothesis}\label{hypothesis3}
Given $f$ satisfying either Hypothesis \ref{hypothesis1} or Hypothesis \ref{hypothesis2}, it is assumed that:
\begin{enumerate}[label=$(\mathbf{J\arabic*})$, ref = $\mathbf{J \arabic*}$]
\begin{multicols}{2}
    \item \label{J1}
    $ \displaystyle \rho_{1,0},\rho_{2,0} \in \sP^{ac}(\Omega)$;
    \item \label{J2}
    $\displaystyle \mathrm{TV}\left(\frac{\rho_{1,0}}{\rho_{1,0}+\rho_{2,0}}\right)<+\infty$;
    \item \label{J3}
    $\displaystyle V_1-V_2\in W^{3,1}(\Omega);$
    \item \label{J4}
    $\displaystyle V_1,V_2 \in W^{2,1}(\Omega);$
    \item \label{J5}
    $\displaystyle \partial_x V_1 = \partial_x V_2 = 0 \text{ on } \partial\Omega$.
    \item \label{J6}
    $\displaystyle \mathscr{H}(\rho_{1,0}+\rho_{2,0}) < +\infty$.
    \end{multicols}
\end{enumerate}
If $f$ satisfies Hypothesis \ref{hypothesis2} but not \ref{hypothesis1} then further assume that:
\begin{enumerate}[label=$(\mathbf{J\arabic*})$,resume*, ref = $\mathbf{J \arabic*}$]
    \item \label{J7}
    If $\alpha \in [\frac{2}{3},1)$ then there exists $c_{\inf} > 0$ such that 
    \[
    (\rho_{1,0}+\rho_{2,0}) > c_{\inf} > 0 \text{ on } \overline{\Omega}.
    \]
\end{enumerate}
\end{hypothesis}
In detail, the Main Theorem of this manuscript then reads:
\begin{thm}\label{maintheorem}
    Assume that $f$ satisfies either Hypothesis \ref{hypothesis1} or Hypothesis \ref{hypothesis2} and that the potentials and initial data satisfy Hypothesis \ref{hypothesis3}.
    Then, there exists a weak solution of System \eqref{eq:cdid}.
\end{thm}

\begin{remark}\label{ratioremark}
Concerning Assumption \ref{J2}, we remark that, for a generic choice of $\rho_{1,0},\rho_{2,0} \in \sP^{ac}(\Omega)$, the ratio 
\[
\frac{\rho_{1,0}}{\rho_{1,0}+\rho_{2,0}}
\]
may not exist, even in a distributional sense.
However, in Proposition \ref{ratiopropn} it is shown that, given any $\rho_{1,0},\rho_{2,0} \in \sP^{ac}(\Omega)$, there exists $r_0 \co \Omega \to [0,1]$ which satisfies the equality
\begin{equation}\label{eq:rnotunique}
  r_0(\rho_{1,0}+\rho_{2,0}) = \rho_{1,0}  
\end{equation}
almost everywhere in $\Omega$.

\medskip 
It may not be the case that a function $r_0$ satisfying Equation \eqref{eq:rnotunique} is uniquely defined, even on the support of $\rho_{1,0}$, and so instead Assumption \ref{J2} should be understood to mean that:

\begin{itemize}
    \item Of the functions $r_0\co\Omega \to [0,1]$ satisfying Equation \eqref{eq:rnotunique}, there exists a representative such that $TV(r_0) < +\infty$.
\end{itemize}
\end{remark}

\begin{remark}\label{remarkJ7}
    It is expected that Assumption \ref{J7} is only technical for $f$ satisfying Hypothesis \ref{hypothesis2}.
    Indeed, if $f$ satisfies Hypothesis \ref{hypothesis2} which correspond exactly to a power law (as in Hypothesis \ref{hypothesis1}) then it will be shown that Assumption \ref{J7} is unnecessary for the regularity theory (cf. Remark \ref{remark:bounds}.
\end{remark}

\subsection*{Structure}
The manuscript is hence structured as follows:
\begin{itemize}
    \item In Section \hyperref[sec:approx]{Two}, the well-posedness of a smooth viscous approximation is established.
    \item In Section \hyperref[sec:ratio]{Three}, the non-linearity $\xi$ and the inhomogeneous ratio $\phi$ are introduced. 
    The evolution equation for $\phi$ is derived and Properties \ref{B1} and \ref{B3} are established.
    \item In Section \hyperref[sec:aggregate]{Four}, a priori estimates are derived for the aggregate density. This facilitates the proof of Property \ref{B2}.
    \item In Section \hyperref[sec:BV]{Five}, we conduct the main energy dissipation argument, obtaining $L^\infty([0,T];BV(\Omega))$ estimates for $r$.
    \item In Section \hyperref[sec:convergence]{Six}, sufficient compactness is established and the Main Theorem \ref{maintheorem} is proven by studying the convergence of the viscous approximation.
\end{itemize}

\subsection*{Notation and Preliminaries}
Let $\sP(\Omega)$ denote the set of probability measures on $\Omega$ and let $\sP^{ac}(\Omega)$ denote the set of $\mu \in \sP(\Omega)$ which admit a density with respect to the Lebesgue measure on $\Omega$.

\medskip 

For $k \in \N, p \in [1,+\infty]$, the set $W^{k,p}(\Omega)$ denotes the usual Sobolev space admitting $k$ weak derivatives of $p$-integrability with the usual identification $H^k(\Omega) = W^{k,2}(\Omega)$. 
Meanwhile, for $p \in [1,\infty)$, the set $W^{k,p}_0(\Omega)$ denotes the closure of $C_c^\infty(\Omega)$ within $W^{k,p}(\Omega)$. 
For $p \in (1,+\infty)$ the dual of $W^{k,p}_0(\Omega)$ is denoted $W^{-k,q}(\Omega)$ where $q$ is the H\"older conjugate of $p$.

\medskip 

Given $f \in L^1(\Omega)$ the total variation of $f$ denoted $\mathrm{TV}(f)$ is defined by 
\[
\mathrm{TV}(f) \coloneqq \sup_{\varphi \in X^{TV}} \int_{\Omega} f \partial_x \varphi \diff x , \quad X^{TV}\coloneqq \{\varphi \in C^1_c(\Omega) \ | \ \|\varphi\|_{L^\infty(\Omega)} \leqslant 1\}.
\]
Then, the space of functions of bounded variation on $\Omega$, denoted $BV(\Omega)$ is given by
\[
BV(\Omega) \coloneqq \{f \in L^1(\Omega) \ | \ \mathrm{TV}(f) < +\infty\}.
\]
Equipped with the norm 
\[
\|f\|_{BV(\Omega)} \coloneqq \|f\|_{L^1(\Omega)} + \mathrm{TV}(f),
\]
the space $BV(\Omega)$ defines a Banach space.

\medskip 
Moreover, for any $f \in W^{1,1}(\Omega)$, the equality 
\[
\mathrm{TV}(f) = \|\partial_x f\|_{L^1(\Omega)}
\]
and so, throughout this manuscript and when referring to functions of $W^{1,1}(\Omega)$ regularity, the two quantities will be used interchangeably depending on the context.

\medskip 

Lastly, in this manuscript a constant $c_i$ denotes a constant which will be used locally with in a proof.
A constant $\mathcal{K}_i$ denotes a global constant whose dependence will be kept track of.
\newpage
\section{The Viscous Approximation}\label{sec:approx}
Recall that the viscous approximation of System \eqref{eq:cdid} is given as follows.
\begin{equation}\label{eq:vcdid}
\begin{cases}
\begin{alignedat}{2}
  \partial_t \rho_i &= \eta\partial_{xx}\rho_i +\partial_x(\rho_i\partial_x(f'(\rho_1+\rho_2)+V_i)) &&\text{ in } (0,T)\times \Omega,\\
    \rho_{i,0} &= \rho_i &&\text{ on } \{ 0 \} \times \overline{\Omega}, \\
     0 &= \eta\partial_x\rho_i + \rho_i\partial_x(f'(\rho_1+\rho_2)+V_i) &&\text{ on } (0,T)\times \partial\Omega.
\end{alignedat}
\end{cases}
\end{equation}
To guarantee the existence of a unique classical solution to System \eqref{eq:vcdid}, possessing suitable regularity, Assumptions \ref{K1}-\ref{K4} are made. 

 \begin{enumerate}[label=$(\mathbf{K \arabic*})$, ref = $\mathbf{K\arabic*}$]
    \item \label{K1}
    $V_1,V_2\in C^{3,1}(\overline{\Omega})$ and $\rho_{1,0},\rho_{2,0}  \in C^{3,1}(\overline{\Omega})\cap \sP(\Omega)$;
    \item \label{K2}
    $\rho_{1,0} + \rho_{2,0} > 0 \text{ on } \overline{\Omega}$;
    \item \label{K3}
    The first and second compatibility conditions: Equations \eqref{eq:zeroorder} and \eqref{eq:firstorder} are satisfied, for $i \in \{1,2\}$, on  $\partial\Omega$.
    \item\label{K4}
    There exists $\beta \in (0,1]$ such that $f \in C^{4,\beta}_{loc}(0,+\infty)$ and $f'' \geqslant 0$ on $(0, +\infty)$.
\end{enumerate}
For each $i \in \{1,2\}$, the first  compatibility condition reads:
\begin{equation}\label{eq:zeroorder}
    \eta\partial_x\rho_{i,0} + \rho_{i,0} \partial_x (f'(\rho_{1,0}+\rho_{2,0})+V_i) = 0.
\end{equation}
To concisely define the second compatibility condition, consider the operators 
$\mathfrak{D}_i\co C^2(\Omega;G) \to C(\Omega)$ given, for $ i \in \{1,2\}$ by 
\begin{equation}\label{eq:boundaryop}
\begin{split}
    \mathfrak{D}_1 (\varphi,\psi) &\coloneqq \eta\partial_{xx}\varphi +\partial_x(\varphi\partial_x(f'(\varphi+\psi)+V_i)),\\
    \mathfrak{D}_2 (\varphi,\psi) &\coloneqq \eta\partial_{xx}\psi +\partial_x(\psi\partial_x(f'(\varphi+\psi)+V_i)).
\end{split}
\end{equation}
Then, for each $i \in \{1,2\}$, the second compatibility condition reads:
\begin{equation}\label{eq:firstorder}
\begin{split}
    \eta\partial_x&\mathfrak{D}_i(\rho_{1,0},\rho_{2,0}) + \mathfrak{D}_i(\rho_{1,0},\rho_{2,0}) \partial_x (f'(\rho_{1,0}+\rho_{2,0})+V_i) \\
    & + \rho_{i,0}\partial_x (f''(\rho_{1,0}+\rho_{2,0})\mathfrak{D}_1(\rho_{1,0},\rho_{2,0})+ \mathfrak{D}_2(\rho_{1,0},\rho_{2,0})) = 0.
\end{split}
\end{equation}
\begin{thm}\label{classical}
    Assume that \ref{K1}-\ref{K4} are satisfied. 
    Then, there exists a unique classical solution of System \eqref{eq:vcdid} with initial data $(\rho_{1,0},\rho_{2,0})$. 
    In particular, this solution possesses the regularity 
    \[
    \rho_1,\rho_2 \in C([0,T];C^3(\overline{\Omega})) \cap C^1([0,T];C^1(\overline{\Omega}))
    \]
    and satisfies System \eqref{eq:vcdid} pointwise on $[0,T)\times \overline{\Omega}$.
    \end{thm}
\begin{proof}
    See Appendix \hyperref[appendix:A]{A}.
\end{proof}
Throughout the remainder of this manuscript, it is assumed that any classical solution of System \eqref{eq:vcdid} is constructed by means of  Appendix \hyperref[appendix:A]{A}.
Moreover, it is tacitly assumed that any such classical solution satisfies Assumptions \ref{K1}-\ref{K4} and \ref{J5}.

\section{An Inhomogeneous Ratio Variable}\label{sec:ratio}
In light of the definition of the viscous approximation, System \eqref{eq:vcdid}, a family of inhomogeneous ratio functions $(\phi^\eta)_{\eta \in I}$ is constructed.
For such a construction, it is also necessary to approximate the non-linearity $\xi$ which was aforementioned in Energy \ref{eq:energyformal}.
Thus, the section begins by defining such a family $(\xi_\eta)_{\eta \in I}$ and then studies the relationship between the functions $\xi$ and $\xi_\eta$.
\begin{definition}\label{xidefn}
    Define the non-linearity $\xi \co[0,+\infty) \to \R$ given by 
    \[
    \xi(s) \coloneqq -s\int_s^\infty \f{1}{y^3 f''(y)} \diff y. 
    \]
    Then, for $\eta \in I$, define the non-linearity $\xi_\eta \co [0,+\infty) \to \R$ such that 
        \begin{equation}\label{eq:xi}
         \xi_\eta(s) \coloneqq  -s \exp\left(\int_0^{s}\frac{2\eta}{y^2f''(y)}\diff y \right)\cdot \int_s^\infty  \exp\left(-\int_0^{y}\frac{2\eta}{z^2f''(z)}\diff z\right)\frac{1}{y^3f''(y)} \diff y.
        \end{equation}
    In particular, further choosing $\eta = 0$ in Equation \eqref{eq:xi}, the function $\xi_0$ satisfies $\xi_0 = \xi$.
\end{definition}
As a consequence of Definition \ref{xidefn}, the non-linearity $\xi_\eta$ solves the following  ordinary differential equation    for each $\eta \in I \cup \{0\}$.
    \begin{equation}\label{eq:ode}
            \xi_\eta'(s)s^2f''(s) - \xi_\eta(s)(2\eta + sf''(s)) = 1
        \end{equation}
\begin{remark}\label{remark:logxi}
    Note that, for the choice of logarithmic pressure,
    that is 
    \[
    f(s) \coloneqq \lambda s\log (s),
    \]
    the integral term given in the definition of $\xi$ is not finite.
    In this instance, the variable $\xi$ is instead given by 
    \[
    \xi(s) = \xi \coloneqq - \lambda.
    \]
\end{remark}
\subsection{Bounding the Non-Linearity}
For $f$ satisfying Hypothesis \ref{hypothesis1}, the non-linearity $\xi_\eta$ may be calculated explicitly and thus compared to $\xi$. 
On the other hand, if $f$ satisfies only Hypothesis \ref{hypothesis2}, it is not clear how $\xi_\eta$ and its respective derivatives relate to $\xi$. 
The following Proposition shows how $\xi_\eta, \xi_\eta'$ and $\xi_\eta''$ may be controlled by $\xi$ and its respective derivatives.

\begin{proposition}\label{xiepstoxi}
            Assume that $f$ satisfies Hypothesis \ref{hypothesis2} and let $\xi, \xi_\eta$ be as given in Definition \ref{xidefn}.
            Then, for each $\eta \in I$, the following inequalities are satisfied on $[0,+\infty)$.
       \begin{align}
                & |\xi_\eta(s)| \leqslant |\xi(s)|, \label{ord:zero} \\
                & |\xi'_\eta(s)|\leqslant  \frac{|\xi(s)|}{s} + \frac{1}{s^2f''(s)} + \frac{2\eta|\xi(s)|}{s^2f''(s)},\label{ord:one} \\
                & |\xi_\eta''(s) f''(s) s^2| \leqslant 2\eta |\xi'_\eta(s)| +  (1+ 2\eta|\xi(s)|)|(\log(f''(s)s))'|.\label{ord:two}
       \end{align}
       \end{proposition}
        \begin{proof}
        To prove the zero order inequality, observe the following upper bound.
        \begin{align*}
         \int_s^\infty  & \exp\left(-\int_0^{y}\frac{2\eta}{z^2f''(z)}\diff z\right)\frac{1}{y^3f''(y)} \diff y  \\
         \leqslant & \sup_{s\leqslant y}\left(\exp\left(-\int_0^{y}\frac{2\eta}{z^2f''(z)}\diff z\right)\right) \int_s^\infty  \frac{1}{y^3f''(y)} \diff y\\
          \leqslant &  \exp\left(-\int_0^{s}\frac{2\eta}{y^2f''(y)}\diff y\right)\int_s^\infty  \frac{1}{y^3f''(y)} \diff y.
        \end{align*}
        In the above, the first inequality holds as a consequence of H\"older's $L^1-L^\infty$ inequality.
        Meanwhile, the second Inequality hold's due to the inequality $f'' \geqslant 0$ which means that the subsequent supremum is achieved at $y=s$.

        \medskip 
        
        Multiplying both sides of the above inequality by the factor
        \[
        s\exp\left(\int_0^{s}\frac{2\eta}{y^2f''(y)}\diff y\right)
        \]
        it follows that 
        \[
           |\xi_\eta(s)| \leqslant |\xi(s)|.
        \]
        To derive the first order inequality, recall that $\xi_\eta$ satisfies \eqref{eq:ode}. Consequently, the equality 
        \[
            |\xi'_\eta(s)| \leqslant \frac{1}{s^2f''(s)} + \left(\frac{1}{s}+ \frac{2\eta}{s^2f''(s)}\right)|\xi_\eta(s)|
        \]
        is satisfied. 
        To conclude the verity of Inequality \ref{ord:one}, the right hand-side of the above inequality may further be bounded using Inequality \ref{ord:zero}.

        \medskip 

        The second order inequality is realised by differentiating both sides of Equation \eqref{eq:ode} in the $s$ variable.
        This yields the equality
        \begin{equation}\label{eq:2ndode}
        \begin{split}
            \xi''_\eta(s)f''(s)s^2 & = 2\eta \xi_\eta'(s) + (\xi_\eta (s) -s\xi_\eta'(s))(sf''(s))'\\
            & = 2\eta \xi_\eta'(s)  -(1+2\eta \xi_\eta(s))\frac{(sf''(s))'}{sf''(s)}\\
            & = 2\eta \xi_\eta'(s)  -(1+2\eta \xi_\eta(s))(\log(f''(s)s))'.    
        \end{split}
        \end{equation}
        Take absolute values on both sides of Inequality \eqref{eq:2ndode}.
        Then, by using Inequality \ref{ord:zero} to bound the zero order terms, it follows that 
       \[
              \left|\xi_\eta''(s) f''(s) s^2\right| \leqslant 2\eta |\xi'_\eta(s)| + (1+ 2\eta|\xi(s)|)|(\log(f''(s)s))'|.
        \]
        \end{proof}
        Utilising Assumptions \ref{I2} and \ref{I3}, the bounds established in Proposition \ref{xiepstoxi} show that $\xi_\eta$ can be dominated by a linear combination of power-laws parametrised in the variable $\eta$. 

         \begin{corollary}
         Assume that $f$ satisfies Hypothesis \ref{hypothesis2}. 
         Then, $\xi_\eta$ satisfies the following inequalities on $[0,+\infty)$ for each $\eta \in I$.
         \begin{align}
            & |\xi_\eta(s)| \leqslant \kappa s^{1-\alpha}, \label{power:zero} \\
            & |\xi'_\eta(s)| \leqslant (1+\alpha)\kappa s^{-\alpha} + 2\eta\alpha \kappa^2 s^{1-2\alpha}, \label{power:one}\\
            & |\xi_\eta''(s) f''(s) s^2| \leqslant (\kappa+1)s^{-1} + 2\eta \kappa(2+ \alpha + \kappa)s^{-\alpha} + 4\alpha\eta^2 \kappa^2s^{1-2\alpha}, \label{power:two}
         \end{align}
        \end{corollary}
        \begin{proof}
        Firstly, recall from Definition \ref{xidefn} that $\xi$ satisfies
        \[
        \xi(s) \coloneqq -s\int_s^\infty \f{1}{y^3 f''(y)} \diff y. 
        \]
        Then, thanks to Assumption \ref{I2}, it follows that
        \[
        |\xi(s)| \leqslant \kappa\alpha s\int_s^\infty \f{1}{y^{1+\alpha}} = \kappa s^{1-\alpha}.
        \]
        Now, recalling Inequality \eqref{ord:one} from Proposition \ref{xiepstoxi}, $\xi_\eta$ satisfies
        \[
        |\xi_\eta (s)| \leqslant |\xi(s)| \leqslant \kappa s^{1-\alpha}.
        \]
        The derivation of Inequality \eqref{power:one} then follows from Inequality \eqref{ord:one} which may be bounded from above using Inequality \eqref{power:zero} and Assumption \ref{I2}. 
        \begin{align*}
        |\xi'_\eta(s)| & \leqslant  \frac{|\xi(s)|}{s} + \frac{1}{s^2f''(s)} + \frac{2\eta|\xi(s)|}{s^2f''(s)} \\
        & \leqslant \kappa s^{-\alpha} + \alpha \kappa s^{-\alpha} + 2\eta \alpha \kappa^2 s^{1-2\alpha}.
        \end{align*}
        Lastly, we treat Inequality \eqref{power:two}.
        Recall Inequality \eqref{ord:two}.
        Inequality \eqref{power:one} is used as an upper bound for the quantity $|\xi_\eta'|$ on the right hand-side of Inequality \eqref{ord:two}, meanwhile, Assumption \ref{I3} is used to find an upper bound for the logarithmic derivative.
        In particular, it follows that
        \begin{align*}
           |s^2\xi''_\eta(s)f''(s)| & \leqslant 2\eta |\xi_\eta'(s)|  + (1+2\eta |\xi_\eta(s)|)\left(\left|\log(f''(s))'\right| + \frac{1}{s}\right)
           \\
           & \leqslant 2\eta |\xi_\eta'(s)| + \left(\frac{1}{s}+2\eta \kappa s^{-\alpha}\right)(\kappa+1) 
           \\
           & \leqslant 
            (\kappa+1)s^{-1} + 2\eta \kappa(2+ \alpha + \kappa)s^{-\alpha} + 4\alpha\eta^2 \kappa^2 s^{1-2\alpha}. 
        \end{align*}
        \end{proof}
\subsection{The Transformation of Variables}
Given the family of non-linearities $(\xi_\eta)_{\eta \in I}$, the ratio $r$ and, subsequently, the inhomogeneous ratios $(\phi^\eta)_{\eta \in I}$ are defined as follows.
\begin{definition}\label{defnrsigma} 
Let $(\rho_1,\rho_2)$ denote a solution of System \eqref{eq:vcdid}.
The variables $\sigma,r \co [0,T)\times\Omega \to \R$ are defined
\begin{equation}\label{eq:defnrsgima}
\sigma(t,x)\coloneqq \rho_1(t,x) + \rho_2(t,x), \quad r(t,x) \coloneqq \f{\rho_1(t,x)}{\rho_1(t,x) + \rho_2(t,x)}.  
\end{equation}
\end{definition}

Thanks to Theorem \ref{shorttime}, a solution of System \eqref{eq:vcdid} satisfies $\sigma > 0 $ on $[0,T) \times \overline{\Omega}$.
Consequently, the coordinate transformation defined in Definition \ref{defnrsigma} is smooth for $(t,x) \in [0,T) \times \overline{\Omega}$. 
This means that 
\[
(\sigma,r) \in C([0,T);C^2(\overline{\Omega}))\cap C^1([0,T);C(\overline{\Omega})).
\]

\begin{definition}\label{vdefnphi}
Given $\eta \in I$ and $(\rho_1,\rho_2)$: a classical solution of System \eqref{eq:vcdid} for which $(\sigma,r)$ denote the aggregate density and the ratio respectively (cf. Definition \ref{defnrsigma}), let the map $\omega\co[0,T]\times \Omega \to \R$ be defined
    \begin{equation}\label{eq:vdefnPhi}
    \omega(t,x) \coloneqq \exp\left(-\int_0^x \partial_s(V_1(s)- V_2(s)
) \xi_\eta(\sigma(t,s)) \diff s\right).
\end{equation}
Then, let the map $\phi^\eta \co[0,T]\times \Omega \to \R$ be defined 
\begin{equation}
    \phi^\eta(t,x)\coloneqq  \f{r(t,x)\omega(t,x)}{1-r(t,x) +r(t,x)\omega(t,x)}.
\end{equation}
\end{definition}

\begin{remark}\label{remark:logomega}
    Remark \ref{remark:logxi} highlighted that, when the pressure law is logarithmic, $\xi$ is a constant. 
    Consequent to this, it follows that, when $f$ takes the logarithmic pressure the inhomogeneity $\omega$ is independent of time, satisfying the equality 
    \[
    \omega(t,x) = \omega(x) \coloneqq \exp(\lambda(V_1(x)-V_2(x))).
    \]
\end{remark}

\begin{proposition}\label{transform}
    Let $(\rho_1,\rho_2)$ denote a solution of System \eqref{eq:vcdid}.
    Then $(\sigma, r)$ defines a classical solution of the system
    \begin{align}
        &\begin{cases}
            \partial_t\sigma = \partial_x (\sigma(\eta\partial_x\log(\sigma)+ \partial_xf'(\sigma)+ \bs{v}_\sigma)), \hfill {\rm\ in\ } (0,T) \times \Omega. \\
            \partial_t r = \eta \partial_{xx}r + (2\eta\partial_x\log(\sigma)+\partial_xf'(\sigma) +\bs{v}_r)\partial_xr  + \bs{w}_rr(1-r),
        \end{cases} \label{eq:rsigmaevolution}\\
        &\begin{cases}
            0 = \sigma(\eta\partial_x\log(\sigma) + \partial_xf'(\sigma) + \bs{v}_\sigma),\\
            0 = \partial_x r, \hfill   \text{ on } (0,T) \times \partial\Omega,
        \end{cases}
      \label{eq:rsigmabc}
    \end{align}
    where the coefficient functions $\bs{v}_\sigma, \bs{v}_r, \bs{w}_r\co[0,T) \times \overline{\Omega}\to \R$ are defined as follows 
    \begin{align}
        \bs{v}_\sigma &\coloneqq r\partial_xV_1+(1-r)\partial_xV_2, \label{eq:vsigma} \\
        \bs{v}_r &\coloneqq r\partial_xV_2 + (1-r)\partial_xV_1,  \label{eq:vr} \\
        \bs{w}_r &\coloneqq \partial_{xx}(V_1-V_2) + \partial_x\log(\sigma)\partial_x(V_1-V_2). \label{eq:wr}
    \end{align}
\end{proposition}
\begin{proof}
To derive the equation governing $\sigma$, sum the equations of System \eqref{eq:vcdid} and then rewrite the subsequent boundary value problem in the variables $r,\sigma$.

\medskip 

To derive the evolution equation for $r$, observe that, $\sigma$ is strictly positive on $[0,T) \times \overline{\Omega}$ due to Proposition \ref{shorttime}.
Consequently, the following equality holds due to the product rule.
\begin{equation}\label{eq:dtr}
   \partial_t r = \f{\partial_t(r\sigma)-r\partial_t\sigma}{\sigma} = \f{\partial_t\rho_1 - r\partial_t\sigma}{\sigma}.
\end{equation}
Further, recall that $\rho_1,\rho_2$ solve System \eqref{eq:vcdid} and hence, $\rho_1$ and $\sigma$ each satisfy the respective equalities

\begin{align}
   \partial_t\rho_1 & = \partial_x(r\sigma \partial_x(V_1-V_2)) + \partial_x(r\sigma \partial_x(f'(\sigma) + V_2)) + \eta\partial_{xx}(r\sigma), \label{eq:dtrho1}\\
    \partial_t\sigma & = \partial_x(r\sigma \partial_x(V_1-V_2)) + \partial_x(\sigma \partial_x(f'(\sigma) + V_2)) + \eta\partial_{xx}\sigma.\label{eq:dtsigma} 
\end{align}
Substituting Equalities \eqref{eq:dtrho1} and \eqref{eq:dtsigma} within Equation \eqref{eq:dtr} then yields the following equality.
\begin{equation}\label{eq:dtr2}
\begin{split}
    \partial_t r & = \f{1-r}{\sigma}\partial_x(r\sigma \partial_x(V_1-V_2)) +\partial_x r \cdot\partial_x\left(f'(\sigma)+ V_2\right) \\
    & + \eta \partial_{xx} r + 2\eta \partial_x\log(\sigma) \partial_x r.
\end{split}
\end{equation}
Moreover, by expanding each derivative on the right hand-side of Equality \eqref{eq:dtr2} and factorising together terms which are of the same order in r, the desired equality is derived.
\begin{equation}\label{eq:logformaldtrho3}
\begin{split}
    \partial_t r & = \eta \partial_{xx}r + (2\eta\partial_x\log(\sigma)+\partial_xf'(\sigma) +r\partial_xV_2  + (1-r)\partial_xV_1)\partial_xr\\
    & +  (\partial_{xx}(V_1-V_2) + \partial_x\log(\sigma)\partial_x(V_1-V_2))r(1-r).
\end{split}
\end{equation}
To derive the boundary condition for $r$, consider that the following boundary conditions are satisfied on $(0,T)\times \partial\Omega$ by $\rho_1$ and $\sigma$ respectively
\begin{align}
   0 &= r\sigma \partial_x(V_1-V_2) + r\sigma \partial_x(f'(\sigma) + V_2)+ \eta\partial_x (r\sigma), \label{eq:bcrho1}\\
   0 &= r\sigma \partial_x(V_1-V_2) + \sigma \partial_x(f'(\sigma) + V_2) + \eta\partial_{x}\sigma.\label{eq:bcsigma} 
\end{align}
Multiply Equality \eqref{eq:bcsigma} by $r$ and then subtract the resulting equality from Equation \eqref{eq:bcrho1}.
This yields the equality
\begin{equation}\label{eq:bcr}
    0 = \sigma(r(1-r)\partial_x(V_1-V_2) +\eta\partial_x r).
\end{equation}
The density of $\sigma$ is strictly positive on $[0,T)\times \overline{\Omega}$, and so it is meaningful to divide through Equation \eqref{eq:bcr} by $\sigma$. 
In addition, Assumption \ref{J5} enforces that $\partial_xV_1=\partial_xV_2$ on $\partial\Omega$ which thus yields the desired boundary condition.
\end{proof}
\subsection{The Evolution Equation}
   \begin{definition}\label{coefficients}
        Let $ \eta \in I$, let $(\rho_1,\rho_2)$ denote a classical solution to System \eqref{eq:vcdid} and let $\phi^\eta$ be as given in Definition \ref{vdefnphi}. 
        Further, recall the functions $\bs{v}_r$ and $\bs{v}_\sigma$ from Proposition \ref{transform}.

        \medskip 
        
        The functions $a,m\co[0,T]\times \Omega \to \R$ are given by
    \begin{align}
    a & \coloneqq \partial_x f'(\sigma) + \bs{v}_r, \label{eq:Af} \\
    m & \coloneqq \partial_x(V_1-V_2)(\xi_\eta(\sigma)\bs{v}_r - \sigma \xi_\eta'(\sigma)\bs{v}_\sigma) + \partial_{xx}(V_1-V_2)\notag\\
    & + \int_0^x \partial_s(\xi'_\eta(\sigma)\partial_s(V_1-V_2))(\sigma\partial_s f'(\sigma) + \sigma\bs{v}_\sigma)\diff s 
     \label{eq:Cf}
  \end{align}
   and the functions $a^\eta,b^\eta,m^\eta \co[0,T]\times \Omega \to \R$ are given by 
       \begin{align}
       a^\eta & \coloneqq a+\eta((2-3\phi)\partial_x(V_1-V_2)\xi_\eta(\sigma)+ 2\partial_x\log(\sigma)), \label{eq:Afeta} \\
       b^\eta & \coloneqq m^\eta \phi^\eta(1-\phi^\eta), \label{eq:Bfeta}\\
        m^\eta & \coloneqq m + \eta(\partial_{xx}(V_1-V_2)\xi_\eta(\sigma) + (1-\phi^\eta)\left(\partial_x(V_1-V_2)\xi_\eta(\sigma)\right)^2)\notag\\  
    & +  \eta\int_0^x \partial_s(\xi'_\eta(\sigma)\partial_s(V_1-V_2))\partial_s\sigma \diff s .  \label{eq:Cfeta}  
        \end{align}
        \end{definition}
\begin{proposition}\label{phievolution}
    Let $ \eta \in I$, let $(\rho_1,\rho_2)$ denote a classical solution to System \eqref{eq:vcdid} and let $\phi^\eta$ be as given in Definition \ref{vdefnphi}. 
    Then, then the variable $\phi^\eta$ satisfies 
    \[
    \phi^\eta \in C([0,T];C^3(\overline{\Omega})) \cap C^1([0,T];C^1(\overline{\Omega})).
    \]
    Then, $\phi^\eta$
    satisfies the System
    \begin{equation}\label{eq:phievolution}
    \begin{cases}
    \begin{alignedat}{2}
         \partial_t \phi^\eta &= \eta\partial_{xx}\phi^\eta + a^\eta\partial_x\phi^\eta + b^\eta \hspace{3em} &&\mathrm{ in } \ (0,T) \times \Omega,\\
        \phi_0 &= \phi^\eta &&\mathrm{ on } \ \{0\} \times \overline{\Omega},\\
         \partial_x\phi &= 0 &&\mathrm{ on } \ (0,T) \times \partial\Omega.
    \end{alignedat}
    \end{cases}
    \end{equation}
  \end{proposition}
  \begin{proof}
       Consider $\eta$ fixed and denote $\phi^\eta = \phi$, dropping the super-script for ease of exposition.
         Firstly, assuming that the potentials satisfy Assumption \ref{K1} and that $\sigma$ is strictly positive on $[0,T]\times\overline{\Omega}$. Then, the coordinate transformation 
        \[
        (\rho_1,\rho_2)\mapsto (\sigma, \phi) 
        \]
        is $C^{3,1}$ smooth. Hence, 
        \[
        \phi \in C([0,T];C^3(\overline{\Omega})) \cap C^1([0,T];C^1(\overline{\Omega}))
        \] 
        thanks to the regularity established in Theorem \ref{classical}.

        \medskip 
        
        Subsequently, the variable $\phi$ satisfies
\begin{equation}\label{eq:dtphi1}
     \partial_t \phi = \f{\omega\partial_tr + r(1-r)\omega \partial_t\log(\omega)}{((1-r)+r\omega)^2}.
\end{equation}
We then recall here the evolution equation for $r$ from System \eqref{eq:rsigmaevolution}.
In particular, $r$ satisfies the following equality on $(0,T)\times \Omega$.
\begin{equation}\label{eq:revo}
\begin{split}
         \partial_t r & = \eta \partial_{xx}r + (2\eta\partial_x\log(\sigma) + \partial_xf'(\sigma) + \bs{v}_r) \partial_x r \\
         & + (\partial_{xx}(V_1-V_2) + \partial_x\log(\sigma)\partial_x(V_1-V_2))r(1-r). 
\end{split}
\end{equation}
To re-factorise Equation \eqref{eq:revo}, recognise the equality 
\begin{equation}\label{eq:refactorise}
    \partial_x\log(\sigma)= \f{1}{2\eta + \sigma f''(\sigma)} \cdot (2\eta\partial_x\log(\sigma) + \partial_xf'(\sigma)).
\end{equation}
Consequently for any function $j \in C(0,+\infty)$, $r$ satisfies the evolution equation 
\begin{equation}\label{eq:revo2}
\begin{split}
         & \partial_t r  = \eta \partial_{xx}r \\
         & + (2\eta\partial_x\log(\sigma) + \partial_xf'(\sigma) + \bs{v}_r)\cdot\left(\partial_x r +\partial_x(V_1-V_2) \cdot  \f{ r(1-r)j(\sigma)}{2\eta + \sigma f''(\sigma)} \right)\\
         & + r(1-r)\cdot\left(\partial_{xx}(V_1-V_2) - \partial_x(V_1-V_2) \cdot  \f{\bs{v}_r j(\sigma)}{2\eta + \sigma f''(\sigma)}\right) \\
         & +r(1-r)(1-j(\sigma))\partial_x\log(\sigma)\partial_x(V_1-V_2). 
\end{split}
\end{equation}
For now, the function $j$ is not fixed, but will later be chosen in according to the non-linearity $\xi_\eta$.
We now derive the evolution equation for $\partial_t\log(\omega)$ which, in turn, depends on the evolution $\partial_t \xi_\eta(\sigma)$.
Recall the evolution of $\sigma$ from System \eqref{eq:rsigmaevolution}.
Then, it follows from the chain rule that
\begin{equation}\label{eq:xievo}
\begin{split}
    \partial_t\xi_\eta(\sigma) & = \xi'_\eta(\sigma)(\eta +\sigma f''(\sigma))\partial_{xx}\sigma 
    \\
    & + \xi'_\eta(\sigma)\partial_\sigma(\sigma f''(\sigma) )|\partial_x\sigma|^2 +\xi'_\eta(\sigma)\partial_x(\sigma\bs{v}_\sigma).
\end{split}
\end{equation}
Consequently, by applying the Leibniz rule to facilitate the interchange of differentiation and integration, it follows that
\begin{equation}\label{eq:omegaevo}
\begin{split}
    \partial_t \log(\omega) = & -\int_0^x\partial_t\xi_\eta(\sigma)\partial_s(V_1-V_2)\diff s \\
    = & - \int_0^x\xi'_\eta(\sigma)(\eta +\sigma f''(\sigma))\partial_{ss}\sigma\partial_s(V_1-V_2) \diff s
    \\
    & - \int_0^x\xi'_\eta(\sigma)\partial_\sigma(\sigma f''(\sigma) )|\partial_s\sigma|^2\partial_s(V_1-V_2) \diff s 
    \\
    & -\int_0^x \xi'_\eta(\sigma)\partial_s(\sigma\bs{v}_\sigma)\partial_s(V_1-V_2) \diff s.
\end{split}
\end{equation}
Studying the evolution of $\|\partial_x \phi^\eta\|_{L^1(\Omega)}$ will involve taking a spatial derivative of $\partial_t \omega$.
Subsequently, to remove the higher order derivatives which feature in the second and fourth line of Equality \eqref{eq:omegaevo}, we can perform integration by parts, integrating the term $\partial_{ss}\sigma$ in the second line and 
$\partial_s(\sigma\bs{v}_\sigma)$ in the fourth line.

\medskip 

In particular, since it is assumed that $\partial_s(V_1-V_2)$ vanishes at $0 \in \partial\Omega$, there is no contribution from the boundary integral when evaluated at $0$ and hence, it follows that

\begin{equation}\label{eq:omegaevo2}
\begin{split}
    \partial_t \log(\omega)  = & - (\eta\partial_x\sigma +\sigma \partial_xf'(\sigma)+ \sigma \bs{v}_\sigma)\xi_{\eta}'(\sigma)\partial_x(V_1-V_2)
    \\
    & 
    + \int_0^x\partial_s(\xi'_\eta(\sigma)(\eta +\sigma f''(\sigma))\partial_s(V_1-V_2))\partial_{s}\sigma \diff s
    \\
    & - \int_0^x\xi'_\eta(\sigma)\partial_\sigma(\sigma f''(\sigma) )|\partial_s\sigma|^2\partial_s(V_1-V_2) \diff s 
    \\
    & + \int_0^x \partial_s(\xi'_\eta(\sigma)\partial_s(V_1-V_2))\sigma\bs{v}_\sigma \diff s.
\end{split}
\end{equation}
The expression further simplifies since 
\begin{align*}
 & \partial_x(\xi'_\eta(\sigma)(\eta+\sigma f''(\sigma))\partial_x(V_1-V_2))\partial_x\sigma \\
 & - \xi'_\eta(\sigma)\partial_\sigma(\sigma f''(\sigma) )|\partial_x\sigma|^2\partial_x(V_1-V_2)\\
    & = \partial_x(\xi'_\eta(\sigma)\partial_x(V_1-V_2))(\eta\partial_x\sigma+\sigma \partial_xf'(\sigma)).
\end{align*}
Subsequently, $\partial_t\log(\omega)$ satisfies the equality
\begin{equation}\label{eq:omegaevo3}
\begin{split}
    \partial_t \log(\omega)  = & - (\eta\partial_x\sigma +\sigma \partial_xf'(\sigma)+ \sigma \bs{v}_\sigma)\xi_{\eta}'(\sigma)\partial_x(V_1-V_2)
    \\
    & 
    + \int_0^x \partial_s(\xi'_\eta(\sigma)\partial_s(V_1-V_2))(\eta\partial_s\sigma+\sigma \partial_sf'(\sigma)) \diff s
    \\
    & + \int_0^x \partial_s(\xi'_\eta(\sigma)\partial_s(V_1-V_2))\sigma\bs{v}_\sigma \diff s.
\end{split}
\end{equation}
Now, observe that the last line of Equality \eqref{eq:revo2} features a term of the form 
\begin{equation}\label{eq:amalgam1}
    r(1-r)(1-j(\sigma))\partial_x\log(\sigma)\partial_x(V_1-V_2)
\end{equation}
whilst the first line of Equality \eqref{eq:omegaevo3} features a term of the form 
\begin{equation}\label{eq:amalgam2}
    -\xi_{\eta}'(\sigma)\sigma \partial_xf'(\sigma)\partial_x(V_1-V_2).
\end{equation}
In particular, if $j(\sigma)$ is chosen to be the continuous function given by  
\[
1- \xi_\eta'(\sigma) f''(\sigma)\sigma^2 = j(\sigma)
\]
then the summation of Term \eqref{eq:amalgam1} with $r(1-r)$ multiplied by Term \eqref{eq:amalgam2} yields the equality
\begin{equation}\label{eq:cancel}
     r(1-r)\partial_x(V_1-V_2)((1-j(\sigma))\partial_x\log(\sigma)- \xi_{\eta}'(\sigma)\sigma \partial_xf'(\sigma)) = 0.
\end{equation}
Fixing this particular choice of $j$, we recall from Definition \ref{xidefn} that $\xi_\eta$ was defined to satisfy Equation \eqref{eq:ode} and consequently, $j$ satisfies the equality
\begin{equation}\label{eq:factor}
     \f{j(\sigma)}{2\eta + \sigma f''(\sigma)}  = - \xi_\eta(\sigma).
\end{equation}
Thereby, when amalgamating Equalities \eqref{eq:omegaevo3} and \eqref{eq:revo2}, we are able to cancel the two Terms \eqref{eq:amalgam1} and \eqref{eq:amalgam2} thanks to Equality \eqref{eq:cancel} and also re-write the term featured in the second factor of the second line of Equality \eqref{eq:revo2} in terms of $\xi_\eta$ thanks to Equality \eqref{eq:factor}.
In particular, it follows that 
\[
\partial_t r + r(1-r)\partial_t\log(\omega)
\]
satisfies the equality
\begin{equation}\label{eq:rlong}
\begin{split}
    & \partial_t r +r(1-r)\partial_t\log(\omega) = \eta \partial_{xx}r \\
         & + (2\eta\partial_x\log(\sigma) + \partial_xf'(\sigma) + \bs{v}_r)\cdot\left(\partial_x r-\partial_x(V_1-V_2) r(1-r)\xi_\eta(\sigma) \right)\\
         & + r(1-r)\cdot\left(\partial_{xx}(V_1-V_2) + \partial_x(V_1-V_2)\bs{v}_r \xi_\eta(\sigma)\right) \\
         & - r(1-r) \sigma\xi_{\eta}'(\sigma)\bs{v}_\sigma\partial_x(V_1-V_2) - \eta r(1-r)\partial_x\xi_\eta(\sigma)\partial_x(V_1-V_2)
    \\
    & 
    + r(1-r)\int_0^x \partial_s(\xi'_\eta(\sigma)\partial_s(V_1-V_2))(\eta\partial_s\sigma+\sigma \partial_sf'(\sigma)) \diff s \\
    & + r(1-r)\int_0^x \partial_s(\xi'_\eta(\sigma)\partial_s(V_1-V_2))\sigma\bs{v}_\sigma \diff s.
\end{split}
\end{equation}
Now, $\phi$ satisfies both of the equalities 
\begin{align}
    \phi(1-\phi) & = \f{r(1-r)\omega}{((1-r)+r\omega)^2},\label{eq:phiphi2} \\
    \partial_x \phi & = \frac{\omega\left(\partial_xr - r(1-r)\partial_x(V_1-V_2)\xi_\eta(\sigma)\right)}{((1-r)+r\omega)^2}\notag\\
    & = \f{\phi}{r} \partial_x r - \phi(1-\phi) \partial_x(V_1-V_2)\xi_\eta(\sigma) \label{eq:dxphi2}
\end{align}
in addition to Equation \eqref{eq:dtphi1}.
Consequently, if both sides of Equality \eqref{eq:rlong} are multiplied by a factor of
\[
\f{\omega}{((1-r)+r\omega)^2},
\]
then it follows from the chain rule that 
\begin{equation}\label{eq:rlong2}
\begin{split}
    & \partial_t \phi = \eta \frac{\phi}{r}\partial_{xx}r - \eta\phi(1-\phi) \partial_x\xi_\eta(\sigma)\partial_x(V_1-V_2)\\
         & + (2\eta\partial_x\log(\sigma) + \partial_xf'(\sigma) + \bs{v}_r)\partial_x\phi \\
         & + \phi(1-\phi)\cdot\left(\partial_{xx}(V_1-V_2) + \partial_x(V_1-V_2)\bs{v}_r \xi_\eta(\sigma)\right) \\
         & -  \phi(1-\phi)\sigma\xi_{\eta}'(\sigma) \bs{v}_\sigma \partial_x(V_1-V_2)
    \\
    & 
    +  \phi(1-\phi)\int_0^x \partial_s(\xi'_\eta(\sigma)\partial_s(V_1-V_2))(\eta\partial_s\sigma+\sigma \partial_sf'(\sigma)) \diff s \\
    & +  \phi(1-\phi)\int_0^x \partial_s(\xi'_\eta(\sigma)\partial_s(V_1-V_2))\sigma\bs{v}_\sigma \diff s.
\end{split}
\end{equation}
By considering the spatial differentiation of Equality \eqref{eq:dxphi2}, we now address the second order term in Equation \eqref{eq:rlong}.
In particular, $\partial_{xx}\phi$ satisfies the equality
\begin{equation}\label{eq:dxxphi3}
\begin{split}
     \partial_{xx}\phi & = \frac{\phi}{r}\partial_{xx}r - \phi(1-\phi) \partial_x\xi_\eta(\sigma)\partial_x(V_1-V_2) \\ 
     & + \partial_x \phi\left(\partial_x\log(r) - (1-2\phi)\partial_x(V_1-V_2)\xi_\eta(\sigma)\right) \\
    & - \phi(1-\phi)\partial_{xx}(V_1-V_2)\xi_\eta(\sigma) - \underbrace{\phi \left(\f{\partial_x r}{r}\right)^2}_{\coloneqq \mathcal{I}} 
\end{split}
\end{equation}

In particular, recalling Equality \eqref{eq:dxphi2}, the term $\mathcal{I}$ satisfies
\begin{equation}\label{eq:I2}
\begin{split}
    \mathcal{I} & = \partial_x \phi\partial_x\log(r) + \phi\f{\partial_xr}{r}(1-\phi)\partial_x(V_1-V_2)\xi_\eta(\sigma) \\
    & = \partial_x \phi(\partial_x\log(r) + (1-\phi)\xi_\eta(\sigma)\partial_x(V_1-V_2)) \\
    & +\phi(1-\phi)\cdot (1-\phi) \left(\partial_x(V_1-V_2)\xi_\eta(\sigma)\right)^2.
\end{split}
\end{equation}
Consequently, by substituting Equality \eqref{eq:I2} within Equation \eqref{eq:dxxphi3}, it follows that the second derivative satisfies the equality
\begin{equation}\label{eq:dxxphi4}
\begin{split}
     \partial_{xx}\phi & = \frac{\phi}{r}\partial_{xx}r - \phi(1-\phi) \partial_x\xi_\eta(\sigma)\partial_x(V_1-V_2) \\ 
     & - \left((2-3\phi)\partial_x(V_1-V_2)\xi_\eta(\sigma)\right) \partial_x \phi\\
    & - \phi(1-\phi)(\partial_{xx}(V_1-V_2)\xi_\eta(\sigma) + (1-\phi)\left(\partial_x(V_1-V_2)\xi_\eta(\sigma)\right)^2).
\end{split}
\end{equation}
Substituting Equality \eqref{eq:dxxphi4} into the first line of Equation \eqref{eq:rlong2}, it follows that the evolution of $\phi$ satisfies the equality

\begin{equation}\label{eq:rlong3}
\begin{split}
    & \partial_t \phi = \eta \partial_{xx}\phi + (\partial_xf'(\sigma) + \bs{v}_r)\partial_x\phi \\
         & +\eta \left((2-3\phi)\partial_x(V_1-V_2)\xi_\eta(\sigma) + 2\partial_x\log(\sigma)\right)\partial_x \phi\\
         & + \phi(1-\phi)\cdot\left(\partial_{xx}(V_1-V_2) + \partial_x(V_1-V_2)(\bs{v}_r \xi_\eta(\sigma)- \bs{v}_\sigma \sigma\xi_\eta'(\sigma))\right) \\
    & 
    +  \phi(1-\phi)\int_0^x \partial_s(\xi'_\eta(\sigma)\partial_s(V_1-V_2))(\sigma\bs{v}_\sigma + \sigma \partial_sf'(\sigma)) \diff s \\
    & +\eta \phi(1-\phi)(\partial_{xx}(V_1-V_2)\xi_\eta(\sigma) + (1-\phi)\left(\partial_x(V_1-V_2)\xi_\eta(\sigma)\right)^2)\\
    & 
    +  \eta\phi(1-\phi)\int_0^x \partial_s(\xi'_\eta(\sigma)\partial_s(V_1-V_2))\partial_s\sigma \diff s .
\end{split}
\end{equation}
Lastly, recalling Equalities \eqref{eq:Afeta}, \eqref{eq:Bfeta} and \eqref{eq:Cfeta} satisfied by the coefficients $a^\eta, b^\eta,m^\eta$, Equation \eqref{eq:rlong3} must coincide with the equality
\begin{align*}
    \partial_t\phi & = \eta\partial_{xx}\phi + a^\eta\partial_x\phi + m^\eta \phi(1-\phi)\\
    & =  \eta\partial_{xx}\phi + a^\eta\partial_x\phi + b^\eta.
\end{align*}
  \end{proof}
\subsection{An Inequality Between Ratios}
Since $\phi^\eta$ satisfies Equation \eqref{eq:phievolution}, Property \ref{B1} is satisfied.
Here, it is further established that the pair $(r,\phi^\eta)$ satisfies Property \ref{B3} for suitably smooth drifts.
\begin{lemma}\label{phir}
    Let $r,\phi^\eta$ be as in Definitions \ref{defnrsigma} and \ref{vdefnphi} respectively. 
    Then, there exist $\Lambda_1,\Lambda_2 > 0$, whose size depends only on the Lipschitz constant of $V_1-V_2$, such that the following inequalities hold for each $t \in [0,T]$.
     \begin{align}
     \|\partial_x \phi^\eta(t)\|_{L^1(\Omega)} 
     &\leqslant \Lambda_1(\|\partial_x r(t)\|_{L^1(\Omega)}+1), \label{eq:phir2}\\
    \|\partial_x r(t)\|_{L^1(\Omega)} 
    &\leqslant \Lambda_2(\|\partial_x \phi^\eta(t) \|_{L^1(\Omega)}+1 ).\label{eq:rphi2}
    \end{align}
    By taking the supremum of the above inequalities over $t \in [0,T]$, it then follows that
    \begin{align*}
     \|\partial_x \phi^\eta\|_{L^\infty([0,T];L^1(\Omega))} 
     &\leqslant \Lambda_1(\|\partial_x r\|_{L^\infty([0,T];L^1(\Omega))}+1), \\
    \|\partial_x r\|_{L^\infty([0,T];L^1(\Omega))}  
      &\leqslant  \Lambda_2(\|\partial_x \phi^\eta \|_{L^\infty([0,T];L^1(\Omega))}+1 ).
    \end{align*}
\end{lemma}
\begin{proof}
Consequent to the chain rule, the variable $\phi^\eta$ satisfies
\[
    \partial_x \phi^\eta = \frac{\omega\left(\partial_xr +r(1-r)\partial_x \log(\omega)\right)}{((1-r)+r\omega)^2}.
\]
Hence, it follows that
\begin{equation}\label{eq:ineqomegaphi}
|\partial_x \phi^\eta| \leqslant c \left(|\partial_x r| +|\partial_x\log(\omega)|\right), \quad \quad 
|\partial_x r |\leqslant c |\partial_x \phi^\eta| + |\partial_x \log(\omega)|
\end{equation}
where
\[
c\coloneqq \sup_{[0,T]\times\Omega} \frac{(1+ \omega)^2}{\omega} = 2 + \sup_{[0,T]\times\Omega} \left(\omega +   \frac{1}{\omega}\right).
\]
The quantities $\omega$ and $\frac{1}{\omega}$ are bounded if $\log(\omega)$ is bounded. 
Indeed, this can be seen via the inequalities
\[
|\omega| \leqslant \exp\left(|\log(\omega)|\right) \text{ and } \left|\frac{1}{\omega}\right| \leqslant \exp\left(|\log\left(\frac{1}{\omega}\right)|\right) =  \exp\left(|\log(\omega)|\right).
\]
Further, acknowledging the inequality
\begin{align*}
   \sup_{x\in \Omega} |\log(\omega(t,x))| & \leqslant \sup_{\Omega}\int_0^x |\partial_s\log(\omega(t,s))|\diff s  + |\log(\omega(t,0))| \\
   & \leqslant \|\partial_x\log(\omega(t))\|_{L^1(\Omega)}
\end{align*}
means that the size of $c$ can be chosen to depend only on
\[
\|\partial_x \log(\omega)\|_{L^1(\Omega)}.
\]
In particular, it follows from Definition \ref{vdefnphi} that $\log(\omega(t,0)) = 0$ for every $t \in [0,T]$.

\medskip 

For the general pressure law, the following inequality is satisfied.
\begin{equation}\label{eq:omegaholder}
\begin{split}
\|& \partial_x\log(\omega)\|_{L^\infty([0,T];L^1(\Omega))} = \|\xi_\eta(\sigma)\partial_x(V_1-V_2)\|_{L^\infty([0,T];L^1(\Omega))} \\
& \leqslant \|\xi_\eta(\sigma)\|_{L^\infty([0,T];L^1(\Omega))}\|\partial_x(V_1-V_2)\|_{L^\infty(\Omega)}.
\end{split}
\end{equation}
Now, for $f$ satisfying Hypothesis \ref{hypothesis2}, we recall from Inequality \eqref{power:zero} that, when Assumption \ref{I2} is satisfied, $\xi_\eta$ satisfies 
\[
\|\xi_\eta(\sigma)\|_{L^\infty([0,T];L^1(\Omega))}\leqslant \kappa\|\sigma^{1-\alpha}\|_{L^\infty([0,T];L^1(\Omega))}.
\]
Otherwise, if $f$ satisfies Hypothesis \ref{hypothesis1}, then $\xi$ is either constant for the logarithmic pressure or obeys a similar power law type inequality.

\medskip 

Moreover, since $\rho_{1,0},\rho_{2,0} \in \sP(\Omega)$, $\sigma$ satisfies
\[
\|\sigma\|_{L^\infty([0,T];L^1(\Omega))} =2 
\]
and since $\alpha \in (0,1)$, it follows from Jensen's inequality that
\begin{equation}\label{eq:omegajensen}
  \|\xi_\eta(\sigma)\|_{L^\infty([0,T];L^1(\Omega))}\leqslant \kappa\|\sigma^{1-\alpha}\|_{L^\infty([0,T];L^1(\Omega))}\leqslant \kappa |\Omega|^\alpha 2^{1-\alpha}.
\end{equation}
Inequality \eqref{eq:omegajensen} may be used to derive an upper bound for the second line in Inequality \eqref{eq:omegaholder}. 
In particular, it follows that 
\[
\|\partial_x\log(\omega)\|_{L^\infty([0,T];L^1(\Omega))} \leqslant \kappa |\Omega|^\alpha 2^{1-\alpha}\|\partial_x(V_1-V_2)\|_{L^\infty(\Omega)} < +\infty.
\]

Having concluded that the constant $c$ may be taken to depend only on the Lipschitz constant of $V_1-V_2$, integrating each Inequality in \eqref{eq:ineqomegaphi} over $\Omega$ yields the first stated inequality.
Then, since the inequality is uniform over $[0,T]$, taking the supremum in time yields the second. 
\end{proof}

\section{Aggregate Estimates}\label{sec:aggregate}
    In Proposition \ref{phievolution}, it was proven that $\phi^\eta$ satisfies an evolution equation satisfying Property \ref{B1}.
    Observing Equality \eqref{eq:Cfeta}, it is clear that the coefficient function $b^\eta$ depends on the aggregate density $\sigma$.
    Consequently, this section is dedicated to establish sufficient Sobolev estimates for $\sigma$ and non-linear functions depending on $\sigma$ with which to establish Property \ref{B2}.

    \medskip

    Indeed, such estimates were already established in \cite{meszarosParker} and \cite{elbarfilippo2025} for the logarithmic pressure and for the fast-diffusive power law and so, instead of addressing $f$ satisfying Hypothesis \ref{hypothesis1}, our attention is instead focused on the general setting of Hypothesis \ref{hypothesis2}.

    \medskip 
    
    The analysis begins with the following Lemma, proven via the Sobolev embedding theory, and may be seen as a sharpening of the argument presented in \cite[Proposition 2.1]{elbarfilippo2025}.
    \begin{lemma}\label{cfineq}
        Suppose that $\sigma \in L^\infty([0,T];L^1(\Omega))$ and that there exists $p>  0$ such that $\partial_x \sigma^\frac{p}{2} \in L^2([0,T];L^2(\Omega))$.
        Then, there exists $\mathcal{K}_{0,p} > 0$ whose size depends only on $p, |\Omega|, T$ and 
        \[
        \|\sigma\|_{L^\infty([0,T];L^1(\Omega))}
        \]
        such that 
        \[
        \|\sigma\|_{L^{p+2}([0,T]\times\Omega)}^{p+2} \leqslant \mathcal{K}_{0,p}\left(\|\partial_x \sigma^\frac{p}{2} \|_{ L^2([0,T];L^2(\Omega))}^2+1\right).
        \]
    \end{lemma} 
    \begin{proof}
        Denote 
        \[
        c_1 \coloneqq \|\sigma\|_{L^\infty([0,T];L^1(\Omega))}.
        \]
        Then, the following inequality is satisfied by means of the $L^\infty-L^1$ H\"older inequality.
        \[
        \int_0^T \int_{\Omega}\sigma ^{p+2} \diff x \diff t \leqslant   c_1 \int_0^T\|\sigma^{p+1}(t)\|_{L^\infty(\Omega)} \diff t.
        \]
        Subsequently, since $W^{1,1}(\Omega) \hookrightarrow L^\infty(\Omega)$, there exists $c_2$ depending only on $ \Omega $ and $c_1$ for which 
        \begin{align*}
        \int_0^T & \int_{\Omega}\sigma ^{p+2} \diff x \diff t \leqslant   c_2 \int_0^T\int_{\Omega} |\partial_x  \sigma ^{p+1}| + |\sigma^{p+1}| \diff x \diff t \\
        & \leqslant c_2 \int_0^T \int_{\Omega}2\left(1+\f{1}{p}\right)\left|\partial_x \sigma^\frac{p}{2}\right|\sigma^{\frac{p+2}{2}} + |\sigma^{p+1}|  \diff x \diff t.
        \end{align*}
        Apply the $L^2$-H\"older inequality to the first summand in the second line of the above inequality and apply Jensen's inequality to the second summand to deduce that 
        \begin{equation}\label{eq:intermediate}
        \begin{split}
                \int_0^T & \int_{\Omega}\sigma ^{p+2} \diff x \diff t \leqslant c_2(|\Omega|T)^{\frac{1}{p+2}}\left(\int_0^T \int_{\Omega}|\sigma ^{p+2}|\diff x \diff t\right)^\frac{p+1}{p+2}\\
        &+ 2c_2\left(1+\f{1}{p}\right)\sqrt{\int_0^T\int_{\Omega} \left|\partial_x \sigma^\frac{p}{2}\right|^2 \diff x \diff t }\sqrt{\int_0^T \int_{\Omega}\sigma^{p+2} \diff x \diff t}.
        \end{split}
        \end{equation}
        Now, suppose that 
        \begin{equation}\label{eq:presume}
             \int_0^T \int_{\Omega}\sigma ^{p+2} \diff x \diff t \geq (2c_2)^{p+2}|\Omega| T.
        \end{equation}
        Then, it follows that 
        \begin{align*}
                \int_0^T & \int_{\Omega}\sigma ^{p+2} \diff x \diff t -  c_2(|\Omega|T)^{\frac{1}{p+2}}\left(\int_0^T \int_{\Omega}\sigma ^{p+2} \diff x \diff t\right)^\frac{p+1}{p+2}\\
                & \geqslant \frac{1}{2}  \int_0^T \int_{\Omega}\sigma ^{p+2} \diff x \diff t.
        \end{align*}
        Continuing to suppose that Inequality \eqref{eq:presume} holds, it follows from Inequality \eqref{eq:intermediate} that
        \begin{equation}\label{eq:intermediate2}
            \begin{split}
                \frac{1}{2}\sqrt{\int_0^T \int_{\Omega}\sigma ^{p+2} \diff x \diff t} & \leqslant 2c_2\left(1+\f{1}{p}\right)\sqrt{\int_0^T\int_{\Omega} \left|\partial_x \sigma^\frac{p}{2}\right|^2 \diff x \diff t }.
            \end{split}
        \end{equation}
        Square each side of Inequality \eqref{eq:intermediate2} and then denote 
        \[
        c_3 \coloneqq \left(4 \left(1+\f{1}{p}\right) c_2\right)^2.
        \]
        Consequently, the inequality 
        \[
        \|\sigma\|_{L^{p+2}([0,T]\times\Omega)}^{p+2} \leqslant c_3 \|\partial_x \sigma^\frac{p}{2}\|_{L^2([0,T]\times\Omega)}^2 
        \]
        holds for all $\sigma$ satisfying Inequality \eqref{eq:presume}.

        \medskip 
        
        Moreover, without assuming Inequality \eqref{eq:presume}, the following Inequality holds
        \[
        \|\sigma\|_{L^{p+2}([0,T]\times\Omega)}^{p+2} \leqslant \mathcal{K}_{0,p}\left(\|\partial_x \sigma^\frac{p}{2}\|_{L^2([0,T]\times\Omega)}^2 +1\right) 
        \]
        where $\mathcal{K}_{0,p} \coloneqq \max\{c_3, (2c_2)^{p+2}|\Omega|T\}$.
        In particular, the value of $\mathcal{K}_{0,p}$ depends only on $|\Omega|, T, p$ and the quantity $\|\sigma\|_{L^\infty([0,T];L^1(\Omega))}$.
    \end{proof} 
\begin{thm}\label{sumestimates}
    Let $f$ satisfy Hypothesis \ref{hypothesis1} or \ref{hypothesis2}.
        If $f$ satisfies Hypothesis \ref{hypothesis2} recall the constants $\kappa$ and $\alpha$. 
     Let $\eta \in I$, $\theta \in [1,+\infty)$ and let $(\rho_1,\rho_2)$ denote a solution of System \eqref{eq:vcdid} with initial data $\rho_{1,0},\rho_{2,0} \in L^\theta(\Omega)\cap \sP(\Omega)$. 
     Then, if $\theta \in (1,+\infty)$, there exists a constant $\mathcal{K}_{1,\theta}$ which depends only on $\alpha,\kappa, \theta, |\Omega|,T$, the size of the norm 
     \[
     \|\sigma_0\|_{L^{\theta}(\Omega)},
     \] 
     and the Lipschitz constants of the potentials, such that
     \begin{equation}\label{eq:thetamuineq}
     \begin{split}
     \left\|\sigma^{p}\right\|_{L^\infty([0,T];L^1(\Omega))} & + \left\|\partial_x\sigma^\f{p+\alpha-1}{2}\right\|_{L^2([0,T]\times \Omega)} \\
     & < \caK_{1,\theta} \text{ for every } p\in \left(0,\theta\right]
     \end{split}
     \end{equation}
     with the choice $p = 1-\alpha $ replaced by $\partial_x \log(\sigma)$ and the $\alpha$ taking the value $1$ for the logarithmic pressure law.
     In particular, the constant $\caK_{1,\theta}$ is independent of $\eta$.

     \medskip 

     If $\theta = 1$, then Inequality \eqref{eq:thetamuineq} instead holds for every $p \in \left(0,1\right)$.
\end{thm}
\begin{proof}
For any $p \in \R\setminus \{1,0\}$, a solution of System \eqref{eq:rsigmaevolution} satisfies the dissipation equality
\begin{equation}\label{eq:sumede} 
\begin{split}
\f{1}{p(p-1)}\partial_t\int_{\Omega}  \sigma^p \diff x & + \int_{\Omega} \sigma^{p-1}\left(f''(\sigma) + \frac{\eta}{\sigma}\right)|\partial_x\sigma|^2 \diff x \\
& = -\int_{\Omega} \sigma^{p-1} \partial_x \sigma \bs{v}_\sigma \diff x.
\end{split}
\end{equation}    
Indeed, such an equality is derived by integrating by parts and utilising the no-flux boundary condition.

\medskip 

In the case that the pressure law is then $f''(\sigma) = \frac{1}{\sigma}$ and the estimates follow from standard argument for the linear Fokker--Planck equation (see, for example \cite[Proposition 3.2]{DiMarinoSantambrogio}), on the other hand, if $f$ satisfies Hypothesis \ref{hypothesis1} with a power law or \ref{hypothesis2}, then we may use Assumption \ref{I2} to provide a lower bound for the first line of Inequality \eqref{eq:sumede} and Young's $L^2$ Inequality to act as an upper bound for the second line. 

\medskip 

Subsequently, the following dissipation inequality holds for each $t \in (0,T)$.

\begin{align*}
& \f{1}{p(p-1)} \partial_t\int_{\Omega}  \sigma^p \diff x + \frac{1}{\alpha\kappa}\sigma^{p-3+\alpha}|\partial_x\sigma|^2 \diff x \\
& \leqslant  \int_{\Omega} \sigma^{p-1} \partial_x \sigma \bs{v}_\sigma \diff x
\\
& \leqslant  \int_{\Omega} \frac{1}{2\alpha\kappa}\sigma^{p-3+\alpha}|\partial_x\sigma|^2 + \frac{\alpha\kappa}{2}|\bs{v}_\sigma|^2 \sigma^{p+1-\alpha} \diff x\\
& \leqslant  \frac{1}{2\alpha\kappa}\int_{\Omega} \sigma^{p-3+\alpha}|\partial_x\sigma|^2 \diff x+ \frac{\alpha\kappa}{2}\|\bs{v}_\sigma\|^2_{L^\infty([0,T]\times\Omega)} \int_{\Omega} \sigma^{p+1-\alpha} \diff x.
\end{align*}    

In particular,
\[
\|\bs{v}_\sigma\|^2_{L^\infty([0,T]\times\Omega)} \leqslant \sum_{i=1}^2 \|\partial_xV_i\|_{L^\infty(\Omega)} \eqqcolon c_1
\]
 and so, if $p \neq 1-\alpha$, then, integrating over $[0,T]$, it follows that
\begin{equation}\label{eq:sumedi}
\begin{split}
     & \f{1}{p(p-1)} \left(\int_{\Omega}  \sigma^p \diff x \right)\bigg |_{t=T} + \frac{2}{\alpha\kappa(p+\alpha-1)^2}\int_0^T\int_{\Omega} \left|\partial_x\sigma^\frac{p+\alpha-1}{2}\right|^2 \diff x \diff t \\
      & \leqslant  \f{1}{p(p-1)} \left(\int_{\Omega}  \sigma^p \diff x \right)\bigg |_{t=0}  + \frac{\alpha\kappa}{2}c_1 \int_0^T\int_{\Omega} \sigma^{p+1-\alpha} \diff x \diff t.
\end{split}
\end{equation}
If $ p > 1$, then 
\[
\frac{1}{p(p-1)} > 0
\]
and hence, 
it follows from Inequality \eqref{eq:sumedi} that 
\[
\|\sigma^p\|_{L^\infty([0,T];L^1(\Omega))} + \|\partial_x\sigma^\frac{p+\alpha-1}{2}\|_{L^2([0,T]\times\Omega)}
\]
may be bounded in terms of $c_1$,
\[
\|\sigma^{p+1-\alpha}\|_{L^1([0,T]\times\Omega)} \text{ and } \|\sigma^p_0\|_{L^1(\Omega)}.
\]
On the other hand, equipped with the assumption $\rho_{1,0},\rho_{2,0} \in \sP(\Omega)$, it follows that 
\[
\|\sigma\|_{L^\infty([0,T];L^1(\Omega))} =2.
\]
Consequent to the conservation of mass, it follows from Jensen's inequality that, for any $p \in (0,1]$, the density $\sigma^p$ is bounded in $L^1([0,T]\times\Omega)\cap L^\infty([0,T]:L^1(\Omega))$
by a constant depending only on $T$ and $|\Omega|$.
Moreover, if $p \in (0,1)$ then it follows from Inequality \eqref{eq:sumedi} that
\[
\|\partial_x\sigma^\frac{p+\alpha-1}{2}\|_{L^2([0,T]\times\Omega)}
\]
may be bounded by the norm 
\[
\|\sigma^{p+1-\alpha}\|_{L^1([0,T]\times\Omega)},
\]
 the constant $c_1$ and a constant which depends only on $\alpha,\kappa,\Omega, T$.

\medskip 

Indeed, by first choosing $p\in (0,\alpha]$, it follows from Inequality \eqref{eq:sumedi} that
\begin{equation}\label{eq:bootbegin}
    \|\partial_x\sigma^\frac{p+\alpha-1}{2}\|_{L^2([0,T]\times\Omega)}, \text{ for } \ p \in (0, \alpha],
\end{equation}

is bounded by a constant whose size depends only on $\alpha,\kappa,|\Omega|, T$ and the Lipschitz constants of the potentials. 

To produce the desired integrability for further positive values of $p$, we recall Lemma \ref{cfineq}.
In particular, since 
\[
\|\sigma\|_{L^\infty([0,T];L^1(\Omega))} =2,
\]
there exists a constant $\caK_{0,p+\alpha-1}$, whose size depends only on $p,\alpha, |\Omega|, T$, such that 
\begin{equation}\label{eq:bootstrap}
    \|\sigma\|_{L^{p+\alpha+1}([0,T]\times\Omega)}^{p+\alpha+1} \leqslant \caK_{0,p+\alpha-1}\left(\|\partial_x \sigma^\frac{p+\alpha-1}{2} \|_{ L^2([0,T];L^2(\Omega))}^2+1\right).
\end{equation}
By using Inequality \eqref{eq:bootstrap} in conjunction with Inequality \eqref{eq:sumedi} it is possible to employ a bootstrapping argument, which will be detailed in the following paragraphs, in order to ensure that the second summand in the second line of Inequality \eqref{eq:sumedi} is bounded, for any $p \in (0,\theta]$, by a constant which depends only on $\kappa, p,\alpha$, the Lipschitz constants of the potentials, $T$ and $|\Omega|$.

\medskip 

Moreover, by facilitating such a range of $p$, it is established via Inequality \eqref{eq:sumedi} that 
 \begin{align*}
    \|\sigma^{p+1-\alpha}\|_{L^\infty([0,T];L^1(\Omega))} & + \|\partial_x\sigma^\f{p}{2}\|_{L^2([0,T]\times \Omega)} \\
    & < \caK_{1,\theta} \text{ for every } p \in \left(0,\theta\right]
 \end{align*}
The bootstrapping argument reads as follows:

\medskip 

Since quantity \eqref{eq:bootbegin} is bounded, Inequality \eqref{eq:bootstrap} may be utilised for any choice of $p_1\in (0,\alpha]$.
Indeed, for this range of $p_1$, 
\[
\|\sigma^{p_1+\alpha+1}\|_{L^{1}([0,T]\times\Omega)}
\]
may be bounded by a constant whose size depends only on $\alpha,\kappa, p_1, |\Omega|, T$ and the Lipschitz constants of the potentials. 
Moreover, equipped with this knowledge, it is possible to return to Inequality \eqref{eq:sumedi} with the substitution $p = p_2 = \min\{\theta,p_1+ 2\alpha\}$.
In particular, for the choice $p_2$, it is known that the second line of Inequality \eqref{eq:sumedi} will be suitably bounded thanks to Inequality \eqref{eq:bootstrap}.
If $\min\{\theta,p_1+ 2\alpha\} = \theta$ then the bootstrapping argument is complete.

\medskip 

On the other hand, if  $\min\{\theta,p_1+ 2\alpha\} = p_1+2\alpha$, then Inequality \eqref{eq:sumedi} establishes a bound for 
\begin{equation}\label{eq:boot2}
    \|\partial_x\sigma^\frac{p_1 +3\alpha-1}{2}\|_{L^2([0,T]\times\Omega)}.
\end{equation}
From this estimate, we may return to Inequality \eqref{eq:bootstrap} for the choice $p = p_1+3\alpha -1$ and repeat the argument. 
In particular, each further iterate of the bootstrapping method will raise the range of admissible $p$ by $2\alpha$.
Moreover, since $\alpha > 0$, it is possible to iterate the method a finite number of times to achieve integrability up to the value of $\theta < + \infty$ which was fixed in the statement of the Theorem.

\medskip 

Furthermore, each of the bounds introduced in the bootstrapping argument depended only on $\alpha,\kappa, p, |\Omega|,T$ the Lipschitz constants of the potentials and the integrability of the initial data.
Specifically the dependence on the initial data induced by each iteration of the bootstrap is only through norm 
\[
\|\sigma_0\|_{L^{p_i}(\Omega)}
\]
which, by Jensen's inequality may be dominated by the norm 
\[
\|\sigma_0\|_{L^{\theta}(\Omega)}
\]
multiplied by a constant which depends only on $|\Omega|$.
Moreover, by aggregating the bounds attained from each iterate of the bootstrap and letting $\mathcal{K}_{1,\theta}$ denote the subsequent constant, we conclude the claim.

\medskip 

The case $\theta =1$ must be treated with care since Inequality \eqref{eq:sumedi} does not hold for $p=1$.
In this instance, one may instead consider the bootstrapping argument up to any $p \in (0,1)$.
\end{proof}

The previous theorem established $p$-integrability for the aggregate density uniformly in $\eta$.
Complementary to this, the following Theorem establishes $p$-integrability for its reciprocal.

\begin{thm}\label{reciprocalestimates}
     Let $f$ satisfy Hypothesis \ref{hypothesis1} or Hypothesis \ref{hypothesis2}.
     If $f$ satisfies Hypothesis \ref{hypothesis2} recall the constants $\kappa$ and $\alpha$. 
     Let $\eta \in I$, $\theta \in (1-\alpha,+\infty)$ and let $(\rho_1,\rho_2)$ denote a solution of System \eqref{eq:vcdid} with $\rho_{1,0},\rho_{2,0} \in \sP(\Omega)$ and $\sigma_0^{-\theta} \in L^1(\Omega)$. 
     Then, there exists $\caK_{2,\theta} > 0$ whose size depends only on $\alpha,\kappa, \theta, |\Omega|,T$, the norm 
     \[
     \|\sigma_0^{-\theta}\|_{L^1(\Omega)},
     \]
     and the Lipschitz constants of the potentials such that 
     \begin{equation}\label{eq:thetamuineq2}
     \left\|\sigma^{-\theta}\right\|_{L^\infty([0,T];L^1(\Omega))} + \left\|\partial_x\sigma^\f{\theta+\alpha -1}{2}\right\|_{L^2([0,T]\times \Omega)}  < \caK_{2,\theta} 
     \end{equation}
     In particular, the constant $\caK_{2,\theta} $ is independent of $\eta$.
\end{thm}
\begin{proof}
We present the argument for $f$ satisfying Hypothesis \ref{hypothesis1} with a power law or for $f$ satisfying Hypothesis \ref{hypothesis2}.
A similar argument may be made in the case of the logarithmic pressure.

\medskip 

Recall from Inequality \eqref{eq:sumedi} that, for any $p \in \R\setminus \{1,0\}$, a solution of System \eqref{eq:rsigmaevolution} satisfies the following dissipation inequality for each $t \in (0,T)$
\begin{equation}\label{eq:sumedi2}
\begin{split}
      \f{1}{p(p-1)} \partial_t \int_{\Omega}  \sigma^p \diff x & + \frac{2}{\alpha\kappa(p+\alpha-1)^2}\int_{\Omega} \left|\partial_x\sigma^\frac{p+\alpha-1}{2}\right|^2 \diff x \\
     & \leqslant \frac{\alpha\kappa}{2}c_1 \int_{\Omega} \sigma^{p+1-\alpha} \diff x
\end{split}
\end{equation}
with 
\[
\sum_{i=1}^2 \|\partial_xV_i\|_{L^\infty(\Omega)} \eqqcolon c_1.
\]
If $ p < \alpha - 1 $ then $p+1-\alpha< 0$ and so it is possible to apply a Jensen inequality to the integral in the second line of Inequality \eqref{eq:sumedi2} which satisfies
\begin{equation}\label{eq:apjensen}
    \int_{\Omega} \sigma^{p+1-\alpha} \diff x \leqslant |\Omega|^\frac{\alpha-1}{p} \left(\int_{\Omega} \sigma^p \diff x\right)^\frac{p+1-\alpha}{p}\leqslant |\Omega|^\frac{\alpha-1}{p} \left(1 + \int_{\Omega} \sigma^p \diff x\right).
\end{equation}
By using Inequality \eqref{eq:apjensen} as an upper bound for Inequality \eqref{eq:sumedi2}, it follows that 
\begin{equation}\label{eq:sumedi3}
\begin{split}
      \f{1}{p(p-1)} \partial_t \int_{\Omega}  \sigma^p \diff x & + \frac{2}{\alpha\kappa(p+\alpha-1)^2}\int_{\Omega} \left|\partial_x\sigma^\frac{p+\alpha-1}{2}\right|^2 \diff x \\
     & \leqslant \frac{\alpha\kappa}{2}c_v |\Omega|^\frac{\alpha-1}{p} \left(1 + \int_{\Omega} \sigma^p \diff x\right).
\end{split}
\end{equation}
To conclude the argument, consider the case $p = -\theta$, let $c_2$ denote the constant
\[
c_2 \coloneqq \alpha \kappa\frac{p(p-1)}{2}|\Omega|^\frac{\alpha-1}{p}
\]
and apply a Gr\"onwall argument to Inequality \eqref{eq:sumedi3}.
This establishes the claim that
\begin{equation}\label{eq:recipgronwall}
    1+ \int_{\Omega}  \sigma(t)^{-\theta} \diff x \leqslant \exp(c_2t)\left(1+ \int_{\Omega} \sigma(0)^{-\theta} \diff x\right) 
\end{equation}
for each $t \in [0,T]$. 
Hence, the norm
\[
\left\|\sigma^{-\theta}\right\|_{L^\infty([0,T];L^1(\Omega))} 
\]
is bounded by a constant that depends only on $\alpha, \kappa, |\Omega|, \theta, T$, the Lipschitz constants of the potentials, and the integrability of the initial data. 

\medskip 

To deduce the gradient estimate, it is left to integrate Inequality \eqref{eq:sumedi3} over $[0,T]$ equipped with the estimate for the norm 
\[
\|\sigma^{-\theta}\|_{L^\infty([0,T];L^1(\Omega))}.
\]
\end{proof}

The third aggregate density estimate we require concerns the pressure. 
Indeed, since the theory accommodates a class of $f$ which may not behave precisely like a power-law, it is not immediately clear that the quantity
\[
\|\partial_xf(\sigma) \sqrt{\sigma}\|_{L^2([0,T]\times\Omega)}
\]
can be controlled by the previously established estimates of Theorems \ref{sumestimates} and \ref{reciprocalestimates}.
Instead, it is necessary to utilise the gradient flow structure of the system to establish the integrability of the Fisher information.
In particular, the energy associated to the viscous system is defined as follows.
\begin{definition}\label{gfenergy}
    For each $\eta \in I$, let $\mathscr{F}_\eta\co\sP^{ac}(\Omega)\times\sP^{ac}(\Omega)\to \R$ denote the energy given by
    \begin{equation}\label{eq:gfenergy}
    \begin{split}
    \mathscr{F}_\eta[\rho_{1},\rho_{2}] \coloneqq  \int_{\Omega} f(\rho_1+\rho_2) & + \eta\rho_1\log(\rho_1) +  \eta\rho_2\log(\rho_2)\diff x \\
    & + \int_{\Omega} V_1 \rho_1 + V_2\rho_2 \diff x.
    \end{split}
    \end{equation}
\end{definition}
\begin{proposition}\label{gfedi}
    Let $\eta \in I$ and let $(\rho_1,\rho_2)$ denote a solution of System \eqref{eq:vcdid}.
    Then, the energy dissipation equality 
    \begin{equation}\label{eq:gfedi}
    \begin{split}
    \partial_t & \mathscr{F}_\eta[\rho_{1}(t),\rho_{2}(t)] \\
    & = -\sum_{i=1}^2 \int_{\Omega}\left|\partial_x\left(f'(\rho_1(t)+\rho_2(t))+ V_i + \eta\log(\rho_i(t))\right)\right|^2\rho_i(t) \diff x.    
    \end{split}
    \end{equation}
    is satisfied for each $t \in (0,T)$.

    \medskip 

    Consequently, since Assumption \ref{I4} is satisfied, there exists a constant $\mathcal{K}_3>0$, which depends only on $\mathcal{K}_{\inf},T$, the Lipschitz constants of the potentials, and 
    \[
    \mathscr{F}_\eta[\rho_{1,0},\rho_{2,0}],
    \] 
    such that 
    \begin{equation}\label{eq:fisher}
    \|\partial_xf(\sigma) \sqrt{\sigma}\|_{L^2([0,T]\times\Omega)}^2 < \mathcal{K}_3.  
    \end{equation}
    In particular, $\mathcal{K}_3$ does not depend on $\eta$.
\end{proposition}
\begin{proof}
For ease of exposition let 
\[
\psi_i^\eta = f'(\rho_1+\rho_2) + \eta(\log(\rho_i)+1) +V_i.
\]
Then, since System \eqref{eq:vcdid} is satisfied, the derivative of the energy satisfies the following equality.
In particular, in the following equality, integration by parts does not produce a boundary term thanks to the no-flux condition.
\begin{align*}
    \partial_t & \mathscr{F}_\eta[\rho_{1}(t),\rho_{2}(t)] \\
    & = \sum_{i=1}^2 \int_{\Omega}\partial_t\rho_i (f'(\rho_1+\rho_2) + \eta(\log(\rho_i)+1) +V_i) \diff x
    \\   & = \sum_{i=1}^2 \int_{\Omega}\partial_x(\rho_i\partial_x\psi_i^\eta) \psi_i^\eta \diff x = -\sum_{i=1}^2 \int_{\Omega}|\partial_x\psi_i^\eta|^2 \rho_i \diff x
\end{align*}
\end{proof}
Subsequently, consider the expansion
   \begin{equation}\label{eq:fisherex}
   \begin{split}
      & \sum_{i=1}^2 \int_{\Omega}|\psi_i^\eta|^2\rho_i \diff x = \int_{\Omega}|\eta\partial_x\log(\rho_i) + \partial_x f'(\sigma)|^2\rho_i  \diff x \\ 
        & +   \sum_{i=1}^2\int_{\Omega} 2\partial_x V_i \cdot (\eta\partial_x \log(\rho_i) + \partial_x f'(\sigma))\rho_i \diff x + \sum_{i=1}^2\int_{\Omega} |\partial_x V_i|^2\rho_i \diff x
   \end{split}
   \end{equation}
   and let
   \[
   c_1 \coloneqq \sum_{i=1}^2\|\partial_x V_i\|_{L^\infty(\Omega)}.
   \]
   Then, by applying Young's inequality to the right hand-side of Equation \eqref{eq:fisherex}, it follows that  
   \begin{equation}\label{eq:fisheryoung}
   \begin{split}
      & \sum_{i=1}^2 \int_{\Omega}|\psi_i^\eta|^2\rho_i \diff x \\
      & \geqslant \frac{1}{2}\sum_{i=1}^2\int_{\Omega}|\eta\partial_x\log(\rho_i) + \partial_x f'(\sigma)|^2\rho_i  \diff x - \sum_{i=1}^2\int_{\Omega} |\partial_x V_i|^2\rho_i \diff x
      \\
      & \geqslant \frac{1}{2}\sum_{i=1}^2\int_{\Omega}|\eta\partial_x\log(\rho_i) + \partial_x f'(\sigma)|^2\rho_i  \diff x - c_1.
   \end{split}
   \end{equation}
   Further, by expanding the square in the third line of Inequality \eqref{eq:fisheryoung}, it follows that 
   \begin{equation}\label{eq:fishersquare}
   \begin{split}
      & \sum_{i=1}^2 \int_{\Omega}|\psi_i^\eta|^2\rho_i \diff x \geqslant \frac{1}{2}\int_{\Omega}|\partial_x f'(\sigma)|^2\sigma  \diff x - c_1.
   \end{split}
   \end{equation}
    Subsequent to Inequality \eqref{eq:fishersquare}, the estimate \eqref{eq:fisher} is derived by Integrating Equality \eqref{eq:gfedi} over the interval $[0,T]$ such that $(\rho_1,\rho_2)$ satisfies the inequality
    \begin{align*}
    \mathscr{F}_\eta[\rho_1(0),\rho_2(0)] & + c_1 T \\
    & = \mathscr{F}_\eta[\rho_1(T),\rho_2(T)] + \sum_{i=1}^2 \int_0^T\int_{\Omega}|\psi_i^\eta|^2 \rho_i \diff x \diff t + c_1 T\\
    & \geqslant \mathscr{F}_\eta[\rho_1(T),\rho_2(T)] + \frac{1}{2}\int_0^T\int_{\Omega}|\partial_x f'(\sigma)|^2\sigma  \diff x \diff t .
    \end{align*}
    Since $f$ satisfies Assumption \ref{I4}, the energy $\mathcal{F}_\eta$ is bounded uniformly from below on $\sP^{ac}(\Omega)\times \sP^{ac}(\Omega)$.
    Indeed, let $c_2,c_3$ denote the infima
    \begin{align*}
    c_2 & \coloneqq \inf_{(\mu_1,\mu_2) \in (\sP(\Omega))^2}\int_{\Omega} \mu_1\log(\mu_1) +  \mu_2\log(\mu_2)\diff x \\
    c_3 & \coloneq \inf_{(\mu_1,\mu_2) \in (\sP(\Omega))^2} \int_{\Omega} V_1 \mu_1 + V_2\mu_2 \diff x.
    \end{align*}
    Then, to conclude, the following inequality is satisfied.
     \begin{align*}
    \mathcal{K}_3 \coloneqq 2(\mathscr{F}_\eta[\rho_1(0),\rho_2(0)] &+ 2\mathcal{K}_{\inf} - c_3 - \eta\min\{c_2,0\} + c_1 T) \\
    & \geqslant \|\partial_x f'(\sigma)\sqrt{\sigma}\|^2_{L^2([0,T]\times\Omega)}.
    \end{align*}

\section{Deriving a BV Estimate}\label{sec:BV}

Equipped with sufficient a priori bounds for the aggregate density, this section works to establish that $(\phi^\eta,b^\eta)$ satisfy Property \ref{B2}.
Then, equipped with Properties \ref{B1}, \ref{B2}, \ref{B3}, the main energy dissipation argument for 
\[
t \mapsto \|\partial_x\phi^\eta(t)\|_{L^1(\Omega)}
\]
is established.

\subsection{Coefficient Bounds}

\begin{proposition}\label{sumrequirements}
    Let $(\rho_1,\rho_2)$ denote a solution of System \eqref{eq:vcdid}.
     Then, there exists $\mathcal{B} \in L^1([0,T])$
     such that the following inequality holds for each $t \in [0,T]$.
    \begin{align}
    &\|\partial_x b^\eta(t)\|_{L^1(\Omega)}\leqslant \mathcal{B}(t)\left(1+ \|\partial_x \phi_\eta(t)\|_{L^1(\Omega)}\right)\label{eq:Breq}.
\end{align} 
In particular, the quantity
\[
\|\mathcal{B}\|_{L^1([0,T])}
\]
depends only on: 
\begin{itemize}
    \item The quantities $\kappa$ and $\alpha$ which were introduced in Assumption \ref{I2}.
    \item The potential norms:
\end{itemize}
 \[
\|V_1\|_{W^{2,1}(\Omega)}, \  \|V_2\|_{W^{2,1}(\Omega)} \text{ and } \|V_1-V_2\|_{W^{3,1}(\Omega)};
\]
\begin{itemize}
     \item and lastly, the following norms:
     \begin{multicols}{2}  
        \item $\|\partial_x f'(\sigma) \sqrt{\sigma}\|_{L^2([0,T]\times\Omega)}$,
     \item $\|\sigma^{1-\alpha}\|_{L^1([0,T];W^{1,1}(\Omega))} $,
    \item $\|\partial_x \log(\sigma)\|_{L^2([0,T]\times\Omega)}$,
        \item $\|\sigma^{1-2\alpha}\|_{L^1([0,T]\times\Omega)}$,
    \end{multicols}
    \end{itemize}
    \begin{itemize}
    \begin{multicols}{2}
      \item $\eta\|\sigma^{2-2\alpha}\|_{L^1([0,T];W^{1,1}(\Omega))} $
    \item $\eta\|\partial_x \sigma^{\frac{1-\alpha}{2}}\|_{L^2([0,T]\times \Omega)} $,
    \newcolumn
     \item $\eta^2\|\partial_x \sigma^{3-3\alpha}\|_{L^1([0,T]\times \Omega)} $
     \item $\eta^2\|\partial_x \sigma^{1-\alpha}\|_{L^2([0,T]\times \Omega)} $,
     \item  $\eta^3\|\partial_x \sigma^{\frac{3-3\alpha}{2}}\|_{L^2([0,T]\times \Omega)} $.
     \end{multicols}
     \end{itemize}
\end{proposition}
\begin{proof}
Recall from Definition \ref{coefficients} that $m,m^\eta,b^\eta$ satisfy the following equalities.
 \begin{align*}
      m & = \partial_x(V_1-V_2)(\xi_\eta(\sigma)\bs{v}_r- \sigma \xi_\eta'(\sigma)\bs{v}_\sigma) + \partial_{xx}(V_1-V_2)\\
       & + \int_0^x \partial_s(\xi'_\eta(\sigma)\partial_s(V_1-V_2))(\sigma\partial_s f'(\sigma) + \sigma\bs{v}_\sigma)\diff s,\\
    m^\eta & = m + \eta(\partial_{xx}(V_1-V_2)\xi_\eta(\sigma) + (1-\phi)\left(\partial_x(V_1-V_2)\xi_\eta(\sigma)\right)^2),\\  
    & +  \eta\int_0^x \partial_s(\xi'_\eta(\sigma)\partial_s(V_1-V_2))\partial_s\sigma \diff s \\
       b^\eta & = m^\eta \phi^\eta(1-\phi^\eta).
\end{align*}
Consequent to the complicated structure of the coefficient functions, this proof is split into three steps. 
First, a bound is established for
\[
\|\partial_x m(t)\|_{L^1(\Omega)}.
\]
The next step involves bounding 
\[
\|\partial_x(m^\eta(t)- m(t))\|_{L^1(\Omega)}
\]
and finally, 
a bound is derived for 
\[
\|m^\eta(t)\|_{L^\infty(\Omega)}.
\]
Indeed, differentiating the equation for $m$ yields the equality 
\begin{align*}
    \partial_x m & = \partial_{xxx}(V_1-V_2)\\
    & + \partial_{xx}(V_1-V_2)(\underbrace{\xi_\eta(\sigma)\bs{v}_r}_{\eqqcolon \mathcal{I}_1}+ \underbrace{\sigma \xi'_\eta(\sigma)\partial_xf'(\sigma)}_{\eqqcolon \mathcal{I}_2})\\
    & + \partial_x(V_1-V_2)(\underbrace{\xi_{\eta}(\sigma)\partial_x\bs{v}_r}_{\eqqcolon \mathcal{I}_3}  -  \underbrace{\sigma \xi'_{\eta}(\sigma)\partial_x\bs{v}_\sigma}_{\eqqcolon \mathcal{I}_4}) \\
    & + \partial_x(V_1-V_2)(\underbrace{\partial_x\xi_\eta(\sigma)(\bs{v}_r-\bs{v}_\sigma)}_{\eqqcolon \mathcal{I}_5}) \\
    & + \partial_x(V_1-V_2) 
    (\underbrace{\partial_x\xi'_\eta(\sigma)(\sigma \partial_x f'(\sigma))}_{\eqqcolon \mathcal{I}_6}
\end{align*}
Recall the embedding $W^{k+1,1}(\Omega)\hookrightarrow W^{k,\infty}(\Omega)$.
Thanks to this embedding, the quantities 
\[
\|\partial_{xx}(V_1-V_2)\|_{L^\infty(\Omega)}, \quad \quad  \|\partial_{x}(V_1-V_2)\|_{L^\infty(\Omega)}
\]
may be bounded by a constant, which depends only on $\Omega$, multiplied by
\[
\|V_1-V_2\|_{W^{3,1}(\Omega)}.
\]
This estimate ensures that the drift factors multiplying each of the terms $\mathcal{I}_1,\mathcal{I}_2,\dots, \mathcal{I}_6$ are bounded in $L\infty(\Omega)$.
Then, to obtain the desired bound, the terms $\mathcal{I}_1,\mathcal{I}_2,\dots, \mathcal{I}_6$ are treated in sequence.
To begin the process, recall that the quantities 
\[
\|\bs{v}_r\|_{L^\infty([0,T]\times\Omega)} \text{ and }\|\bs{v}_\sigma\|_{L^\infty([0,T]\times\Omega)}
\]
depend only on the size of the constant 
\[
c_1 \coloneqq \sum_{i=1}^2 \|\partial_x V_i\|_{L^\infty(\Omega)}
\]
As such, $\mathcal{I}_1$ satisfies the inequality
\begin{equation}\label{eq:i1pt1}
 \|\mathcal{I}_1(t)\|_{L^1(\Omega)} \leqslant c_1\|\xi_\eta(\sigma(t))\|_{L^1(\Omega)}.
\end{equation}
Further, recall the zero and first order estimates for the non-linearity $\xi_\eta$ which, thanks to Assumption \ref{I2}, satisfies 
\begin{equation}\label{eq:recallxi1}
            |\xi_\eta(s)| \leqslant \kappa s^{1-\alpha}, \quad 
            |\xi'_\eta(s)| \leqslant (1+\alpha)\kappa s^{-\alpha} + 2\eta\alpha \kappa^2 s^{1-2\alpha}.  
\end{equation}
By using the zero order inequality to bound Equation \eqref{eq:i1pt1} from above, it follows that 
\begin{equation}\label{eq:i1pt2}
 \|\mathcal{I}_1(t)\|_{L^1(\Omega)} \leqslant \kappa c_1\|\sigma(t)^{1-\alpha}\|_{L^1(\Omega)}.
\end{equation}
To bound $\mathcal{I}_2$, recall Inequality \eqref{ord:one}:
\[
 |\xi'_\eta(s)|\leqslant  \frac{|\xi(s)|}{s} + \frac{1}{s^2f''(s)} + \frac{2\eta|\xi(s)|}{s^2f''(s)}.
\]
Subsequently, $\mathcal{I}_2$ satisfies 
\begin{equation}\label{eq:i2pt1}
\begin{split}
    \|\mathcal{I}_2(t)\|_{L^1(\Omega)} & \leqslant \|\xi(\sigma(t))\partial_x f'(\sigma(t))\|_{L^1(\Omega)} + \|\partial_x \log (\sigma(t))\|_{L^1(\Omega)} \\
    & + \eta \|\sigma(t)^{-1}\xi(\sigma(t))\partial_x \sigma(t)\|_{L^1(\Omega)}
\end{split}
\end{equation}
To address the first term on the right hand-side of Inequality \eqref{eq:i2pt1}, Young's inequality is used. 
Then, the zero order Inequality of \eqref{eq:recallxi1} is used to bound the remaining dependence on $\xi$ by a power law.
In particular, 
$\mathcal{I}_2$ satisfies 
\begin{equation}\label{eq:i2pt2}
\begin{split}
    \|& \mathcal{I}_2(t)\|_{L^1(\Omega)} \leqslant \|\sqrt{\sigma(t)}\partial_x f'(\sigma(t))\|_{L^2(\Omega)}^2 + \left\|\frac{\xi(\sigma(t))}{\sqrt{\sigma(t)}}\right\|_{L^2(\Omega)}^2 \\ & + \|\partial_x \log (\sigma(t))\|_{L^1(\Omega)} + \eta \|\sigma(t)^{-1}\xi(\sigma(t))\partial_x \sigma(t)\|_{L^1(\Omega)}\\
    & \leqslant \|\sqrt{\sigma(t)}\partial_x f'(\sigma(t))\|_{L^2(\Omega)}^2 + \kappa^2 \left\|\sigma(t)^{1-2\alpha}\right\|_{L^1(\Omega)} \\ 
    & + \sqrt{|\Omega|}\|\partial_x \log (\sigma(t))\|_{L^2(\Omega)} + \eta\frac{\kappa}{1-\alpha} \|\partial_x \sigma(t)^{1-\alpha}\|_{L^1(\Omega)}
\end{split}
\end{equation}
Using the zero and first order inequalities of \eqref{eq:recallxi1}, there exists $c_2 \coloneqq c_2(\alpha,\kappa) > 0$, depending only on $\kappa,\alpha$, for which the terms $\mathcal{I}_3$ and $\mathcal{I}_4$ satisfy the respective inequalities 
\begin{equation}\label{eq:i3pt1}
\begin{split}
 \|\mathcal{I}_3(t)\|_{L^1(\Omega)} \leqslant c_2 \|\sigma(t)^{1-\alpha}\|_{L^\infty(\Omega)}\|\partial_x \bs{v}_r(t)\|_{L^1(\Omega)}.
\end{split}
\end{equation}
and 
\begin{equation}\label{eq:i4pt1}
\begin{split}
 & \|\mathcal{I}_4(t)\|_{L^1(\Omega)} \\
 & \leqslant c_2\left(\|\sigma(t)^{1-\alpha}\|_{L^\infty(\Omega)} + \eta \|\sigma(t)^{2-2\alpha}\|_{L^\infty(\Omega)}\right) \|\partial_x \bs{v}_\sigma(t)\|_{L^1(\Omega)}.
\end{split}
\end{equation}
Furthermore, recalling that 
 \begin{align*}
        \bs{v}_\sigma &\coloneqq r\partial_xV_1+(1-r)\partial_xV_2, \\
        \bs{v}_r &\coloneqq r\partial_xV_2 + (1-r)\partial_xV_1, 
\end{align*}
there exists a constant $c_3 \coloneqq c_3(V_1,V_2) > 0$, whose size depends only on 
$\|\partial_{x}V_1\|_{W^{1,1}(\Omega)}$ and $\|\partial_{x}V_2\|_{W^{1,1}(\Omega)}$, for which 
\begin{align}
     \|\partial_x\bs{v}_\sigma(t)\|_{L^1(\Omega)} &\leqslant c_3\left( 1+ \| \partial_x r(t)\|_{L^1(\Omega)}\right), \label{eq:dxvsigma} \\
     \|\partial_x\bs{v}_r(t)\|_{L^1(\Omega)} &\leqslant c_3\left( 1+ \| \partial_x r(t)\|_{L^1(\Omega)}\right). 
     \label{eq:dxvr}
\end{align}
Subsequently, utilising Lemma \ref{phir}, the right hand-side of Inequalities \eqref{eq:dxvsigma} and \eqref{eq:dxvr} may be bounded as follows
\begin{equation}\label{eq:dxvrsigma}
    \|\partial_x\bs{v}_\sigma(t)\|_{L^1(\Omega)},    \|\partial_x\bs{v}_r(t)\|_{L^1(\Omega)} \leqslant \Lambda_2c_3\left( 1+ \| \partial_x \phi^\eta(t)\|_{L^1(\Omega)}\right).
\end{equation}
Then, in summing Inequalities \eqref{eq:i3pt1} and \eqref{eq:i4pt1} and using Inequality \eqref{eq:dxvrsigma} as an upper bound, the quantities $\mathcal{I}_3$ and $\mathcal{I}_4$ satisfy
\begin{equation}\label{eq:i34pt2}
\begin{split}
  \|\mathcal{I}_3(t)\|_{L^1(\Omega)} & +  \|\mathcal{I}_4(t)\|_{L^1(\Omega)} \\
  & \leqslant 2c_2c_3\Lambda_2\|\sigma(t)^{1-\alpha}\|_{L^\infty(\Omega)}\left( 1+ \| \partial_x \phi^\eta(t)\|_{L^1(\Omega)}\right)\\
  &  + \eta2c_2c_3\Lambda_2 \|\sigma(t)^{2-2\alpha}\|_{L^\infty(\Omega)}\left( 1+ \| \partial_x \phi^\eta(t)\|_{L^1(\Omega)}\right).
\end{split}
\end{equation}
Lastly concerning $\mathcal{I}_3,\mathcal{I}_4$, thanks to the embedding $W^{1,1}(\Omega)\hookrightarrow L^\infty(\Omega)$ there exists $c_4\coloneqq c_4(\Omega) > 0 $ depending only on $|\Omega|$ for which, when letting $c_5$ denote the constant $c_5 \coloneqq 2c_2c_3c_4\Lambda_2$, the following inequality is satisfied
\begin{equation}\label{eq:i34pt3}
\begin{split}
  & \|\mathcal{I}_3(t)\|_{L^1(\Omega)} +  \|\mathcal{I}_4(t)\|_{L^1(\Omega)} \\
  & \leqslant c_5\left(\|\sigma(t)^{1-\alpha}\|_{W^{1,1}(\Omega)}+ \eta \|\sigma(t)^{2-2\alpha}\|_{W^{1,1}(\Omega)}\right)\left( 1+ \| \partial_x \phi^\eta(t)\|_{L^1(\Omega)}\right).
\end{split}
\end{equation}
Next, thanks to the first order Inequality in \eqref{eq:recallxi1}, there exists $c_6\coloneqq c_6(\alpha,\kappa) > 0$, depending only on $\kappa,\alpha$, such that $\mathcal{I}_5$ satisfies 
\begin{equation}\label{eq:i5pt1}
    \|\mathcal{I}_5(t)\|_{L^1(\Omega)} \leqslant c_1c_6\left(\|\partial_x\sigma(t)^{1-\alpha}\|_{L^1(\Omega)}+ \eta \|\partial_x\sigma(t)^{2-2\alpha}\|_{L^1(\Omega)} \right).
\end{equation}
To bound $\mathcal{I}_6$ recall the second order inequality 
\begin{align}
            |\xi_\eta''(s) f''(s) s^2| \leqslant (\kappa+1)s^{-1} + 2\eta \kappa(2+ \alpha + \kappa)s^{-\alpha} + 4\alpha\eta^2 \kappa^2s^{1-2\alpha}.\label{eq:recallxi2}
\end{align}
Then, there exists there exist $c_7 \coloneqq c_7(\alpha,\kappa) > 0$, depending only on $\kappa,\alpha$, such that $\mathcal{I}_6$ satisfies the inequality 
\begin{equation}\label{eq:i6pt1}
\begin{split}
  & \|\mathcal{I}_6(t) \|_{L^1(\Omega)} =
    \|(\partial_x\sigma(t))^2 \sigma f''(\sigma(t)) \xi_\eta''(\sigma(t))\|_{L^1(\Omega)} 
    \\
    & \leqslant c_7\|\partial_x\log(\sigma(t))\|_{L^2([0,T]\times\Omega)}^2\\
    & + c_7\left(\eta \|\partial_x\sigma(t)^{\frac{1-\alpha}{2}}\|^2_{L^2([0,T]\times\Omega)}+ \eta^2 \|\partial_x\sigma(t)^{1-\alpha}\|_{L^2([0,T]\times\Omega)}^2\right).  
\end{split}
\end{equation}
Having now bounded the quantities $\mathcal{I}_1,\mathcal{I}_2,\dots, \mathcal{I}_6$ as presented above, the bounds are collected, aggregating the dependence of the constants $c_1,\dots,c_7$ and $\Lambda_2$. 
In particular, there exists a constant $C_1$, whose size depends only on $\kappa,\alpha,\Omega$ and the norms 
\[
\|V_1-V_2\|_{W^{3,1}(\Omega)}, \quad \|V_1\|_{W^{2,1}(\Omega)}, \quad \|V_2\|_{W^{2,1}(\Omega)}
\]
such that 
\begin{equation}\label{eq:dxc}
\begin{split}
    & \| \partial_x m(t)\|_{L^1(\Omega)} \leqslant C_1\left(\|\sqrt{\sigma(t)}\partial_x f'(\sigma(t))\|_{L^2(\Omega)}^2 + \left\|\sigma(t)^{1-2\alpha}\right\|_{L^1(\Omega)} \right)\\ 
 & +C_1\left(\|\sigma(t)^{1-\alpha}\|_{W^{1,1}(\Omega)}+ \eta \|\sigma(t)^{2-2\alpha}\|_{W^{1,1}(\Omega)}\right)\left( 1+ \| \partial_x \phi^\eta(t)\|_{L^1(\Omega)}\right)  \\
 & + C_1\left(\|\partial_x\log(\sigma)\|_{L^2([0,T]\times\Omega)}^2 + \eta \|\partial_x\sigma^{\frac{1-\alpha}{2}}\|^2_{L^2([0,T]\times\Omega)}\right) \\
 & +  C_1\left(\eta^2 \|\partial_x\sigma^{1-\alpha}\|_{L^2([0,T]\times\Omega)}^2 + 1\right).
\end{split}
\end{equation}
After considering integration over the time domain $[0,T]$, all of the terms on the right hand-side of Inequality \eqref{eq:dxc} may suitably be bounded by the norms listed in the statement of the proposition. 

\medskip 

The proof now attends to a bound for the derivative of
\[
\frac{m^\eta - m}{\eta}
\]
which satisfies the equality 
\begin{align*}
    \partial_x& \left(\frac{m^\eta - m}{\eta}\right)  = \partial_{xxx}(V_1-V_2)\underbrace{\xi_\eta(\sigma)}_{\eqqcolon \mathcal{N}_1}+ 2\partial_{xx}(V_1-V_2)\underbrace{\partial_x\xi_\eta(\sigma)}_{\eqqcolon \mathcal{N}_2} \\
    & + (\partial_x(\partial_x(V_1-V_2)^2)(1-\phi)-\partial_x(V_1-V_2)^2\partial_x \phi)\underbrace{\xi_\eta^2(\sigma)}_{\eqqcolon \mathcal{N}_3}\\
    & + (1-\phi)\partial_x(V_1-V_2)^2\underbrace{\partial_x(\xi_\eta(\sigma)^2)}_{\eqqcolon \mathcal{N}_4}\\
    & + \partial_x(V_1-V_2)\underbrace{\partial_x\xi_\eta'(\sigma)\partial_x\sigma}_{\eqqcolon \mathcal{N}_5}.
\end{align*}
To obtain the desired bound, the terms $\mathcal{N}_1,\mathcal{N}_2,\dots, \mathcal{N}_5$ are once more treated in sequence.
First, recall from the previous calculations concerning $\mathcal{I}_1$ and $\mathcal{I}_5$ that 
\begin{align}
\|\mathcal{N}_1(t)\|_{L^1(\Omega)} & \leqslant \kappa \|\sigma(t)^{1-\alpha}\|_{L^1(\Omega)} \label{eq:j1pt1},\\
\|\mathcal{N}_2(t)\|_{L^1(\Omega)} & \leqslant c_6\left(\|\partial_x\sigma(t)^{1-\alpha}\|_{L^1(\Omega)}  + \eta \|\partial_x\sigma(t)^{2-2\alpha}\|_{L^1(\Omega)} \right).\label{eq:j2pt1}
\end{align}
Further, from the Sobolev embedding $W^{1,1}(\Omega)\hookrightarrow L^\infty(\Omega)$, $\mathcal{N}_3$ and $\mathcal{N}_4$ satisfy
\eqref{eq:recallxi1}, 
\begin{align*}
\|\mathcal{N}_3(t)\|_{L^\infty(\Omega)} & \leqslant c_4\left(\|\mathcal{N}_4(t)\|_{L^1(\Omega)} + \|\xi_\eta(\sigma(t))^2\|_{L^1(\Omega)}\right) \\
& \leqslant  c_4\left(\|\mathcal{N}_4(t)\|_{L^1(\Omega)} + \kappa^2\|\sigma(t)^{2-2\alpha}\|_{L^1(\Omega)}\right).
\end{align*}
Then, expanding the derivative in $\mathcal{N}_4$, and using the zero and first inequalities in \eqref{eq:recallxi1} there exists $c_8 \coloneqq c_8(\alpha,\kappa)> 0$, depending only on $\alpha,\kappa$, for which the following inequality is satisfied.
\begin{equation}\label{eq:j4pt1}
    \begin{split}
        \|&\mathcal{N}_4(t)\|_{L^1(\Omega)} \leqslant 2\|\partial_x\sigma(t) \xi_\eta'(\sigma(t))\xi_\eta(\sigma(t))\|_{L^1(\Omega)}\\
        & \leqslant c_{8}\left( \|\partial_x\sigma(t)^{2-2\alpha}\|_{L^1(\Omega)} + \eta \|\partial_x\sigma(t)^{3-3\alpha}\|_{L^1(\Omega)}\right).
    \end{split}
\end{equation}
Lastly, the quantity $\mathcal{N}_5$ is bounded using Inequality \eqref{eq:recallxi2} in conjunction with Assumption \ref{I2} which asserts that 
\[
\frac{1}{\alpha \kappa} s^\alpha \leqslant s^2 f''(s).
\]
In particular, there exists $c_9\coloneqq c_9(\alpha,\kappa)$, depending only on $\kappa,\alpha$, for which $\xi''$ satisfies the inequality
\[
\left|\xi_\eta''(s)\right| \leqslant c_9\left(s^{-1-\alpha} + s^{-2\alpha} +s^{1-3\alpha}\right).
\]
Hence, there further exists $c_{10}\coloneqq c_{10}(\alpha,\kappa)$, depending only on $\kappa$ and $\alpha$, for which $\mathcal{N}_5$ satisfies the inequality
\begin{equation}
\begin{split}
     \|\mathcal{N}_5(t)\|_{L^1(\Omega)} & \leqslant c_{10}\|\partial_x \sigma(t) ^\frac{1-\alpha}{2}\|_{L^2(\Omega)}^2 \\ 
    & + c_{10}\left(\eta \|\partial_x \sigma(t) ^{1-\alpha}\|_{L^2(\Omega)}^2 +  \eta^2 \|\partial_x \sigma(t)^\frac{3-3\alpha}{2}\|_{L^2(\Omega)}^2\right).
    \end{split}
\end{equation}
In collecting the inequalities satisfied by $\mathcal{N}_1,\mathcal{N}_2,\dots, \mathcal{N}_5$ and aggregating the dependence on the constants $c_4,c_6,c_8,c_{10}$, it follows that there exists a constant $C_2$, whose size depends only on $\kappa,\alpha,\Omega$ and the norm
\[
\|V_1-V_2\|_{W^{3,1}(\Omega)}
\]
such that the following equality is satisfied.
\begin{equation}\label{eq:dxceta}
\begin{split}
    & \left\| \partial_x \left(m^\eta - m\right)\right\|_{L^1(\Omega)} \leqslant \eta C_2\left(\|\sigma(t)^{1-\alpha}\|_{W^{1,1}(\Omega)} +  \|\partial_x \sigma(t) ^\frac{1-\alpha}{2}\|_{L^2(\Omega)}^2\right)\\ 
    & + \eta^2 C_2\left(\|\partial_x \sigma(t) ^{1-\alpha}\|_{L^2(\Omega)}^2 +  \eta \|\partial_x \sigma(t)^\frac{3-3\alpha}{2}\|_{L^2(\Omega)}^2\right)\\
    & + \eta C_2\|\sigma(t)^{2-2\alpha}\|_{W^{1,1}(\Omega)} \left( 1+ \| \partial_x \phi^\eta(t)\|_{L^1(\Omega)}\right)\\
    & + \eta ^2C_2 \|\partial_x \sigma(t)^{3-3\alpha}\|_{L^1(\Omega)}  \left( 1+ \| \partial_x \phi^\eta(t)\|_{L^1(\Omega)}\right).
\end{split}
\end{equation}
After considering integration over the time domain $[0,T]$, all of the terms on the right hand-side of Inequality \eqref{eq:dxceta} may suitably be bounded by the norms listed in the statement of proposition. 
As such, the proof now turns it's attention to bounding $m^\eta$ in $L^\infty(\Omega)$.

\medskip 

To avoid repeating steps of the proof, it is first acknowledged that each of the integral terms which feature in Equalities \eqref{eq:Cf} and \eqref{eq:Cfeta} has already been shown, in the bounding of $m^\eta$, to have a spatial derivative which is suitably bounded in $L^1([0,T];L^1(\Omega))$.

\medskip 

Further, any function $\varphi \in W^{1,1}(\Omega)$ which vanishes at $x = 0$ satisfies the inequality 
\begin{equation}\label{eq:acineq}
  \sup_{x\in \Omega}|\varphi(x)| \leqslant \sup_{x\in \Omega} \left(\int_0^x |\varphi'(s)| \diff s \right) = \|\varphi'\|_{L^1(\Omega)}.
\end{equation}
Since each integral term vanishes at $x= 0$, it follows from the $L^1([0,T];L^1(\Omega))$ bound obtained for the derivative, used in conjunction with Inequality \eqref{eq:acineq}
that each integral term is appropriately bounded in $L^1([0,T];L^\infty(\Omega))$.
Inspecting the formula for $m^\eta$, it is then left to establish an $L^1([0,T];L^\infty(\Omega))$ control for the following two quantities.
\begin{align*}
    \mathcal{Q}_1 & \coloneqq \partial_x(V_1-V_2)(\xi_\eta(\sigma)\bs{v}_r- \sigma \xi_\eta'(\sigma)\bs{v}_\sigma) + \partial_{xx}(V_1-V_2) \\
    \mathcal{Q}_2 & \coloneqq \eta(\partial_{xx}(V_1-V_2)\xi_\eta(\sigma) + (1-\phi)\left(\partial_x(V_1-V_2)\xi_\eta(\sigma)\right)^2).
\end{align*}
Thanks to the embedding $W^{1,1}(\Omega)\hookrightarrow L^\infty(\Omega)$ there exists a constant $c_{11} \coloneq c_{11}(\Omega,V_1,V_2)$, which depends only on
$|\Omega|$ and the norms 
\[
\|V_1-V_2\|_{W^{3,1}(\Omega)}, \quad \|V_1\|_{W^{2,1}(\Omega)}, \quad \|V_2\|_{W^{2,1}(\Omega)}
\]
such that the quantities $\mathcal{Q}_1,\mathcal{Q}_2$ satisfy
\begin{align*}
& \|\mathcal{Q}_1(t)\|_{L^\infty(\Omega)} + \|\mathcal{Q}_2(t)\|_{L^\infty(\Omega)} \\
& \leqslant c_{11}\left(1 +  \|\sigma (t)\xi'_\eta(\sigma(t))\|_{L^\infty(\Omega)}\right)\\
& + c_{11}\left( \|(\xi_\eta(\sigma(t)))^2\|_{W^{1,1}(\Omega)} +\| \xi_\eta(\sigma(t))\|_{W^{1,1}(\Omega)}\right)\\
& =  c_{11}\left(1 + \|\sigma(t) \xi'_\eta(\sigma(t))\|_{L^\infty(\Omega)} + \sum_{i=1}^4\|\mathcal{N}_i(t)\|_{L^1(\Omega)}\right).
\end{align*}
Recalling Inequalities \eqref{eq:j1pt1}, \eqref{eq:j2pt1} and \eqref{eq:j4pt1}, it has previously been shown that each of the terms $\|\mathcal{N}_i(t)\|_{L^1(\Omega)}, i \in \{1,2,3,4\}$ are bounded appropriately -- that is, in an $L^1([0,T])$ sense and by the quantities listed in the statement of the proposition. 
Hence, it is left to derive a bound for
\[
\|\sigma(t)\xi'_\eta(\sigma(t))\|_{L^\infty(\Omega)} 
\]
which thanks to the first order inequality \eqref{eq:recallxi1} admits a constant $c_{12}\coloneqq c_{12}(\alpha,\kappa)$, depending only on $\kappa,\alpha$, for which 
\begin{align*}
  & \|\sigma(t)\xi'_\eta(\sigma(t))\|_{L^\infty(\Omega)} \\
& \leqslant 
c_{12}\left(\|\sigma(t)^{1-\alpha}\|_{L^\infty(\Omega)}  + \eta \|\sigma(t)^{2-2\alpha}\|_{L^\infty(\Omega)} \right)\\
& \leqslant c_4 c_{12}\left(\|\sigma(t)^{1-\alpha}\|_{W^{1,1}(\Omega)}  + \eta \|\sigma(t)^{2-2\alpha}\|_{W^{1,1}(\Omega)}\right).  
\end{align*}
To conclude the proof, recognise that $\partial_x b^\eta $ satisfies the inequality 
\[
\|\partial_x b^\eta (t)\|_{L^1(\Omega)} \leqslant \|\partial_x m^\eta (t)\|_{L^1(\Omega)} +\|m^\eta (t)\|_{L^\infty(\Omega)} \|\partial_x \phi (t)\|_{L^1(\Omega)}.
\]
Then, in collecting Inequalities \eqref{eq:dxc} and \eqref{eq:dxceta}, it follows that 
\[
\|\partial_x m^\eta (t)\|_{L^1(\Omega)} \leqslant  \widetilde{\mathcal{B}}(t)\left(1+ \|\partial_x \phi_\eta(t)\|_{L^1(\Omega)}\right)
\]
where, in particular, the size of 
\[
\|\widetilde{\mathcal{B}}\|_{L^1([0,T])}
\]
depends only on the size of the quantities and norms which were listed in the statement of the proposition.

\medskip 

Similarly, it follows from the above discussion concerning the boundedness of $m^\eta$ in $L^\infty(\Omega)$ that 
\[
\|m^\eta (t)\|_{L^\infty(\Omega)} \leqslant  \widehat{\mathcal{B}}(t)
\]
where, in particular, the size of the quantity
\[
\|\widehat{\mathcal{B}}\|_{L^1([0,T])}
\]
depends only on the size of the quantities and norms which were listed in the statement of the proposition.
By letting $\mathcal{B}(t) \coloneqq \max\{\widehat{\mathcal{B}}(t), \widetilde{\mathcal{B}}(t)\}$ the proof is concluded.
\end{proof}

\begin{remark}\label{remark:bounds}
    In the case that $f$ satisfies Hypothesis \ref{hypothesis1} with a power law structure the dependence of $\|\mathcal{B}\|_{L^1([0,T])}$ on the quantity 
    \[
    \|\sigma^{1-2\alpha}\|_{L^1([0,T]\times\Omega)}
    \]
    may be removed. 
    Indeed, if $f$ satisfies Hypothesis \ref{hypothesis1} then there exists $\hat{\lambda} \in \R$ for which the following equality is satisfied
    \[
    \sigma \xi'(\sigma) \partial_x f'(\sigma) = \hat{\lambda}\partial_x\log(\sigma).
    \]
    This means that the term $\mathcal{I}_2$ may be bounded more sharply, satisfying the inequality 
    \[
    \|\mathcal{I}_2(t)\|_{L^1(\Omega)} \leqslant \|\partial_x\log(\sigma(t))\|_{L^1(\Omega)} + \mathcal{O}(\eta).
    \]
    On the other hand, if $f$ coincides with the logarithmic pressure law, then the bound of $\mathcal{B}$ is asymptotically independent of any bound on $\sigma$ since $\xi$ is a constant.
\end{remark}
\subsection{The Energy Dissipation Inequality}
Having already made the calculations with which to establish sufficient coefficient bounds for $b^\eta$, the subsequent energy dissipation argument with which to establish $L^\infty([0,T];BV(\Omega))$ estimates for $r$.

  \begin{thm}\label{fConcludingGronwall}
    Let $\eta \in I$ and let $(\rho_1,\rho_2)$ denote a classical solution of System \eqref{eq:vcdid}. 
    Then there exists $\mathcal{K}_4 > 0$
    for which the following inequality holds.
    \[
   \|\partial_x r\|_{L^\infty([0,T];L^1(\Omega))} < \mathcal{K}_4.
    \]
    In particular, the size of $\mathcal{K}_4$ depends only on 
    whose size depends only on $T$ in addition to:
    \begin{itemize}
        \item the size of $\|\partial_xr_0\|_{L^1(\Omega)}$;
        \item the size of $\|\mathcal{B}\|_{L^1([0,T])}$, introduced in Proposition \ref{sumrequirements};
        \item the Lipschitz constant of $V_1-V_2$.
    \end{itemize} 
\end{thm}
\begin{proof}
    Since $\partial_x \phi^\eta \in C^1([0,T];C(\overline{\Omega}))$, it follows from (see Proposition \ref{energypropn}) that the map
    \[
    t\mapsto \int_{\Omega} |\partial_x \phi^\eta|(t) \diff x
    \]
    is differentiable for almost every $t \in [0,T]$ and, for almost every such $t \in [0,T]$ the interchange of integration and differentiation is justified.
    Subsequently, the following inequality holds for almost every $t \in [0,T]$.
    \begin{equation}\label{eq:phiedi}
    \begin{split}
       \partial_t & \int_{\Omega} |\partial_x \phi^\eta| \diff x = \int_{\Omega}   \sgn(\partial_x \phi^\eta)\partial_x \partial_t \phi^\eta\diff x \\
       = & \int_{\Omega} \sgn(\partial_x \phi^\eta)\partial_{xxx}\phi^\eta + \partial_x (a^\eta|\partial_x\phi^\eta|) + \sgn(\partial_x \phi^\eta)\partial_x b^\eta \diff x. \\
    \end{split}
    \end{equation}

     The first summand featured in the second line of Equality \eqref{eq:phiedi} is addressed by applying Kato's inequality (see \cite[Lemma A]{Kato1972} and \cite{brezisKato}). 
    In particular, recall from Proposition \ref{phievolution} that $\phi \in C([0,T];C^3(\overline{\Omega}))$. 
    Since $\partial_{xxx}\phi^\eta \in C([0,T];C(\overline{\Omega}))$, it follows that, for each $t \in [0,T]$ the inequality 
    \begin{equation}\label{eq:Katoapplied}
         \sgn(\partial_x \phi^\eta)\partial_{xxx}\phi^\eta  \leqslant \partial_{xx}(|\partial_x\phi^\eta|)  
    \end{equation}
    is satisfied in the sense of distributions. 
    Moreover, it follows from Inequality \eqref{eq:Katoapplied} that
    \[
    \int_{\Omega} \sgn(\partial_x \phi^\eta)\partial_{xxx}\phi^\eta \diff x \leqslant \int_{\partial\Omega} \partial_x(|\partial_x\phi^\eta|) \cdot \nu \diff x \leqslant 0.
    \]
    In particular, the one-sided derivative $ \partial_x(|\partial_x\phi^\eta|)$ must point in the opposite direction to the outward pointing normal on the boundary and hence the product 
    \[
    \partial_x(|\partial_x\phi^\eta|)\cdot \nu 
    \]
    must achieve a value less than or equal to zero on $\partial\Omega$. 
    This is because the value of $|\partial_x \phi^\eta|$ on $\partial\Omega$ must coincide with the minimum of $|\partial_x \phi^\eta|$ over $\Omega$ due to the boundary condition $\partial_x\phi^\eta = 0$ on $(0,T)\times\partial\Omega$.
    
    \medskip 

    Since the second summand featured in the second line of Equality \eqref{eq:phiedi} is a total derivative, the equality 
    \[
    \int_{\Omega} \partial_x(a|\partial_x\phi^\eta|) \diff x =   \int_{\partial \Omega} a|\partial_x\phi^\eta| \cdot\nu \diff x  = 0
    \]
    is satisfied. 
    Such an equality is recognised by integrating by parts and once more, utilising the boundary condition $\partial_x\phi^\eta = 0$ on $(0,T)\times\partial\Omega$. 

    \medskip 
    
    After addressing the first and second summands in the second line of Equality \eqref{eq:phiedi}, it follows that 
    \begin{equation}\label{eq:phiedi2}
           \partial_t \int_{\Omega} |\partial_x \phi^\eta| \diff x \leqslant \int_{\Omega} |\partial_x b^\eta| \diff x.
    \end{equation}
    To address the final term, recall Inequality \eqref{eq:Breq}.
    By putting together Inequalities \eqref{eq:phiedi2} and \eqref{eq:Breq}, it follows that 
    \begin{equation}\label{eq:fpregronwall}
    \partial_t \int_{\Omega} |\partial_x \phi| \diff x \leqslant \mathcal{B}\left( 1 + \int_{\Omega} |\partial_x \phi| \diff x\right).
    \end{equation}
    
    \medskip
    Moreover, applying Gr\"onwall's Lemma to Inequality \eqref{eq:fpregronwall} establishes a bound for $|\partial_x\phi|$ through the following inequality, which holds for each $t \in [0,T]$.
    \begin{equation}\label{eq:fgronwall}
    \|\partial_x \phi(t) \|_{L^1(\Omega)}+1 \leqslant \exp\left(\int_0^t\mathcal{B}(s) \diff s \right) \left( 1 + \|\partial_x \phi_0\|_{L^1(\Omega)}\right).
    \end{equation}
    To conclude, recall from Lemma \ref{phir} that there exist $\Lambda_1, \Lambda_2> 0 $ for which 
    \begin{align*}
     \|\partial_x\phi_0\|_{L^1(\Omega)} & \leqslant \Lambda_1(\|\partial_x r_0 \|_{L^1(\Omega)}+1), \\\
    \|\partial_x r\|_{L^\infty([0,T];L^1(\Omega))} & \leqslant  \Lambda_2(\|\partial_x\phi \|_{L^\infty([0,T];L^1(\Omega))}+1 ).
    \end{align*}
    By substituting each of the above inequalities within Inequality \eqref{eq:fgronwall}, the control on $|\partial_x \phi|$ is transformed into a control on $|\partial_x r|$ which thus satisfies the following estimate.
     \begin{align*}
        \|\partial_x r& \|_{L^\infty([0,T];L^1(\Omega))}  \\
        & \leqslant \Lambda_2 \left(1 + \Lambda_1\left(1+ \|\partial_x r_0 \|_{L^1(\Omega)}\right) \right)\exp\left(\|\mathcal{B}\|_{L^1([0,T])}\right)\eqqcolon \mathcal{K}_4.
     \end{align*}
    In particular, it was possible to choose the size of the constants $\Lambda_1,\Lambda_2$ to depend only on the Lipschitz constant of $V_1-V_2$.
    \end{proof}
\section{Convergence to the Limiting System}\label{sec:convergence}
Equipped with the fundamental $L^\infty([0,T];BV(\Omega))$ estimate for $r$, along with Sobolev estimates for non-linear functions of $\sigma$, this section establishes the necessary compactness for the approximating system, facilitating the extraction of a subsequence which is shown to converge to a weak solution of System \eqref{eq:cdid}.

\medskip 

Firstly, however, it is proven that:
\begin{itemize}
    \item There exists $\theta_\alpha > 1$ such that $\partial_x \sigma$ is bounded in $L^{\theta_\alpha}_{t,x}$, uniformly in $\eta$.
    \item There exists $s > 0$ such that $\partial_t \rho_i$ is bounded in $L^2([0,T];H^{-s}(\Omega))$, uniformly in $\eta$.
\end{itemize}

\begin{proposition}\label{falpreparation}
        Let $\eta \in I $ and let $(\rho_1,\rho_2)$ denote a classical solution of System \eqref{eq:vcdid} and, if $f$ satisfies Hypothesis \ref{hypothesis2}, recall the constants $\kappa$ and $\alpha$.
        Then, there exists $\mathcal{K}_5 > 0$ and $\theta_\alpha > 1$, whose sizes depend only on $\alpha, \kappa, |\Omega|,T$, and the Lipschitz constants of the potentials, such that 
        \begin{equation}\label{eq:dxsigmabound2}
        \|\partial_x \sigma \|_{L^{\theta_\alpha}([0,T]\times\Omega)} \leqslant \mathcal{K}_5.
        \end{equation}
\end{proposition}
    \begin{proof}
    Recall the gradient estimates established in Theorem \ref{sumestimates}.
    By choosing $\theta = 1$ and $p = \f{\alpha}{2}$ (or $p = \frac{1}{2}$ for the logarithmic pressure) in Theorem \ref{sumestimates}, it follows that there exists a constant $\mathcal{K}_{1,1}$, whose size depends only on $\kappa,\alpha, |\Omega|,T$ and the Lipschitz constants of the potentials, such that 
     \begin{equation}\label{eq:alphaineq}
     \|\sigma^{1-\frac{\alpha}{2}}\|_{L^\infty([0,T];L^1(\Omega))} + \|\partial_x\sigma^\f{\alpha}{4}\|_{L^2([0,T]\times \Omega)} < \mathcal{K}_{1,1}.
     \end{equation}
     Then, further recalling 
    Lemma \ref{cfineq}, there exists $\caK_{0,\frac{\alpha}{2}} > 0$, depending only $\alpha, |\Omega|$ and $T$, such that
    \begin{equation}\label{eq:nearlythereagain}
    \begin{split}
        \|\sigma\|_{L^{2+\frac{\alpha}{2}}([0,T]\times\Omega)}^{2+\frac{\alpha}{2}} \leqslant \caK_{0,\frac{\alpha}{2}} \left(\left\|\partial_x \sigma^\frac{\alpha}{4} \right\|_{L^2([0,T];L^2(\Omega))}^2+1\right) \leqslant \caK_{0,\frac{\alpha}{2}} (\mathcal{K}_{1,1}^2+1).
    \end{split}
    \end{equation}
    Since
    \[
    \frac{2+ \f{\alpha}{2}}{1-\f{\alpha}{4}} > 2 \text{ for all } \alpha  \in (0,1),
    \]
    it follows from Inequality \eqref{eq:nearlythereagain} that there exists $\zeta_\alpha  > 2$ such that 
    \[
    \|\sigma^{1-\f{\alpha}{4}}\|_{L^{\zeta_\alpha}([0,T]\times\Omega)} < c_1
    \] 
    where $c_1$ is a positive constant which depends only on $\alpha,\kappa,|\Omega|,T$ and the Lipschitz constants of the potentials.
    Moreover, $\partial_x\sigma$ may be expressed as
    \[
    \partial_x \sigma ^\frac{\alpha}{4}\cdot \sigma^{1-\f{\alpha}{4}} = \frac{\alpha}{4} \partial_x\sigma,
    \]
    that is, the product of a function bounded in $L^2_{t,x}$ with a function bounded in $L^{\zeta_\alpha}_{t,x}$.
    In particular, since $\zeta_\alpha >2$, it follows from applying H\"older's inequality that there similarly exists 
    $\theta_\alpha  > 1$ and a further constant $\mathcal{K}_5$, whose value depends only on $c_1,\alpha$ and $\mathcal{K}_{1,1}$ such that 
    \[
    \|\partial_x\sigma\|_{L^{\theta_\alpha}([0,T]\times\Omega)} < \mathcal{K}_5.
    \] 
    \end{proof}
\begin{proposition}\label{ftdbound}
        Let $\eta \in I $ and let $(\rho_1,\rho_2)$ denote a classical solution of System \eqref{eq:vcdid}.
        Then, there exists $\mathcal{K}_6 > 0$, whose size depends only on $\Omega, T$, the constants $\alpha,\kappa,\mathcal{K}_{\inf}$ introduced in Hypothesis \ref{hypothesis2}, the Lipschitz constants of the potentials, and 
        \[
        \mathscr{F}_\eta[\rho_{1,0},\rho_{2,0}]
        \]
        such that 
        \begin{equation}\label{eq:dtrhoibound2}
        \|\partial_t \rho_i \|_{L^2([0,T];H^{-3}(\Omega))} \leqslant \mathcal{K}_6, \text{ for } i \in \{1,2\}.
        \end{equation}
\end{proposition}
\begin{proof}
Recall from System \eqref{eq:vcdid} that $\rho_i$ satisfies the equality 
\[
\partial_t\rho_i = \eta \partial_{xx}\rho_i + \partial_x(\rho_i \partial_x(f'(\sigma) + V_i))
\]
on $(0,T)\times \Omega$ for $i \in \{1,2\}$. 
To establish a time derivative bound, the viscosity term is addressed first.

\medskip 

Since $\rho_{i,0} \in \sP(\Omega)$, it follows that 
\[
\|\rho_i\|_{L^2([0,T];L^1(\Omega))} = T^\frac{1}{2}.
\]
Further, since $H^3_0(\Omega) \hookrightarrow C^{2,\frac{1}{2}}_{0}(\Omega) \hookrightarrow C^2_0(\Omega)$, it follows that 
$ (C^2_0(\Omega))^{*}\hookrightarrow H^{-3}(\Omega)$.
As such, we now work to establish a bound for the time derivative in 
\[
L^2([0,T];(H_0^3(\Omega))^{\ast})
\]
uniformly in $\eta$.
In particular, the application of the Aubin--Lions Lemma does not require a sharp control on the quantity $\partial_t\rho_i$ and a uniform bound in any negative Sobolev space is sufficient.

\medskip 

There exists 
and $c_1 \coloneqq c_1(\Omega,T)> 0$, depending only on $|\Omega|$ and $T$, for which the viscosity term satisfies 
\begin{equation}\label{eq:dtviscosity}
\begin{split}
\eta\| & \partial_{xx}\rho_i\|_{L^2([0,T];H^{-3}(\Omega))} \leqslant c_1\eta \|\partial_{xx}\rho_i\|_{L^2([0,T];(C^2_0(\Omega))^{*})}\\
& \leqslant c_1\eta \|\rho_i\|_{L^2([0,T];L^1(\Omega))}\leqslant \eta c_1 T^\frac{1}{2}.
\end{split}
\end{equation}
To address the momentum term, first recognise that the drift part satisfies 
\begin{equation}\label{eq:dtdrift}
\|\rho_i \partial_xV_i \|_{L^2([0,T];L^1(\Omega))} \leqslant \|\partial_xV_i\|_{L^\infty(\Omega)} T^\frac{1}{2}
\end{equation}
In addition, thanks to the estimate of Proposition \ref{gfedi}, the diffusive part satisfies
\begin{equation}\label{eq:dtdiff}
\begin{split}
& \|\rho_i\partial_xf'(\sigma)\|_{L^2([0,T];L^1(\Omega))} \\
& \leqslant 
\|\sqrt{\rho_i}\|_{L^\infty([0,T];L^2(\Omega))}\|\sqrt{\rho_i}\partial_xf'(\sigma)\|_{L^2([0,T];L^2(\Omega))}
\leqslant \mathcal{K}_3.
\end{split}
\end{equation}
Together, Estimates \eqref{eq:dtdrift} and \eqref{eq:dtdiff} show that 
\begin{equation}\label{eq:dtmomentum}
\begin{split}
\|\partial_x&(\rho_i \partial_x(f'(\sigma) + V_i))\|_{L^2([0,T];H^{-3}(\Omega))} \\
& \leqslant c_1\|\partial_x(\rho_i \partial_x(f'(\sigma) + V_i))\|_{L^2([0,T];(C_0^2(\Omega))^{*})} \\
& \leqslant c_1\|\rho_i\partial_xf'(\sigma)\|_{L^2([0,T];L^1(\Omega))} + c_1\|\rho_i \partial_xV_i \|_{L^2([0,T];L^1(\Omega))} \\
& \leqslant c_1\|\partial_xV_i\|_{L^\infty(\Omega)} T^\frac{1}{2} + c_1\mathcal{K}_3.
\end{split}
\end{equation}
The conclusive bound for $\partial_t \rho_i$ is subsequently deduced via the triangle inequality.
In particular, the aggregation of the Inequalities \eqref{eq:dtviscosity} and \eqref{eq:dtmomentum} yields the desired estimate for 
\[
 \|\partial_t \rho_i \|_{L^2([0,T];H^{-3}(\Omega))} \leqslant c_1\left(\|\partial_xV_i\|_{L^\infty(\Omega)} T^\frac{1}{2} + \eta T^\frac{1}{2}+ \mathcal{K}_3\right).
\]
\end{proof}
  \subsection{Compactness}
   Having established all of the necessary bounds, it is left to apply the Aubin--Lions Lemma, establishing strong compactness for each approximate family of densities and to pass to the limiting weak formulation of System \eqref{eq:cdid}.
    As was discussed towards the end of the introductory discussion, the Assumptions \ref{K1}-\ref{K3}, imposed of the potentials and initial data to access the classical regularity theory of the viscous system, are much stronger than Assumptions \ref{J1}-\ref{J4}, which are sufficient for the weak existence theory.

    \medskip 

    Consequently, before establishing the sufficient compactness and subsequent convergence to the limiting system, it is first proven that any initial data and potentials satisfying Assumptions \ref{J1}-\ref{J5} may be approximated by families of smooth data and potentials which, for $\eta \in I$, satisfy Assumptions \ref{K1}-\ref{K3} and \ref{J5}.
    Moreover, the approximation may be constructed in such a way that all quantities listed in the statement of Proposition \ref{sumrequirements} are bounded uniformly in $\eta$.

    \medskip 

    The approximation procedure is performed in two steps. 
    Firstly, an approximating family of initial data and potentials is constructed such that the the quantities 
     \[
    \|\partial_x r_0^\eta\|_{L^1(\Omega)}, \ \|V_1\|_{W^{2,1}(\Omega)}, \  \|V_2\|_{W^{2,1}(\Omega)}, \ \|V_2-V_1\|_{W^{3,1}(\Omega)}.
     \]
    are bounded uniformly in $\eta$.
    
    \medskip 
    
    Then, to control the quantity 
    \[
    \|\mathcal{B}(t)\|_{L^1([0,T])}
    \]
    in the context of the fast-diffusive pressure law, it is necessary to ensure the boundedness of several zero- and first order norms of non-linear functions of $\sigma$ (cf. Proposition \ref{sumrequirements}).
    Several of these norms are a priori bounded in the limit as $\eta \to 0$.
    However, the norms lacking an a priori estimate are each multiplied by a pre-factor that vanishes as $\eta \to 0$.

    \medskip 

    Consequently, the second approximation procedure involves ensuring that the blow-up of each norm which is multiplied by a power of $\eta$ is beaten by the rate at which the pre-factor vanishes as $\eta \to 0$.

        \begin{proposition}\label{logapproximation}
        Let $\rho_{1,0},\rho_{2,0} \in \sP^{ac}(\Omega)$ and let $V_1,V_2 \in W^{1,\infty}(\Omega)$.
        If Hypothesis \ref{hypothesis3} is satisfied then, for each $\eta \in I$, there exist $\rho_{1,0}^\eta,\rho_{2,0}^\eta \in \sP^{ac}(\Omega)$ and $V_1^\eta,V_2^\eta \in W^{1,\infty}(\Omega)$ such that:
       \begin{itemize}
            \item $(\rho_{1,0}^\eta,\rho_{2,0}^\eta )$ and $(V_1^\eta,V_2^\eta)$ satisfy Assumptions \ref{K1}-\ref{K3} and \ref{J5};
           \item $\displaystyle\rho_{i,0}^\eta \rightharpoonup \rho_{i,0}$ and $V_i^\eta \rightharpoonup V_i$ for $i \in \{1,2\}$ as $\eta \to 0$;
           \item the following quantities are bounded uniformly in $\eta$:
       \end{itemize}
       \[
       \|\partial_x r_0^\eta\|_{L^1(\Omega)}, \ \|V_1\|_{W^{2,1}(\Omega)}, \  \|V_2\|_{W^{2,1}(\Omega)}, \ \|V_2-V_1\|_{W^{3,1}(\Omega)}.
       \]
    \end{proposition}
    \begin{proof}
        Let 
        \[
        \sigma_0 \coloneqq \rho_{1,0} + \rho_{2,0}.
        \]
        Since $\rho_{1,0},\rho_{2,0}\in \sP^{ac}(\Omega)$, there exists $r_0 \in L^\infty(\Omega)$ satisfying 
        \[
        \rho_{1,0} = r_0\sigma_0
        \]
        almost everywhere in $\Omega$.
        As facilitated by Assumption \ref{J2}, a representative $r_0$ is fixed such that $r_0 \in BV(\Omega)$.
        Then, for each $\eta \in I$, let $\hat{\sigma}_0^\eta \co\Omega \to \R $ and $\hat{r}_0^\eta \co \Omega \to \R$ be defined as follows
        \[
        \hat{\sigma}_0^\eta(x) \coloneqq 
        \begin{cases}
        1 &\text{ if } x \in [0,3\eta],\\
        \max\{\sigma_0(x),\eta\} &\text{ if } x \in (3\eta, L-3\eta), \\
        1 &\text{ if } x \in [L-3\eta, L],
        \end{cases}
        \]
        \[
        \hat{r}_0^\eta(x) \coloneqq 
        \begin{cases}
        0 &\text{ if } \in [0,3\eta],\\
        r_0(x) &\text{ if } x \in (3\eta, L-3\eta), \\
        0 &\text{ if } x \in [L-3\eta, L].
        \end{cases}
        \]        
        Now, let $\varphi\co \R \to \R$ denote a smooth symmetric and positive mollifier supported on $[-1,1]$ and, for $\eta \in I$, let
        $\varphi^\eta\co \R \to \R$ be given by 
        \[
        \varphi^\eta(x) \coloneqq \eta \varphi\left(\frac{x}{\eta}\right)
        \]
        such that $\varphi^\eta$ is smooth and compactly supported on $(-\eta,\eta)$.
        Then, let $\tilde{\sigma}_0^\eta$ and $\tilde{r}_0^\eta$ be defined in the following manner.
        \[
        \tilde{\sigma}_0^\eta(x) \coloneqq 
        \begin{cases}
        1 &\text{ if } x \in [0,\eta],\\
        \hat{\sigma}_0^\eta\ast \varphi^\eta(x) &\text{ if } x \in (\eta, L-\eta), \\
        1 &\text{ if } x \in [L-\eta, L],
        \end{cases}
        \]
        \[
        \tilde{r}_0^\eta(x) \coloneqq 
        \begin{cases}
        0&\text{ if } x \in [0,\eta],\\
        \hat{r}_0^\eta\ast \varphi^\eta(x) &\text{ if } x \in (\eta, L-\eta), \\
        0 &\text{ if } x \in [L-\eta, L].
        \end{cases}
        \]
        The operation $\ast$ denotes the standard convolution on $\R$, however, since $\varphi^\eta$ is compactly supported in $(-\eta,\eta)$ the above definition makes sense for $\tilde{\sigma}_0,\tilde{r}_0 \co \Omega \to \R$.
        Indeed, for $x \in (\eta,L-\eta)$,  $\tilde{\sigma}_0^\eta(x)$ satisfies the equality
        \[
          \tilde{\sigma}_0^\eta(x)  = \int_{\Omega} \hat{\sigma}_0(y)\varphi^\eta(x-y)\diff y
        \]
        and similarly for $\tilde{r}_0^\eta$.
        Moreover, for $x \in(\eta,2\eta)\cup (L-\eta,L-2\eta)$, 
        $ \sigma_0^\eta(x)$ satisfies the equality
        \[
          \sigma_0^\eta(x)  =\int_{\Omega} \varphi^\eta(x-y)\diff y = 1.
        \]
        Subsequently, $\sigma_0^\eta\co \Omega \to \R$ and $r_0^\eta \co \Omega \to \R$ are defined by 
        \[
         \sigma_0^\eta \coloneqq \frac{2\tilde{\sigma}_0^\eta}{\|\tilde{\sigma}_0^\eta\|_{L^1(\Omega)}} \quad \quad  r_0^\eta \coloneqq \frac{\tilde{r}_0^\eta}{c_1^\eta}.
        \]
        for a constant $c_1^\eta$ which will later be fixed for each $\eta \in I$.
        Further, let $\rho_{1,0}^\eta,\rho_{2,0}^\eta\co \Omega\to \R$ be defined by 
        \begin{align*}
        & \rho_{1,0}^\eta(x) \coloneqq  \sigma_0^\eta(x)r_0^\eta(x) \quad \text{ and } \\
        & \rho_{2,0}^\eta(x) \coloneqq\sigma_0^\eta(x)(1-r_0^\eta(x)) \quad  \text{ respectively.}
        \end{align*}
        The constant $c_1^\eta$ is thus fixed such that $\rho_{1,0}^\eta \in \sP(\Omega)$. 
        Moreover, since $\sigma_0^\eta$ was defined to satisfy $\|\sigma_0^\eta\|_{L^1(\Omega)} =2$, it further follows that $\rho_{2,0}^\eta \in \sP(\Omega)$ for such a choice of $c_1^\eta$.

        \medskip 
        
        Thanks to the construction of $ \tilde{\sigma}_0^\eta$ and $\tilde{r}_0^\eta$, observe that the sum $\sigma_0^\eta$ is bounded uniformly from below on $\overline{\Omega}$ by $\eta$ and that each $\rho_{i,0}^\eta$ is, for every $\eta \in I$: 
        \begin{enumerate}
            \item A probability measure,
            \item A function in $C^\infty(\overline{\Omega})$,
            \item Constant on the set $[0,2\eta]\cup[L-2\eta,L]$.
        \end{enumerate}
        Then, to approximate the potentials, recall that $C^\infty(\overline{\Omega})$ is dense in $W^{k,p}(\Omega)$ for $k \in \N$ and $p \in [1,+\infty)$ (see \cite[Theorem 9.2]{Brezisbook}) and, for each $\eta \in I$, there exist $V_1^\eta,V_2^\eta \in C^\infty(\overline{\Omega})$ such that 
        \[
        \lim_{\eta \to 0}\|V_i^\eta-V_i\|_{W^{2,1}(\Omega)} = 0, 
        \]
        and 
         \[
        \lim_{\eta \to 0}\|(V_1^\eta-V_2^\eta)-(V_1-V_2)\|_{W^{3,1}(\Omega)} = 0.
        \]
        Moreover, since 
        \[
        W^{2,1}(\Omega) \hookrightarrow W^{1,\infty}(\Omega),
        \]
        it follows that $\partial_xV_i^\eta$ converges uniformly to $\partial_x V_i$ on $\overline{\Omega}$. 
        Additionally, since $\partial_x V_1 = \partial_x V_2 = 0$ on $\partial\Omega$, it is further possible to assume that $\partial_xV_1^\eta = \partial_x V_2^\eta = 0 $ on $\partial\Omega$.

        \medskip 
        
        It has thus been shown that the approximating potentials and initial data satisfy Assumptions \ref{K1},\ref{K2} and \ref{J5}. 
        It is left to show  
        \begin{itemize}
            \item the approximating data satisfies \ref{K3};
            \item the appropriate convergence of the initial densities;
            \item that $\|\partial_x r_0^\eta\|_{L^1(\Omega)}$ is bounded uniformly in $\eta$.
        \end{itemize}
        For each $\eta \in I$, the approximate initial data has been constructed to be constant in a neighbourhood of $\partial\Omega$ and the approximating potentials have been constructed to satisfy $\partial_x V_i^\eta = 0$ on $\partial\Omega$.
        As such, all orders of derivatives of the terms $\rho_{i,0}^\eta$ featured in the compatibility conditions \eqref{eq:zeroorder} and \eqref{eq:firstorder} vanish.
        Thus, the second compatibility condition is satisfied and the only remaining non-zero term is the product 
        \[
        \rho_{i,0}^\eta\partial_x V_i^\eta = 0
        \]
        featured in Equation \eqref{eq:zeroorder}.
        However, since $\partial_x V_i^\eta = 0$ on $\partial\Omega$, this product must also vanish.

        \medskip
        
        Next, the convergence of the densities is asserted. 
        Since $\sigma_0,r_0 \in L^1(\Omega)$ it follows from the standard properties of mollifiers that $\sigma_0^\eta,r_0^\eta $ must converge strongly to the respective limits $\sigma_0$ and $r_0$ in $L^1(\Omega)$ as $\eta \to 0$.
        Moreover, since $r_0^\eta$ is uniformly bounded in $L^\infty(\Omega)$, $r_0^\eta$ must also converge to $r_0$ in the $L^\infty(\Omega)$ weak-$\ast$ sense as $\eta \to 0$.

        \medskip 

        The limiting objects $r_0,\sigma_0$ satisfy the relation 
        \[
        r_0\sigma_0 = \rho_{1,0}.
        \]
        Moreover, as the product of an $L^1(\Omega)$ strongly convergent sequence and an $L^\infty(\Omega)$ weak-$\ast$ convergent sequence the product $r_0^\eta\sigma_0^\eta$ must converge weakly in $L^1(\Omega)$ to the limiting density $\rho_{1,0}$ as $\eta \to 0$.
        A similar argument applies for the density $\rho_{2,0}$.
        
        \medskip 

        As a consequence of the convergence of initial densities, it follows that $c_1^\eta \to 1 $ as $\eta \to 0$.
        Then, to show the boundedness of the $BV(\Omega)$ norm, recognise that, since $r_0$ is valued in $[0,1]$, the function $\hat{r}_0$ satisfies the estimate
        \[
        TV\left(\hat{r}_0^\eta\right) \leqslant TV\left(r_0\right) +2.
        \]
        Moreover, since the action of mollification can not increase the total variation, also $r_0^\eta$ satisfies the following estimate.
        \[
        TV\left(r_0^\eta\right) \leqslant c_1^{\eta}\left(TV\left(r_0\right) +2\right).
        \]
        Moreover, since the constant $c_1^\eta$ must converge to $1$ as $\eta \to 0$, the family $(c_1^\eta)_{\eta \in I}$ is bounded uniformly in $\eta$ and hence,
        $ TV\left(r_0^\eta\right) $ is bounded uniformly in $\eta$.
    \end{proof}

  \begin{proposition}\label{approximation}
    Let $\rho_{1,0},\rho_{2,0} \in \sP^{ac}(\Omega)$, $V_1,V_2 \in W^{1,\infty}(\Omega)$ and assume Hypothesis \ref{hypothesis3}.
    Further, for $\eta \in I$ and  $i \in \{1,2\}$, let $\rho_{i,0}^\eta, V_i^\eta$ denote the approximating data constructed in Proposition \ref{logapproximation}.

    \medskip 
    
    Assuming Hypothesis \ref{hypothesis2}, let $(\rho_1^\eta,\rho_2^\eta)$ denote the unique classical solution of \eqref{eq:vcdid} with initial data $(\rho_{1,0}^\eta,\rho_{2,0}^\eta)$ and potentials $(V_1^\eta,V_2^\eta)$.
    The families $(\rho_{1,0}^\eta,\rho_{2,0}^\eta)_{\eta \in I}$ and $(V_1^\eta,V_2^\eta)_{\eta \in I}$ may be constructed such that the following norms are all bounded uniformly in $\eta$.

    \medskip 
    
     \begin{itemize}
     \begin{multicols}{2}  
        \item $\|\partial_x f'(\sigma^\eta) \sqrt{\sigma^\eta}\|_{L^2([0,T]\times\Omega)}$,
     \item $\|(\sigma^\eta)^{1-\alpha}\|_{L^1([0,T];W^{1,1}(\Omega))} $,
    \item $\|\partial_x \log(\sigma^\eta)\|_{L^2([0,T]\times\Omega)}$,
        \item $\|(\sigma^\eta)^{1-2\alpha}\|_{L^1([0,T]\times\Omega)}$,
    \end{multicols}
    \end{itemize}
    \begin{itemize}
    \begin{multicols}{2}
      \item $\eta\|(\sigma^\eta)^{2-2\alpha}\|_{L^1([0,T];W^{1,1}(\Omega))} $
    \item $\eta\|\partial_x (\sigma^\eta)^{\frac{1-\alpha}{2}}\|_{L^2([0,T]\times \Omega)} $,
    \newcolumn
     \item $\eta^2\|\partial_x(\sigma^\eta)^{3-3\alpha}\|_{L^1([0,T]\times \Omega)} $
     \item $\eta^2\|\partial_x(\sigma^\eta)^{1-\alpha}\|_{L^2([0,T]\times \Omega)} $,
     \item  $\eta^3\|\partial_x (\sigma^\eta)^{\frac{3-3\alpha}{2}}\|_{L^2([0,T]\times \Omega)} $.
     \end{multicols}
     \end{itemize}
  \end{proposition}
  \begin{proof}
      First, the boundedness of the norm 
      \[
      \|\partial_x f'(\sigma^\eta) \sqrt{\sigma^\eta}\|_{L^2([0,T]\times\Omega)}.
      \]
      is addressed.
      From Proposition \ref{gfedi}, it follows that there exists a constant $\mathcal{K}_3$ depending only on
      $\mathcal{K}_{\inf},T$, the Lipschitz constants of the potentials, and 
     \[
    \mathscr{F}_\eta[\rho_{1,0}^\eta,\rho_{2,0}^\eta],
    \] 
    such that 
    \[
    \|\partial_xf(\sigma^\eta) \sqrt{\sigma^\eta}\|_{L^2([0,T]\times\Omega)}^2 < \mathcal{K}_3.  
    \]
    The potentials $V_i^\eta$ are constructed so that their Lipschitz constant is bounded uniformly in $\eta$.
    Consequently, it is left to ensure that the quantity
    \[
    \eta \sum_{i=1}^2\int_{\Omega} \rho_{i,0}^\eta\log(\rho_{i,0}^\eta)
    \]
    is bounded uniformly in $\eta$.

    \medskip 
    
    Such a bound may be achieved by re-parameterising the mollifier used to construct the family $\rho_{i,0}^\eta$ in 
    Proposition \ref{logapproximation} such that smoothness of $\rho_{i,0}^\eta$ diverges more slowly as $\eta \to 0$.
    Indeed, such a re-parameterisation, will mean that all the norms
    \[
    \|\rho_{i,0}^\eta\|_{W^{k,p}(\Omega)} , \quad k \in \N, \ p \in [1,+\infty]
    \]
    will also blow up more slowly as $\eta \to 0$.

    \medskip
    
    Indeed, the rate at which the mollifier, introduced in \ref{logapproximation}, converges to the singular kernel may be re-parameterised such that the quantity
    \[
    \eta \sum_{i=1}^2\int_{\Omega} \rho_{i,0}^{\eta}\log(\rho_{i,0}^{\eta})
    \]
    is bounded uniformly in $\eta$.

    \medskip 

    The same parameterisation strategy may be applied to establish boundedness for each of the norms listed in the statement of the proposition, which are multiplied by a power of $\eta$.
    Indeed, recalling Theorem \ref{sumestimates}, it has been shown that for each $\theta \in [1,+\infty)$, there exists a constant $\mathcal{K}_{1,\theta}$ whose value depends only on $\theta, \alpha,\kappa, |\Omega|,T$ the norm 
     \[
     \|\sigma_0^\eta\|_{L^\theta(\Omega)},
     \]
     and the Lipschitz constants of the potentials, such that 
     \begin{equation}\label{eq:recallthetamuineq}
     \begin{split}
     \left\|(\sigma^\eta)^{p}\right\|_{L^\infty([0,T];L^1(\Omega))} & + \left\|\partial_x(\sigma^\eta)^\f{p+\alpha-1}{2}\right\|_{L^2([0,T]\times \Omega)} \\
     & < \caK_{1,\theta} \text{ for every } p\in (0,1).
     \end{split}
     \end{equation}
     Further, recalling Theorem \ref{reciprocalestimates}, there exists, for each $\theta \in (1-\alpha, +\infty)$, a constant $\caK_{2,\theta} > 0$ whose size depends only on $\alpha,\kappa, \theta, |\Omega|,T$, the norm 
     \[
     \|(\sigma_0^\eta)^{-\theta}\|_{L^1(\Omega)},
     \]
     and the Lipschitz constants of the potentials such that 
     \begin{equation}\label{eq:recallreciprocal}
     \left\|(\sigma^\eta)^{-\theta}\right\|_{L^\infty([0,T];L^1(\Omega))} + \left\|(\partial_x\sigma^\eta)^\f{\theta+\alpha -1}{2}\right\|_{L^2([0,T]\times \Omega)}  < \caK_{2,\theta}.
     \end{equation}

     \medskip 

     Since the parameters $\alpha,\kappa,|\Omega|,T$ are fixed and since the Lipschitz constant of the potentials are bounded uniformly in $\eta$, it thus follows from the aggregation of Theorems \ref{sumestimates} and \ref{reciprocalestimates} that, for any $\theta \in [1,+\infty)$
     the sum 
     \begin{equation}\label{eq:quantities}
           \|\partial_x (\sigma^\eta)^\theta\|_{L^2([0,T]\times\Omega)} + \|\partial_x (\sigma^\eta)^{-\theta}\|_{L^2([0,T]\times\Omega)}
     \end{equation}
    may be bounded by a constant which depends only $\theta$ being finite, along with a uniform upper bound and the strict positivity of $\sigma_0^\eta$.

    \medskip 
    
    By an application of Jensen's Inequality the quantities in \eqref{eq:quantities} may be used to bound any of the norms listed in the statement of the proposition which are multiplied by a power of $\eta$.
    Moreover, the dependence of the upper and lower bound on $\eta$ which was constructed in Proposition \ref{logapproximation} may be re-parametrised such that the quantities in \eqref{eq:quantities} blow up arbitrarily slowly in $\eta$ and hence, any such norms can be constructed to vanish as $\eta \to 0$ after multiplication by a pre-factor of $\eta$.

    \medskip 
    
    Attention is now paid to the three remaining norms, which are not multiplied by a factor of $\eta$.
    Firstly, recalling Theorem \ref{sumestimates} in the case $\theta=1$, there exists a constant $\mathcal{K}_{1,1}$ which depends only on $\alpha,\kappa, |\Omega|,T$
     and the Lipschitz constants of the potentials, such that 
     \begin{equation}\label{eq:recallthetamuineq2}
     \begin{split}
     \left\|(\sigma^\eta)^{p}\right\|_{L^\infty([0,T];L^1(\Omega))} & + \left\|\partial_x(\sigma^\eta)^\f{p+\alpha-1}{2}\right\|_{L^2([0,T]\times \Omega)} \\
     & < \caK_{1,1} \text{ for every } p\in (0,1)
     \end{split}
     \end{equation}
     with the choice $p = 1-\alpha $ replaced with $\partial_x \log(\sigma^\eta)$.
     In particular, the constant $\caK_{1,\theta}$ is independent of $\eta$.

     \medskip 

     By choosing $p = 1-\alpha$, it follows that $\partial_x \log(\sigma^\eta)$ is bounded in $L^2_{t,x}$ uniformly in $\eta$.

     \medskip 

     To establish the $\eta$-independent $W^{1,1}(\Omega)$ estimate for $(\sigma^\eta)^{1-\alpha}$ it is enough to estimate $\partial_x(\sigma^\eta)^{1-\alpha}$.
     Indeed, since 
     \[
     \|\sigma^\eta\|_{L^\infty([0,T];L^1(\Omega)} =2 
     \]
     and since $\alpha \in (0,1)$, it follows from an application of Jensen's inequality that
     $(\sigma^\eta)^{1-\alpha}$ is bounded in $L^1_{t,x}$ independently of $\eta$.
    
    \medskip
    
    Subsequently, choose $p = 1-\frac{\alpha}{2}$ in Inequality \eqref{eq:recallthetamuineq} and then recall from 
    Lemma \ref{cfineq}, there exists $\caK_{0,\frac{\alpha}{2}} > 0$, depending only $\alpha, |\Omega|$ and $T$, such that
    \begin{equation}\label{eq:recallnearlythereagain}
    \begin{split}
        \|\sigma\|_{L^{2+\frac{\alpha}{2}}([0,T]\times\Omega)}^{2+\frac{\alpha}{2}} \leqslant \caK_{0,\frac{\alpha}{2}} \left(\left\|\partial_x \sigma^\frac{\alpha}{4} \right\|_{L^2([0,T];L^2(\Omega))}^2+1\right) \leqslant \caK_{0,\frac{\alpha}{2}} (\mathcal{K}_{1,1}^2+1).
    \end{split}
    \end{equation}
    In particular, it follows from Inequality \eqref{eq:recallnearlythereagain} that $\sigma^\eta$ is bounded in $L^2_{t,x}$ independently of $\eta$.

    \medskip 
    
    Then, by taking $p= 1-3\alpha$, it follows from Inequality \eqref{eq:recallthetamuineq2} that the norm 
    \[
    \left\|\partial_x(\sigma^\eta)^{-\alpha}\right\|_{L^2([0,T]\times \Omega)} 
    \]
    is also bounded independently of $\eta$ and hence, so is the product 
    \begin{align*}
           & \|\partial_x (\sigma^\eta)^{1-\alpha}\|_{L^1([0,T]\times\Omega)} \\
           &\leqslant \frac{1-\alpha}{\alpha}\|\partial_x(\sigma^\eta)^{-\alpha}\|_{L^2([0,T]\times \Omega)} \|\sigma^\eta\|_{L^2([0,T]\times \Omega)}. 
    \end{align*}
    It is left to consider the last norms:
    \[
    \|(\sigma^\eta)^{1-2\alpha}\|_{L^1([0,T]\times\Omega)}.
    \]
    If $\alpha < \frac{1}{2}$, then such an estimate follows immediately from the equality 
     \[
     \|\sigma^\eta\|_{L^\infty([0,T];L^1(\Omega))} =2.
     \]
    As such, the setting $\frac{1}{2}\leqslant \alpha < \frac{2}{3}$ is first addressed.
    In this instance, it follows from Inequality \eqref{eq:recallthetamuineq2} that 
    \[
    \|\partial_x(\sigma^\eta)^\f{\gamma}{2}\|_{L^2([0,T]\times \Omega)}
    \]
    is bounded independently of $\eta$ for every $\gamma \in (\alpha -1, \alpha)$.
    In particular, this also means that $(\sigma^\eta)^{\gamma}$ is bounded in $L^1([0,T]\times\Omega)$, independently of $\eta$, for every $\gamma \in (\alpha -1,\alpha)$.
    In particular, for $\alpha \in [\frac{1}{2},\frac{2}{3})$, the inclusion 
    \[
    1-2\alpha \in (\alpha -1,\alpha) 
    \]
    is satisfied and hence the bound may be established by choosing 
    $\gamma = 1-2\alpha$.

    \medskip 
    
    It is then left to bound 
     \[
    \|(\sigma^\eta)^{1-2\alpha}\|_{L^1([0,T]\times\Omega)}
    \]
    in the setting $\alpha \in [\frac{2}{3},1)$. 
    Indeed, thanks to Assumption \ref{J7}, it is assumed that $\sigma_0$ is bounded uniformly from below.
    
    \medskip 
    
    Moreover, since $\sigma_0^\eta$ is constructed in Proposition \ref{logapproximation}, by a mollification of a uniformly strictly positive approximation of $\sigma_0$ and since mollification preserves lower bounds, it follows that 
    \[
    \|(\sigma^\eta_0)^{1-2\alpha}\|_{L^1([0,T]\times\Omega)}
    \]
    is bounded uniformly in $\eta$.
    Consequently, it follows from Inequality \eqref{eq:recallreciprocal} that 
    \[
    \|(\sigma^\eta)^{1-2\alpha}\|_{L^1([0,T]\times\Omega)}
    \]
    is bounded uniformly in $\eta$.
  \end{proof}

   \begin{remark}
       In the case that Hypothesis \ref{hypothesis1} holds for a power-law type pressure, Proposition \ref{approximation} may similarly be proven, however, thanks to Remark \ref{remark:bounds}, it is unnecessary to show the uniform boundedness of the form 
       \[
        \|(\sigma^\eta)^{1-2\alpha}\|_{L^1([0,T]\times\Omega)}.
        \]
        Similarly, for the choice of logarithmic pressure, it is not necessary to establish any uniform bound for the aggregate density estimates stated in Proposition \ref{sumrequirements} since the function $\xi$ is instead a constant.
   \end{remark}
    \begin{thm}\label{compactness}
        For each $\eta \in I$, let $(\rho_1^\eta,\rho_2^\eta)$ denote a classical solution of System \eqref{eq:vcdid} with associated potentials potentials $V_1^\eta,V_2^\eta$ and initial data $(\rho_{1,0}^\eta,\rho_{2,0}^\eta)$ satisfying Assumptions \ref{K1}-\ref{K4} and Hypothesis \ref{J5}.
       Further, suppose that 
       \[
       \|\partial_x r_0\|_{L^1(\Omega)}
       \]
       along with all the quantities listed in the statement of Proposition \ref{sumrequirements} are bounded uniformly in $\eta$.
        Then, there exists $\theta_\alpha >1$ such that, for each $i \in \{1,2\}$ and $p \in [1,+\infty)$, the family $(\rho_i^\eta)_{\eta \in I}$ is pre-compact for the strong topology in $L^{\theta_\alpha}([0,T];L^p(\Omega))$.
    \end{thm}
    \begin{proof}
        Thanks to Theorem \ref{fConcludingGronwall}, there exists $\mathcal{K}_4 > 0$ for which the inequality
        \[
        \|\partial_x r ^\eta\|_{L^\infty([0,T];W^{1,1}(\Omega))} < \mathcal{K}_4 
        \]
        holds for each $\eta \in I$. 
 
        \medskip 
        
        Further, thanks to Corollary \ref{falpreparation}, there exist $\mathcal{K}_{5} > 0$ and $\theta_{\alpha} > 1$ for which the inequality 
        \[
        \|\partial_x \sigma ^\eta\|_{L^{\theta_\alpha}([0,T];W^{1,1}(\Omega))} < \mathcal{K}_{5}
        \]
        holds for each $\eta \in I$.

        \medskip 
        
        By using each of these estimates in conjunction with the embedding $W^{1,1}(\Omega) \hookrightarrow L^\infty(\Omega)$, it then follows that the products 
        \[
        r^\eta\sigma^\eta, (1-r^\eta)\sigma^\eta = \rho_1^\eta,\rho_2^\eta
        \]
        are bounded in $L^{\theta_\alpha}([0,T];W^{1,1}(\Omega))$ by a constant depending only on $\Omega, T, \mathcal{K}_4, \mathcal{K}_5$.

        \medskip 
        
        Further, recall from Proposition \ref{ftdbound} that there exists $\mathcal{K}_{6} > 0$ such that
        \[
        \|\partial_t\rho_i^\eta\|_{L^{\theta_\alpha}([0,T];H^{-3}(\Omega))} < \mathcal{K}_{6}
        \]
        for each $i \in \{1,2\}$ and $\eta \in I$.

        \medskip 
        
        Moreover, $|\Omega|,T,\kappa,\alpha$ are all fixed and, hence, if all the quantities listed in the statement of Proposition \ref{sumrequirements} along with the norm
        \[
        \|\partial_x r_0^\eta\|_{L^1(\Omega)}
        \]
        are bounded uniformly in $\eta$ then it may readily be verified from Theorems \ref{sumrequirements}, \ref{falpreparation} and \ref{ftdbound} that the constant $\caK_4,\caK_5$ and $\caK_6$ be chosen independently of $\eta$.

        \medskip 
        
        Now, let the set $X^{\alpha}$ be defined 
        \[
        X^{\alpha} \coloneqq \left\{f \in L^{\theta_\alpha}([0,T]:W^{1,1}(\Omega)) \ | \ \partial_tf \in L^{\theta_\alpha}([0,T];H^{-3}(\Omega))\right\}.
        \]
        For each $p \in [1,+\infty)$, the embedding $W^{1,1}(\Omega)\hookrightarrow L^p(\Omega)$ is compact and the embedding $L^p(\Omega) \hookrightarrow H^{-3}(\Omega)$ is continuous.
        Consequently, it follows from the Aubin--Lions Lemma (see \cite[Section 8; Page 85; Corollary 4]{AubinLionsSimon}) that the embedding 
        $X^{\alpha} \hookrightarrow L^{\theta_\alpha}([0,T];L^p(\Omega))$ is compact.

        \medskip 

        Thanks to the aforementioned uniform in $\eta$ estimates, the families $(\rho_1^\eta)_{\eta \in I}, (\rho_2^\eta)_{\eta \in I}$ inhabit a bounded subset of $X^{\alpha}$ and, hence, inhabit a strongly pre-compact subset of $L^{\theta_\alpha}([0,T];L^p(\Omega))$ for each $p \in [1,+\infty)$.
        \end{proof}
\subsection{Convergence}
Equipped with the requisite compactness for each individual density, it is then left to establish the convergence to System \eqref{eq:cdid}

\begin{proof}[Proof of Theorem \ref{maintheorem}]
   Consider initial data and potentials $\rho_{i,0}\in \sP^{ac}(\Omega)$,$V_i\in W^{1,\infty}(\Omega)$ satisfying Hypothesis \ref{hypothesis3}.
   Let $(\rho_{i,0}^\eta)_{\eta \in I}$ and $(V_i^\eta)_{\eta \in I}$, denote the approximating sequence of data constructed in Proposition \ref{approximation}.

   \medskip 
   
    Now, for each $\eta \in I$, let $(\rho_1^\eta,\rho_2^\eta)$ denote the unique classical solution of System \eqref{eq:vcdid} with initial data $(\rho_{1,0}^\eta,\rho_{2,0}^\eta)$ and potentials $V_1^\eta,V_2^\eta$.
    Thanks to Proposition \ref{compactness}, it follows that there exists $\theta_\alpha > 1$ such that, for each $i \in \{1,2\}$ and $p \in [1,+\infty)$, the families $(\rho_i^\eta)_{i \in I}$ are pre-compact for the strong topology in $L^{\theta_\alpha}([0,T];L^p(\Omega))$. 

    \medskip 
    
    Consequently, extract a sub-sequence (for which we do not re-label the index $\eta$) such that 
    $\rho_{i}^\eta$ converges strongly and point-wise almost everywhere (cf. \cite[Theorem 2.2.5.]{Bogachev}) to a limiting profile $\rho_i$ as $\eta \to 0$.
    
    \medskip 

    To address the convergence of the momentum term $\rho_i^\eta \partial_x\log(\sigma^\eta)$, recognise that such a product may be re-written in the form 
    \[
    \rho_i^\eta \partial_xf'(\sigma^\eta)  = \sqrt{\sigma^\eta} \partial_xf' (\sigma^\eta)\cdot \f{\rho_i^\eta}{\sqrt{\sigma^\eta}}.
    \]
    Now, it was established in Proposition \ref{approximation} that the term $\sqrt{\sigma^\eta} \partial_xf' (\sigma^\eta)$ is bounded in $L^2([0,T]\times\Omega)$ uniformly in $\eta$.
    Hence, up to the extraction of a further sub-sequence, the family $\sqrt{\sigma^\eta} \partial_xf' (\sigma^\eta)$ converges weakly in $L^2([0,T]\times\Omega)$ to a limiting object denoted $\zeta$ as $\eta \to 0$.

    \medskip 
    
    However, $\rho_i^\eta$ converges strongly in $L^{\theta_\alpha}_{t,x}$ to $\rho_i$, the limiting object must satisfy $\zeta = \sqrt{\rho_1+\rho_2} \partial_xf' (\rho_1+\rho_2)$

    \medskip 

    Secondly, the quotient term, 
    \[
    \f{\rho_i^\eta}{\sqrt{\sigma^\eta}}
    \]
    converges point-wise to 
    \[
    \f{\rho_i}{\sqrt{\sigma}}
    \]
    as $\eta \to 0$.
    Moreover, it follows from Egoroff's Theorem \cite[Theorem 2.2.1.]{Bogachev} and the Vitali Convergence Theorem \cite[Theorem 4.5.4.]{Bogachev} that, for any $q \in [1,+\infty)$, this quotient also converges to such a limit, strongly in $L^q([0,T]\times\Omega)$, if the sequence 
    \[
     \left|\f{\rho_i^\eta}{\sqrt{\sigma^\eta}}\right|^q
    \]
    is uniformly integrable.
    To prove a uniform integrability result, we will make use of the De La Vall\'ee-Poussin Theorem \cite[Theorem 4.5.9.]{Bogachev} in the following fashion.
    \medskip 
    
    From Lemma \ref{cfineq}, it follows that 
    \[
    \|\sigma^\eta\|_{L^2([0,T]\times\Omega)}
    \]
    is bounded uniformly in $\eta$.
    Further, recognise the inequality 
    \[
     \left|\f{\rho_i^\eta}{\sqrt{\sigma^\eta}}\right|^4 \leqslant (\sigma^\eta)^2.
    \]
    Then, by letting $[0,+\infty) \ni t\mapsto G(t) $ denote the map
    \[
    t \mapsto t^2
    \]
    it follows that
    \begin{align*}
        \int_0^T& \int_{\Omega} G\left(\left|\f{\rho_i^\eta}{\sqrt{\sigma^\eta)}}\right|^2\right) \diff x \diff t =  \int_0^T\int_{\Omega}  \left|\f{\rho_i^\eta}{\sqrt{\sigma^\eta}}\right|^4 \diff x \diff t \\
        & \leqslant  \|\sigma^\eta\|_{L^2([0,T]\times\Omega)}.
    \end{align*}
    In particular, since 
    \[
    \|\sigma^\eta\|_{L^2([0,T]\times\Omega)}
    \]
    is bounded uniformly in $\eta$, 
    the above inequality shows that the family 
    \[
    \left(\left|\f{\rho_i^\eta}{\sqrt{\sigma^\eta}}\right|^2\right)_{\eta \in I}, 
    \]
    is uniformly integrable for $\eta \in I$.
    This is a consequence of the De La Vall\'ee-Poussin Theorem \cite[Theorem 4.5.9.]{Bogachev} which uses the super-linearity of the map $G$.

    \medskip 

    As mentioned previously, together, the uniform integrability and point-wise convergence mean that, consequent to Vitali's Convergence Theorem, the family 
    \[
    \left(\f{\rho_i^\eta}{\sqrt{\sigma^\eta}}\right)_{\eta \in I}, 
    \]
    admits a sub-sequence which converges strongly to it's point-wise limit in $L^2([0,T]\times\Omega)$.

    \medskip 
    
    As the product of an $L^2_{t,x}$ strongly convergent sequence and an $L^2_{t,x}$ weakly convergent sequence, it now follows that the product 
    \[
    \rho_i^\eta \partial_xf'(\sigma^\eta)  = \sqrt{\sigma^\eta} \partial_xf' (\sigma^\eta)\cdot \f{\rho_i^\eta}{\sqrt{\sigma^\eta}}.
    \]
    must converge weakly in $L^1([0,T]\times\Omega)$, as $\eta \to 0$, to the limiting momentum  
   \[
    \rho_i \partial_xf'(\rho_1+\rho_2) 
    \]
    
    Then, passing to the limit as $\eta \to 0$ in the weak formulation of System \eqref{eq:vcdid} and using the established convergence of the cross-diffusive terms along with the convergence of the potentials established in Proposition \ref{approximation}, we deduce that the following equality holds for $i \in \{1,2\}$ and any $\varphi \in C_{c}^\infty([0,T)\times\overline{\Omega})$ 
    satisfying $\partial_x\varphi = 0$ on $(0,T)\times \partial\Omega$.

    \begin{flalign*}
       & \int_{\Omega} \rho_{i,0}\varphi_0 \diff x + \int_0^T\int_{\Omega}\rho_i\partial_t\varphi \diff x \diff t \\
        & = \lim_{\eta\to 0} \bigg(\int_{\Omega} \rho_{i,0}^\eta\varphi_0 \diff x + \int_0^T\int_{\Omega}\rho_i^\eta\partial_t\varphi \diff x \diff t\bigg) \\
         & = \lim_{\eta\to 0}
        \int_0^T\int_{\Omega} \rho_i^\eta\partial_x(f'(\rho_1^\eta+\rho_2^\eta)+V_i^\eta)\partial_x \varphi - \eta\rho_i^\eta \partial_{xx}^2\varphi \diff x \diff t\\
       & = \int_0^T\int_{\Omega} \rho_i\partial_x(f'(\rho_1+\rho_2)+V_i)\partial_x \varphi \diff x \diff t.
    \end{flalign*}
    Thanks to the above inequality, it follows that $(\rho_1,\rho_2)$ defines a weak solution of System \eqref{eq:cdid}.
    However, to show the disappearance of the viscosity term, it was required to take test functions satisfying $\partial_x\varphi = 0$ on $(0,T)\times \partial\Omega$.
    Consequently, it is not clear that the limiting System satisfies the no-flux boundary condition.

    \medskip 
    
    To recognise that the no-flux boundary condition \emph{is} weakly satisfied, recall that, thanks to the bounds established in Proposition
    \ref{gfedi} and Corollary \ref{falpreparation} there exists $p\ast> 1$ such that the momentum term $\rho_i \partial_x(f'(\rho_1+\rho_2) +V_i)$ belongs to $L^{p\ast}([0,T]\times\Omega)$.
    By letting, $q\ast$ denote the H\"older conjugate of $p\ast$ a density argument allows the set of admissible test functions to extended to $W^{q\ast}([0,T]\times\Omega)$.
    
    \medskip 
    
    Specifically, test functions $\varphi$ belonging to $W^{q\ast}([0,T]\times\Omega)$ do not have to satisfy $\partial_x\varphi = 0 $ on $(0,T)\times \partial\Omega$.
\end{proof}

\paragraph{Acknowledgements}
The second author was supported by the Engineering and Physical Sciences Research Council [Grant Number EP/W524426/1].

\begin{itemize}
    \item Conflict of Interest - The authors declare that there are no known conflicts of interest associated with this
manuscript and no financial support has influenced its outcome.
    \item Data Availability Statement - Data sharing is not applicable to this article as no
 data sets were generated or analysed.
\end{itemize}

\newpage 

\appendix
\section{Well-Posedness for the Viscous Approximation}\label{appendix:A}
\begin{proposition}\label{shorttime}
    Given initial data $(\rho_{1,0},\rho_{2,0})$, there exists $\mathcal{T} > 0$ and a pair of curves $(\rho_1,\rho_2)$ satisfying
    \[
    \rho_1,\rho_2 \in C([0,\caT);C^2(\overline{\Omega}))\cap C^1([0,\caT);C(\overline{\Omega})),
    \]
    such that $(\rho_1,\rho_2)$ defines the unique classical solution to System \eqref{eq:vcdid}.
    In addition, $\rho_1+\rho_2$ is strictly positive on $[0,\caT)\times \overline\Omega$.
\end{proposition}
\begin{proof}
    Let $G \subset \R^2$ denote the open half-plane
    \[
    G:= \left\{(x,y)\in \R^2| x+y > 0\right\}.
    \]
    The map $(v_1,v_2) \mapsto f''(v_1+v_2)$ belongs to $C^{1,1}(G)$ due to Assumption \ref{K4} whilst $(\rho_{1,0},\rho_{2,0}) \in C(\overline{\Omega};G)$  due to Assumption \ref{K2}.
    Moreover, it may readily be verified that, for any $\eta \in (0,1)$, System \eqref{eq:vcdid} is associated with a normally elliptic boundary value problem over $G$ (cf. \cite[Pages 1-3.]{Amannqlsexistence}).
    Consequently, it follows from \cite[Theorem 1.]{Amannqlsexistence} that there exists $\caT > 0$ such that System \eqref{eq:vcdid} admits a unique classical solution $(\rho_1,\rho_2)$ belonging to the class
    \[
    C([0,\caT);C(\overline{\Omega};G))\cap C((0,\caT);C^2(\overline{\Omega};\R^2))\cap C^1((0,\caT);C(\overline{\Omega};R^2)).
    \]
    In particular, $\rho_1+ \rho_2 > 0$ on $[0,\caT)\times \overline{\Omega}$.

    \medskip 

    To establish smoothness up to time $t= 0$ it is further necessary to impose a compatibility condition between the initial and boundary data.
    For System \eqref{eq:vcdid}, this condition coincides with the zero order condition stated in Assumption \ref{K3}.
    Moreover, thanks to the regularity assumed in \ref{K1} and the compatibility condition \ref{K3}, it follows from \cite[Theorem 1.1.]{acquistapaceqps} that 
    \[
      \rho_1,\rho_2 \in C([0,\caT);C^2(\overline{\Omega}))\cap C^1([0,\caT);C(\overline{\Omega})).
    \]
\end{proof}

Here, we derive a uniform upper and lower bound for $\sigma$ which is independent of any short-time existence horizon $\caT$ for which $\caT \leqslant T$.

\medskip 

Establishing the limiting case requires the use of the Moser--Alikakos iteration type argument, used in \cite[Theorem 3.1.]{KimZhang}, which employs the use of Sobolev embedding theory to produce a bound via a convergent telescope product.
\begin{proposition}\label{pedi}
  Let $\caT, \eta > 0 $ , $\bs{v} \in C^1([0,\caT);\overline{\Omega})$ and let 
  \[
  u\in C([0,\caT);C^2(\overline{\Omega}))\cap C^1([0,\caT);C(\overline{\Omega}))
  \]
  denote a strictly positive classical solution of the non-linear Fokker--Planck equation 
  \begin{equation}\label{eq:nlfp2}
      \begin{cases}
      \begin{alignedat}{3}
        \partial_t&u && = \partial_x (u(\eta\partial_x\log(u)+ \partial_xf'(u)+ \bs{v})), &&\text{ in } (0,\caT) \times \Omega,\\ 
          &u_0 && = u &&\text{ on } \{0\}\times \overline{\Omega},\\
        &0 && = u (\eta\partial_x\log(u) + \partial_xf'(u) + \bs{v}) &&\text{ on } (0,\caT)\times \partial\Omega.
      \end{alignedat}
      \end{cases}
  \end{equation}
  Then, there exists $K_1 \coloneqq K_1(\bs{v}) > 0$, whose size depends only on $\|\bs{v}\|_{L^\infty([0,\caT)\times\Omega)}$ such that, for each $p \in \R \setminus \{0,1\}$, the following energy dissipation inequality is satisfied on $(0,\caT)$.
  \begin{equation}\label{eq:pedi}
  \begin{split}
      \frac{1}{p(p-1)}\partial_t \int_{\Omega}u^p \diff x & + \frac{1}{2}\int_{\Omega} \left(\eta+ 2 uf''(u)\right)u^{p-2}|\partial_x u|^2 \diff x \\
      & \leqslant \frac{C_1}{\eta}\int_{\Omega}u^p \diff x.
  \end{split}
\end{equation}
Furthermore, there exist $K_2,K_3,K_4,K_5 > 0$, whose sizes depend only on $K_1,|\Omega|$ and $\eta$, such that, when $|p| > K_5$, the following energy dissipation inequality is satisfied on $(0,\caT)$
  \begin{equation}\label{eq:alikakosineq}
  \partial_t\int_{\Omega}u^p  \diff x+ K_2\int_{\Omega}u^p\diff x\leqslant K_3|p|^{K_4}\left(\int_{\Omega}u^\f{p}{2}\right)^2.
  \end{equation}
\begin{proof}
        Since $u$ is assumed to be strictly positive, both $u$ and its reciprocal belong to $C^1((0,\caT);C(\overline{\Omega}))$. 
        Consequently, the Leibniz rule facilitates the interchange of integration and differentiation for the following energy dissipation equality which is satisfied for $t \in (0,\caT)$ and $p \in \R \setminus\{0,1\}$.
    \begin{equation}\label{eq:pede1}
            \frac{1}{p(p-1)}\partial_t\int_{\Omega} u^p \diff x =  \frac{1}{p-1} \int_{\Omega} u^{p-1}\partial_tu \diff x.
    \end{equation}
    Substitute the value of $\partial_tu$ on the right hand-side of Equality \eqref{eq:pede1} using Equation \eqref{eq:nlfp2} and perform integration by parts on the resulting integral. 
    The resultant boundary term vanishes due to the no-flux boundary condition and hence the following equality is satisfied.
    \begin{equation}\label{eq:pede}
    \begin{split}
            \frac{1}{p(p-1)}\partial_t&\int_{\Omega} u^p \diff x\\
          =  -\frac{1}{p-1} &\int_{\Omega} \partial_x u^{p-1}( u(\eta\partial_x\log(u) + \partial_xf'(u) + \bs{v})) \diff x\\
          = - & \int_{\Omega} (\eta+ uf''(u)) |\partial_x u|^2u^{p-2} + u^{p-2} \partial_x u \cdot u \bs{v}\diff x.    
    \end{split}
    \end{equation}
    The first summand in the third line of Equality \eqref{eq:pede}, which comes from the diffusion, has a negative sign and this helps in controlling the energy dissipation.
    On the other hand, the sign of the product $\partial_xu \cdot \bs{v}$ is not clear.
    However, since $\eta > 0$, the gradient appearing in this product may be controlled using the diffusive contribution. 
    In particular, applying Young's $L^2$ inequality to the drift term, the inequality
    \begin{equation}\label{eq:youngdrift}
    |\partial_x u \cdot u \bs{v}|  - \f{\eta}{2}|\partial_x u|^2 \leqslant  \f{1}{2\eta} |u \bs{v}|^2  
     \leqslant \frac{1}{2\eta}\|\bs{v}\|^2_{L^\infty([0,\caT]\times\Omega)} |u|^2.
    \end{equation}
    is satisfied.

    \medskip 
    
    Moreover, by letting 
    \[
    K_1 \coloneqq \frac{1}{2}\|\bs{v}\|^2_{L^\infty([0,\caT]\times\Omega)}
    \]
    and using Inequality \eqref{eq:youngdrift} as an upper bound for the drift term featured in the right hand-side of the energy dissipation equality \eqref{eq:pede}, Inequality \eqref{eq:pedi} is derived.
    
    \medskip 
    To establish Inequality \eqref{eq:alikakosineq}, first recognise that when $|p| > 2$
    the following inequality is satisfied. 
    \[
    p(p-1)u^{p-2}|\partial_x u|^2 =  4\left(1-\f{1}{p}\right)|\partial_x u^\f{p}{2}|^2 \geqslant 2|\partial_x u^\f{p}{2}|^2
    \]
    By using the above inequality as a lower bound for the squared derivative term in the first line of Inequality \eqref{eq:pedi} and multiplying both sides of the inequality by $p(p-1)$, the following inequality is derived for $|p| > 2$.
        \begin{equation}\label{eq:pedi2}
      \partial_t \int_{\Omega}u^p \diff x + {\eta}\int_{\Omega} \left|\partial_x u^\f{p}{2}\right|^2 \diff x \leqslant \frac{K_1}{\eta}p(p-1)\int_{\Omega}u^p \diff x.
\end{equation}
    Secondly, the Gagliardo--Nirenberg Inequality is used (cf.\cite[Theorem:, Page 125]{Nirenbergineq}).
    In particular, it follows from the Gagliardo-Nirenberg Inequality that there exists $c_1\coloneqq c_1(\Omega)$, whose value depends only on $|\Omega|$, such that the the following inequality for each $p \in \R\setminus \{0\}$.

    \begin{equation}\label{eq:GNineq}
       \left\|u ^\frac{p}{2}\right\|_{L^2(\Omega)} \leqslant c_1 \left(     \left\|\partial_x u ^\frac{p}{2}\right\|_{L^2(\Omega)}^\f{1}{3}     \left\| u ^\frac{p}{2}\right\|_{L^1(\Omega)}^\f{2}{3} + \left\|u ^\frac{p}{2}\right\|_{L^1(\Omega)}\right)
    \end{equation}
    After applying the $L^3 - L^\frac{3}{2}$ Young inequality to the product on the right hand-side of Inequality \eqref{eq:GNineq} and then applying Jensen's inequality, the following inequality is derived for each $\delta > 0$.
    \begin{equation}\label{eq:youngGNineq1}
    \begin{split}
       & \left\|u ^\frac{p}{2}\right\|_{L^2(\Omega)}^2 \\
       & \leqslant c_1^2\left( \frac{\sqrt{\delta}}{\sqrt{p(p-1)}}\left\|\partial_x u ^\frac{p}{2}\right\|_{L^2(\Omega)} +(1+ (\delta^{-1}p(p-1))^\f{1}{4}) \left\| u ^\frac{p}{2}\right\|_{L^1(\Omega)} \right)^2
       \\
       & \leqslant 2c_1^2\left( \frac{\delta}{p(p-1)}\left\|\partial_x u ^\frac{p}{2}\right\|_{L^2(\Omega)}^2 +(1+( \delta^{-1}p(p-1))^\f{1}{4})^2 \left\| u ^\frac{p}{2}\right\|_{L^1(\Omega)}^2 \right).
    \end{split}
    \end{equation}
    Similarly, by applying Young's inequality and then Jensen's Inequality \eqref{eq:GNineq}, without incorporating the factors of $\delta $ and $p(p-1)$, it follows that
        \begin{equation}\label{eq:youngGNineq2}
    \begin{split}
       \left\|u ^\frac{p}{2}\right\|_{L^2(\Omega)}^2 - 8c_1^2\left\| u ^\frac{p}{2}\right\|_{L^1(\Omega)}^2 \leqslant 2c_1^2 \left\|\partial_x u ^\frac{p}{2}\right\|_{L^2(\Omega)}^2.
    \end{split}
    \end{equation}
    On the one hand, Inequality \eqref{eq:youngGNineq1} may be used as an upper bound for the right hand-side of Inequality \eqref{eq:pedi}.
    On the other, Inequality \eqref{eq:youngGNineq2} may be used as a lower bound for the left hand-side.

    \medskip 

    In addition to the Inequalities \eqref{eq:youngGNineq1} and \eqref{eq:youngGNineq2}, there exists $P_1 > 0$, whose size depends only on $\delta$, such that the following inequalities holds for all $|p| > P_1$.
    \begin{equation}\label{eq:p-1top}
         (1+(\delta^{-1}p(p-1))^\frac{1}{4})^2\leqslant \frac{|p|^2}{\sqrt{\delta}} \text{ and } p(p-1) \leqslant 2p^2.
    \end{equation}
    Assuming that $|p| > P_1$, use Inequality \eqref{eq:youngGNineq1} to bound the right hand-side of Inequality \eqref{eq:pedi2} and then use the Inequalities \eqref{eq:p-1top} to bound the factor which multiplies the quantity 
    \[
    \|u^\frac{p}{2}\|_{L^1(\Omega)}^2.
    \]
    From the application of such inequalities it follows that 
    \begin{equation}\label{eq:deltac}
    \begin{split}
   \partial_t & \int_{\Omega}u^p \diff x + {\eta}\int_{\Omega} \left|\partial_x u^\f{p}{2}\right|^2 \diff x \\
    & \leqslant \frac{K_1}{\eta}p(p-1)\int_{\Omega}u^p \diff x\\
     & \leqslant 2\frac{K_1}{\eta}c_1^2\delta \int_{\Omega} \left|\partial_x u^\f{p}{2}\right|^2 \diff x + 4\frac{K_1}{\eta\delta^2}c_1^2p^4 \left(\int_{\Omega} u^\frac{p}{2} \diff x\right)^2.
    \end{split}
      \end{equation}
      Moreover, by choosing $\delta$ such that 
      \[
      2\frac{K_1}{\eta}c_1^2\delta < \frac{\eta}{2}, \text{ denoting }
      K_2 = 4 \frac{K_1}{\eta\delta^2}c_1^2,
      \]
        and rearranging Inequality \eqref{eq:deltac}, it follows that
    \begin{equation}\label{eq:alikakosupper}
    \partial_t \int_{\Omega}u^p \diff x + \f{\eta}{2}\int_{\Omega} \left|\partial_x u^\f{p}{2}\right|^2 \diff x  \leqslant K_2|p|^4\left(\int_{\Omega} u^\frac{p}{2} \diff x\right)^2.
      \end{equation}
      Now, Inequality \eqref{eq:youngGNineq2} may be used as a lower bound for the squared gradient term Inequality \eqref{eq:alikakosupper} from which it follows that
    \begin{align*}
    \partial_t \int_{\Omega}u^p \diff x + \f{\eta}{4c_1^2}\int_{\Omega}u^\f{p}{2} \diff x  \leqslant \left(K_2|p|^4 + 2\eta\right)\left(\int_{\Omega} u^\frac{p}{2} \diff x\right)^2.
    \end{align*}
    Now, let 
    \[
    K_3 = \frac{\eta}{4c_1^2}
    \]
    and recognise that for any $K_4 > 5$ there exists $P_2 > 0 $ such that, for all $\eta \in [0,1]$, the following inequality holds for all $|p| > P_2$
    \[
    K_3|p|^{K_4}> K_3|p|^4 + 4\eta.
    \]
    So far, it has only been necessary to assume that $|p| > \max\{2, P_1,P_2\}$ where, in particular, $P_1$ has depended only on $\delta$ (whose size, subsequently, depended only on $\eta, K_1, c_1$) along with the constants $\eta, K_1,c_1$. 

    \medskip 

    Consequently, there exists a constant $K_5$, depending only on $\eta, K_1, c_1$ such that the following inequality holds for all $|p| > K_5$.
    \begin{equation}\label{eq:alikakosrestated}
     \partial_t\int_{\Omega}u^p  \diff x+ K_2\int_{\Omega}u^p\diff x\leqslant K_3|p|^{K_4}\left(\int_{\Omega}u^\f{p}{2}\right)^2.
    \end{equation}
    Indeed, since the choice of each constant $K_2,\dots,K_5$ was dependent only on $\eta,K_1,c_1$, and since $c_1$ depends only on $|\Omega|$,
    Inequality \eqref{eq:alikakosrestated} concludes the claim.
\end{proof}
\end{proposition}
\begin{corollary}\label{alikikosbound}
Let $(\sigma,r)$ denote a classical solution of System \eqref{eq:rsigmaevolution}.
Then, the density of $\sigma$ is uniformly strictly positive and uniformly bounded from above on $[0,\caT)\times \overline{\Omega}$. 
\end{corollary}
\begin{proof}
The variable $\sigma$ defines a solution of System \eqref{eq:nlfp2} with $\bs{v}= \bs{v}_\sigma$.
In particular, the drift satisfies 
\[
\|\bs{v}_\sigma\|_{L^\infty([0,\caT)\times \Omega)} \leqslant \sum_{i=1}^2 \|\partial_x V_i\|_{L^\infty(\Omega)}
\]
and so the theory of Proposition \ref{pedi} applies for $ u = \sigma$ where the constant $c_v$ featured in the statement of Proposition \ref{pedi} may be taken to depend only on the Lipschitz constant of each potential.

\medskip 

Recall from Proposition \ref{pedi} that Inequalities \eqref{eq:pedi} and \eqref{eq:alikakosineq} are satisfied.
The existence of strict upper and lower bounds then follows from considering Inequality \eqref{eq:alikakosineq} satisfied for the sequences 
\[
p_{k,1} = 2^k, \ p_{k,2} = -(2^k)
\]
and then considering the limit as $k \to \infty$ of 
the telescopic Moser--Alikakos type inequality presented in \cite[Lemma 3.2]{KimZhang}.
\end{proof}
Having established a uniform upper bound and strict positivity of the density $\sigma$, it follows that a short time solution of System \eqref{eq:vcdid} can not blow-up by means of $\sigma$ exiting the domain on which $f$ is smooth.
Subsequently, to extend the existence result to an arbitrary time horizon, it follows from \cite[Theorem 1.]{Amannqlsexistence} that it is sufficient to show that each density is bounded in the space $L^\infty([0,T);H^1(\Omega))$, independently of $\caT$.
As such, we recall the spaces.
\[
\mathcal{V}^k(Q_T)\coloneqq L^\infty([0,T];H^{k-1}(\Omega))\cap L^2([0,T];H^k(\Omega)), k \in \N.
\]
\begin{proposition}\label{sobolev}
Let $(\sigma,r)$ denote a classical short-time solution of System \eqref{eq:rsigmaevolution}.
Then, there exists $K_6$ such that 
\[
\|r\|_{\mathcal{V}_2(Q_\caT)},\|\sigma\|_{\mathcal{V}_2(Q_\caT)} < K_6 \text{ for } i \in \{1,2\}.
\]
In particular, or any short time horizon satisfying $\caT \leqslant T$, the value of $K_6$ may be made independent of $\caT$. 
\end{proposition}
\begin{proof}
Since $\sigma_0$ is assumed to be uniformly strictly positive, the aggregate density $\sigma$ is uniformly bounded and strictly positive on $[0,\caT)\times \Omega$. This is as a consequence of Corollary \ref{alikikosbound} and, consequently, $f$ is $C^{4,\beta}$ on the image of $\sigma$.     
Due to the aforementioned upper and lower bounds, the equation governing $\sigma$ also defines a uniformly parabolic non-linear Fokker--Planck equation with the drift $\bs{v}_\sigma \in L^\infty([0,\caT)\times\Omega)$.

\medskip 

Equipped with no-flux boundary conditions, \cite[Theorem 1.6]{Lpparabolictheory} shows that $\sigma \in L^p([0,\caT];W^{1,p}(\Omega))$ for every $p \in [1,+\infty)$ and in particular, this bound can be made independent of $\caT$ up to time $T$.

\medskip 

By formally differentiating the equation governing $r$ in System \eqref{eq:rsigmaevolution} with respect to $x$, it follows that the derivative of the ratio, denoted $\partial_x r \eqqcolon \tilde{r}$, weakly satisfies the following parabolic equation.
    \begin{equation}\label{eq:dxrevolution}
    \begin{cases}
    \begin{alignedat}{2}
          \partial_x \tilde{r} & = \eta\partial_{xx}\tilde{r} + \partial_x(h^\eta\tilde{r} + \bs{w}_r r(1-r)) &&\text{ in }
    (0,\caT)\times \Omega, \\
    \tilde{r} & = 0 &&\text{ on } (0,\caT)\times\partial\Omega,
    \end{alignedat}
    \end{cases}
    \end{equation}
    where the coefficient $h^\eta \co [0,\caT)\times \overline{\Omega}\to \R $ is defined
    \begin{align*}
    h^\eta & \coloneqq
    \partial_xf'(\sigma)+2\eta \partial_x\log(\sigma) + \bs{v}_r.
    \end{align*}
    However, since $\partial_x \sigma \in L^p([0,\caT]\times \Omega)$ and since $\sigma$ is bounded uniformly from above and below, it follows that $h^\eta, \bs{w}_rr(1-r)  \in L^p([0,\caT]\times \Omega)$.
    Consequently, it follows from \cite[Chapter 3, Theorem 4.2.]{Ladyzhenskaya} that $\tilde{r}$ must coincide with the unique energy solution of Equation \eqref{eq:dxrevolution} which is bounded in $\mathcal{V}^1(Q_\caT)$.
    Thus $r$ is bounded  $\mathcal{V}^2(Q_\caT)$ where, in particular, thanks to Inequality \cite[Chapter 3, Lemma 2.1.]{Ladyzhenskaya},
    the norm $\|r\|_{\mathcal{V}^2(Q_\caT)}$ may be taken to depend only on $T$, the initial data and the regularity of $h^\eta, \bs{w}_r \in L^p([0,\mathcal{T}]\times \Omega)$. 

    \medskip 

    Second order estimates are derived for $\sigma$ as follows.
    Let $\tilde{\sigma} = \partial_x\sigma$. 
    Then, $\tilde{\sigma}$ weakly satisfies the linearised parabolic problem
    \begin{equation}\label{eq:dxsigmaevolution}
            \begin{cases}
            \begin{alignedat}{2}
          \partial_t\tilde{\sigma} & = \partial_{x}((\eta + a_\sigma)\partial_x\tilde{\sigma} + b_\sigma\tilde{\sigma}) + \partial_x f^\eta_{\sigma} &&\text{ in } (0,\caT)\times \Omega,\\
          \tilde{\sigma} & = g_{\sigma} &&\text{ on } (0,\caT) \times \partial\Omega.
        \end{alignedat}
    \end{cases}
    \end{equation}
    where the coefficients $a_\sigma, b_\sigma, f_\sigma\co [0,\caT)\times \Omega$ and $g_\sigma\co [0,\caT)\times \partial\Omega$ are defined
    \begin{align*}
        a_\sigma & \coloneqq \sigma f''(\sigma), \quad \quad b_\sigma \coloneqq \partial_x (\sigma f''(\sigma)) + \bs{v}_\sigma, \\
        f_\sigma & \coloneqq \sigma \partial_x\bs{v}_\sigma, \quad \quad \quad g_\sigma \coloneqq -\f{\sigma \bs{v}_\sigma}{\eta+ \sigma f''(\sigma)}. 
    \end{align*}
    The coefficient $\eta + a_\sigma$ is strictly positive and uniformly bounded from above and so Problem \eqref{eq:dxsigmaevolution} is uniformly parabolic. 
    Moreover, $g_\sigma \in L^\infty((0,\caT)\times \partial\Omega)$ whilst $b_\sigma \in L^p([0,\caT)\times \Omega)$ for every $p \in [1,+\infty)$ thanks to the afore-established regularity of $\sigma$. 
    In addition, $f_\sigma \in L^\infty([0,\caT];L^2(\Omega))$ since it was shown that $r$ is bounded in $\mathcal{V}^2(Q_\caT)$.

    \medskip 

    It follows from \cite[Chapter 3, Theorem 4.2.]{Ladyzhenskaya} that Problem \eqref{eq:dxsigmaevolution} admits a unique weak solution in the class $\mathcal{V}^1(Q_\caT)$.
    However, since $\sigma$ was the unique classical solution of System \eqref{eq:rsigmaevolution}, the weak solution of System \eqref{eq:dxsigmaevolution} must necessarily coincide with $\tilde{\sigma}$.
    Thus, we conclude that $\sigma$ is bounded in $\mathcal{V}^2(Q_\caT)$.
\end{proof} 
\begin{thm}\label{classicalft2}
    For any $T > 0$ and initial data $(\rho_{1,0},\rho_{2,0})$, there exists a unique classical solution of System \eqref{eq:vcdid}. 
    In particular, this solution possesses the regularity 
    \[
    C([0,T];C^3(\overline{\Omega})) \cap C^1([0,T];C^1(\overline{\Omega})).
    \]
    \end{thm}
\begin{proof}
Recall from Corollary \ref{alikikosbound} that $\sigma$ is uniformly bounded and strictly positive.
Further, recall from Proposition \ref{sobolev} that $r, \sigma$ are bounded in 
\[
L^\infty([0,T];H^1(\Omega))
\]
independently of $\caT$ for any short-time existence horizon $\caT \leqslant T$.

\medskip 

Thanks to the embedding $H^1(\Omega) \hookrightarrow L^\infty(\Omega)$, the products $r\sigma, (1-r)\sigma$ must similarly be bounded in $L^\infty([0,T];H^1(\Omega))$.
Moreover, thanks to the aforementioned estimates, there exists $S$, a subset of $H^1(\Omega;\R^2)\cap C(\overline{\Omega};G)$ bounded away from its exterior, such that $(\rho_{1,t},\rho_{2,t}) \in S$ for all $ t \in (0,\caT)$ with $\caT \leqslant T$.
It then follows from \cite[Theorem 1.1.]{Amannqlsexistence} that the unique classical solution of System \eqref{eq:vcdid} exists up to time $t= T$.

\medskip 

By means of the classical regularity theory \cite[Theorem 5.3.]{Ladyzhenskaya} and equipped with a globally defined unique classical solution $\rho_1,\rho_2$, it is possible to employ a boot-strapping argument to establish regularity up to $C([0,T];C^3(\overline{\Omega})) \cap C^1([0,T];C^1(\overline{\Omega}))$.
In particular, the derivation of such regularity is where the imposition of the first order compatibility condition assumed in Assumption \ref{K3} and the fourth order H\"older regularity Assumption \ref{K4} are finally imposed.
\end{proof}

\section{Supplementary Results}
\begin{proposition}\label{ratiopropn}
    Let $f,g \in \sP^{ac}(\Omega)$.
    Then, there exists $r \in L^\infty(\Omega)$, which may not be unique, such that 
    \[
    r(f+g) = f
    \]
    almost everywhere on $\Omega$.
\end{proposition}
\begin{proof}
    Let $(f_\e)_{\e> 0}, (g_\e)_{\e> 0}$ denote smooth uniformly positive approximations of $f,g$ converging strongly in $L^1(\Omega)$ to $f,g$ as $\e \to 0$.
    Then, the ratio 
    \[
    r_\e \coloneqq \frac{f_\e}{f_\e + g_\e}
    \]
    is smooth and well-defined everywhere on $\Omega$.
    Moreover, $r_\e(x) \in [0,1]$ for every $x \in \Omega$ and hence, $r_\e$ converges to a limiting density $r$ in the $L^\infty(\Omega)$ weak-$\ast$ topology.
    Moreover, thanks to the strong convergence in $L^1(\Omega)$ of $f_\e$ and $g_\e$ the following equality must be satisfied for any test function $\varphi \in C(\Omega)$.
    \begin{align*}
        \int_{\Omega} r(f+g)\varphi = \lim_{\e \to 0} \int_{\Omega} r_\e(f_\e+g_\e)\varphi =  \lim_{\e \to 0} \int_{\Omega} f_\e \varphi =\int_{\Omega} f \varphi.
    \end{align*}
    In particular, since the test function $\varphi $ was arbitrary, 
    the equality 
    \[
    r(f+g) = f
    \]
    must hold almost everywhere on $\Omega$.
\end{proof}
\begin{proposition}\label{energypropn}
Let $g \in C^1([0,T];C(\overline{\Omega}))$ then the function $G\co [0,T]\to \R$ given by
\[
G(t) \coloneqq \int_{\Omega}|g_t| \diff x
\]
is differentiable for almost every $t \in [0,T]$ and $G'(t)$ satisfies the equality 
\begin{equation}\label{eq:gprime}
G'(t) = \int_{\Omega} \sgn(g_t)\partial_tg_t\diff x
\end{equation}
for almost every $t \in [0,T]$.
\end{proposition}
\begin{proof}
  Since $g\in C^1([0,T];C(\overline{\Omega}))$ it follows that $t\mapsto g(t,x)$ is uniformly Lipschitz continuous in $t$ for $x\in \Omega$.
  Consequently, $G$ is Lipschitz continuous and hence differentiable almost everywhere on $[0,T]$.
  It is left to establish Equality \eqref{eq:gprime}.

  \medskip 

  The map $t\mapsto |g(t,x)|$ is almost everywhere differentiable,  satisfying the formula 
  \begin{equation}\label{eq:weakchain}
      \partial_t|g(t,x)| = \sgn(g(t,x))\partial_t g(t,x)
  \end{equation}
  almost everywhere on $[0,T]\times \Omega$.
  Consequently, for almost every $t$, the set of $x$ for which Equality \eqref{eq:weakchain} does not hold, must be of zero measure in $\Omega$.
  This means that, for almost every fixed $t \in [0,T]$ the Equality \eqref{eq:weakchain} holds.

  \medskip
  Subsequently, in considering the limit of the difference quotient
  \[
  \lim_{h\to0}\frac{G(t+h)-G(t)}{h}
  \]
  for almost every $t \in [0,T]$, the integral may be exchanged with the limit via the dominated convergence theorem.
  Then, in taking the limit under the integral sign, the subsequent derivative may be identified via Equality \eqref{eq:weakchain} which establishes the claim.
\end{proof}

\printbibliography

\end{document}